\documentclass{article}

\usepackage[affil-it]{authblk}
\usepackage{amsthm}
\usepackage{amsmath,amsfonts,amssymb}
\usepackage{stmaryrd}
\usepackage{url}
\usepackage[export]{adjustbox}
\usepackage[ruled,linesnumbered]{algorithm2e}

\newtheorem{thm}{Theorem}[section]
\newtheorem{prop}[thm]{Proposition}

\newtheorem{cor}[thm]{Corollary}
\newtheorem{lem}[thm]{Lemma}

\newtheorem{definition}[thm]{Definition}

\newcommand{\CL}[1]{\mathcal{#1}}
\newcommand{\RM}[1]{\mathrm{#1}}
\newcommand{\BF}[1]{\mathbf{#1}}
\newcommand{\BB}[1]{\mathbb{#1}}

\newcommand{\ten}[1]{\underline{\mathbf{#1}}}
\def\minimize{\mathop{\text{minimize}}}
\def\maximize{\mathop{\text{maximize}}}%

\title{Very Large-Scale Singular Value Decomposition
	Using Tensor Train Networks} 

\author[a]{Namgil Lee \thanks{namgil.lee@riken.jp}}
\author[a]{Andrzej Cichocki \thanks{cia@brain.riken.jp}}
\affil[a]{Laboratory for Advanced Brain Signal Processing, 
	RIKEN Brain Science Institute, Wako-shi, Saitama 3510198, Japan}

\date{}

\begin{document}
\maketitle

\begin{abstract}
We propose new algorithms for 
singular value decomposition (SVD) 
of very large-scale matrices
based on a low-rank tensor approximation technique
called the tensor train (TT) format. 
The proposed algorithms can compute several 
dominant singular values and corresponding singular 
vectors for large-scale structured matrices given
in a TT format. 
The computational complexity 
of the proposed methods scales logarithmically with 
the matrix size under the assumption that both the matrix 
and the singular vectors admit low-rank TT decompositions. 
The proposed methods, which are called the 
alternating least squares for SVD (ALS-SVD) and modified 
alternating least squares for SVD (MALS-SVD), 
compute the left and right singular vectors approximately 
through block TT decompositions. 
The very large-scale optimization problem is reduced to 
sequential small-scale optimization problems, 
and each core tensor of the block TT decompositions 
can be updated by applying any standard optimization methods. 
The optimal ranks of the block TT decompositions are determined adaptively 
during iteration process, so that we can achieve high approximation accuracy. 
Extensive numerical simulations are conducted for several types of 
TT-structured matrices such as Hilbert matrix, Toeplitz matrix,
 random matrix with prescribed singular 
values, and tridiagonal matrix.  
The simulation results demonstrate the effectiveness 
of the proposed methods compared with standard 
SVD algorithms and TT-based algorithms
developed for symmetric eigenvalue decomposition. 

\vspace{1pc}
{\raggedright KEY WORDS: 
curse-of-dimensionality, 
low-rank tensor approximation, 
matrix factorization, 
symmetric 
eigenvalue decomposition, 
singular value decomposition, 
tensor decomposition, 
tensor network, 
matrix product operator, 
Hankel matrix, 
Toeplitz matrix, 
tridiagonal matrix}

\end{abstract}



\section{Introduction}

The singular value decomposition (SVD) is one of the 
most important matrix factorization techniques in 
numerical analysis. The SVD can be used for 
the best low-rank approximation for matrices, 
computation of pseudo-inverses of matrices, 
solution of unconstrained linear least squares problems, 
principal component analysis, cannonical correlation analysis, 
and estimation of ranks and condition numbers of matrices, 
just to name a few. 
It has a wide range of applications in image processing, 
signal processing, immunology, molecular biology, 
information retrieval, systems biology, 
computational finance, and so on \cite{Wall2003}. 

In this paper, we propose two algorithms for computing $K$ 
dominant singular values and corresponding singular vectors 
of structured very large-scale matrices. 
The $K$ dominant singular values/vectors can be computed by 
solving the following trace maximization problem: 
given $\BF{A}\in\BB{R}^{P\times Q}$, 
	\begin{equation}\label{eqn:maximize}
	\begin{split}
	\maximize_{\BF{U}, \BF{V}}
	& \qquad
	\text{trace}\left( \BF{U}^\RM{T}\BF{A}\BF{V} \right)\\
	\text{subject to}
	& \qquad 
	\BF{U}^\RM{T}\BF{U}=
		\BF{V}^\RM{T}\BF{V}=\BF{I}_K. 
	\end{split}
	\end{equation}
This can be derived based on the fact that the SVD of $\BF{A}$ 
is closely related to the eigenvalue decomposition (EVD) 
of the symmetric matrix $\begin{bmatrix}\BF{0}&\BF{A}\\
\BF{A}^\RM{T}&\BF{0}\end{bmatrix}$ \cite[Theorem 3.3]{Demmel97}, and 
the Ky Fan trace min/max principles \cite[Theorem 1]{Fan1949}. 
See Appendix \ref{sec:app1} for more detail. 
Standard algorithms for computing the SVD of a $P\times Q$ matrix 
with $P\geq Q$ cost $\CL{O}(PQ^2)$ for computing 
full SVD \cite[Table 3]{Comon90}, and $\CL{O}(PQK)$ for computing 
$K$ dominant singular values \cite[Section 2.4]{Liu2013}. 
However, in case that $P$ and $Q$ are exponentially growing, e.g., 
$P=Q=I^N$ for some fixed $I$, the computational and storage 
complexities also grow exponentially with $N$. In order to avoid the 
``curse-of-dimensionality", the Monte-Carlo algorithm \cite{Frieze98}
was suggested but its accuracy is not high enough. 



The basic idea behind the proposed algorithms is to reshape 
(or tensorize) matrices and vectors into high-order tensors and 
compress them by applying a low-rank tensor approximation 
technique \cite{Gra2013}. Once the matrices and vectors are 
represented in low-rank tensor formats such as the tensor train (TT) 
\cite{Ose2011,Ose2009} or hierarchical Tucker (HT) 
\cite{Gra2010,Hac2009} decompositions, all the basic numerical 
operations such as the matrix-by-vector multiplication are performed 
based on the low-rank tensor formats with feasible computational 
complexities growing only linearly in $N$ \cite{Gra2010,Ose2011}. 

On the other hand, traditional low-rank tensor approximation 
techniques such as the CANDECOMP/PARAFAC (CP) 
and Tucker decompositions also compress
high-order tensors into low-parametric tensor formats \cite{Kol2009}. 
Although the CP and Tucker decompositions 
have a wide range of applications 
in chemometrics, signal processing, neuroscience, data mining, 
image processing, and numerical analysis \cite{Kol2009}, 
they have their own limitations. 
The Tucker decomposition cannot avoid the 
curse-of-dimensionality, which prohibits its application to
the tensorized large-scale data matrices \cite{Gra2013}. 
The CP decomposition does not suffer from 
the curse-of-dimensionality, but there does not exist a reliable 
and stable algorithm for best low-rank approximation 
due to the lack of closedness of the set of tensors of bounded 
tensor ranks \cite{desilva08}. 

In this paper, we focus on the TT decomposition, 
which is one of the most simplest tensor network formats 
\cite{Espig2011}. The TT and HT decompositions 
can avoid the curse-of-dimensionality by low-rank approximation, 
and possess the closedness property \cite{Espig2011,FalHac2012}. 
For numerical analysis, basic numerical operations such 
as addition and matrix-by-vector multiplication based on 
low-rank TT formats usually lead to \textit{TT-rank} growth, so 
an efficient rank-truncation should be followed. Efficient 
rank-truncation algorithms for the TT and HT decompositions
were developed in \cite{Gra2010,Ose2011}. 

The computation of extremal eigen/singular values and 
the corresponding eigen/singular vectors are
usually obtained by solving an optimization problem 
such as the one in \eqref{eqn:maximize} or 
by maximizing/minimizing the Rayleigh quotient \cite{Demmel97}. 
In order to solve large-scale optimization problems
based on the TT decomposition,
several different types of optimization algorithms have 
been suggested in the literature. 

First, existing iterative methods can be combined with truncation 
of the TT format \cite{HuckleWald2012,Leb2011,Mach2011}. 
For example, for computing several extremal eigenvalues of 
symmetric matrices, conjugate-gradient type iterative 
algorithms are combined with truncation for minimizing 
the (block) Rayleigh quotient in \cite{Leb2011,Mach2011}. 
In the case that a few eigenvectors should be computed 
simultaneously, the block of orthonormal vectors can be 
efficiently represented in {\it block TT format} \cite{Leb2011}. 
However, the whole matrix-by-vector multiplication
causes all the {\it TT-ranks} to grow at the same time, 
which leads to a very high computational cost in the 
subsequent truncation step. 

Second, alternating least squares (ALS) type algorithms reduce 
the given large optimization problem into sequential relatively 
small optimization problems, for which any standard optimization 
algorithm can be applied. The ALS algorithm developed in 
\cite{Holtz2012} is easy to implement and each iteration is 
relatively fast. But the TT-ranks should be predefined in advance 
and cannot be changed during iteration. The modified alternating 
least squares (MALS) algorithm, or equivalently density matrix 
renormalization group (DMRG) method 
\cite{Holtz2012,Kho2010,USch2011} can adaptively determine the 
TT-ranks by merging two core tensors into one bigger core tensor
and separating it by using the truncated SVD. The MALS shows 
a fast convergence in many numerical simulations. 
However, the reduced small optimization problem is solved over 
the merged bigger core tensor, which increases the computational 
and storage costs considerably in some cases. 
Dolgov et al. \cite{Dol2013b} developed an alternative ALS type 
method based on block TT format, where the mode corresponding 
to the number $K$ of orthonormal vectors is allowed to move 
to the next core tensor via the truncated SVD. This procedure can 
determine the TT-ranks adaptively if $K>1$ for the block TT format. 
Dolgov and Savostyanov \cite{DolSav2014} and 
Kressner et al. \cite{KresSteinUsh2013} further developed an 
ALS type method which adds rank-adaptivity to the block 
TT-based ALS method even if $K=1$.


In this paper, we propose the ALS and MALS type algorithms for 
computing $K$ dominant singular values of matrices which are not 
necessarily symmetric. The ALS algorithm based on block TT format 
was originally developed for block Rayleigh quotient minimization 
for symmetric matrices \cite{Dol2013b}. The MALS algorithm was also 
developed for Rayleigh quotient minimization for symmetric matrices 
\cite{Holtz2012,Kho2010,USch2011}. 
We show that the $K$ dominant singular values can be efficiently 
computed by solving the maximization problem \eqref{eqn:maximize}.
We compare the proposed algorithms with other block TT-based 
algorithms which were originally developed for computing eigenvalues 
of symmetric matrices, by simulated experiments and a theoretical 
analysis of computational complexities. 

Moreover, we present extensive numerical experiments 
for various types of structured matrices such as 
Hilbert matrix, Toeplitz matrix, random matrix 
with prescribed singular values, 
and tridiagonal matrix. 
We compare the performances of several different 
SVD algorithms, and we present the relationship between 
TT-ranks and approximation accuracy based on 
the experimental results. 
We show that the proposed block TT-based algorithms 
can achieve very high accuracy by adaptively determining 
the TT-ranks. It is shown that the proposed algorithms 
can solve very large-scale 
optimization problems for matrices of as large sizes as 
$2^{50}\times 2^{50}$ on desktop computers. 

The paper is organized as follows. 
In Section 2, notations for tensor operations and TT formats 
are described. 
In Section 3, the proposed SVD algorithms based on 
block TT format is presented. Their computational complexities 
are analyzed and computational considerations are discussed. 
In Section 4, extensive experimental results are presented
for analysis and comparison of performances 
of SVD algorithms for several types of structured matrices. 
Conclusion and discussions are given in Section 5.

\section{Tensor Train Formats}

\subsection{Notations}

We refer to \cite{Cic2009,Kol2009,Lee2014} for notations for 
tensors and multilinear operations. Scalars, vectors, and matrices
are denoted by lowercase, lowercase bold, and uppercase bold letters 
as $x$, $\BF{x}$, and $\BF{X}$, respectively. 
An $N$th order tensor $\ten{X}$ is a multi-way array of size 
$I_1\times I_2\times \cdots \times I_N$, where $I_n$ is the size 
of the $n$th dimension or mode. A vector is a 1st order tensor and 
a matrix is a 2nd order tensor. The $(i_1,i_2,\ldots,i_N)$th entry of 
$\ten{X}\in\BB{R}^{I_1\times I_2\times\cdots\times I_N}$ is denoted 
by either $x_{i_1,i_2,\ldots,i_N}$ or $\ten{X}(i_1,i_2,\ldots,i_N)$. 
Let $(i_1,i_2,\ldots,i_N)$ denote the multi-index defined by 
	\begin{equation}\label{eqn:multi_index_paran}
	(i_1,i_2,\ldots,i_N) = i_1+(i_2-1)I_1+\cdots+(i_N-1)I_1I_2\cdots I_{N-1}. 
	\end{equation}
The vectorization of a tensor 
$\ten{X}\in\BB{R}^{I_1\times I_2\times\cdots\times I_N}$
is denoted by 
	\begin{equation}
	\text{vec}( \ten{X} )\in\BB{R}^{I_1I_2\cdots I_N}, 
	\end{equation}
and each entry of $\text{vec}(\ten{X})$ is associated with each entry of 
$\ten{X}$ by 
	\begin{equation}
	\left( \text{vec}(\ten{X}) \right)_{
	(i_1,i_2,\ldots,i_N)} = 
	\ten{X}(i_1,i_2,\ldots,i_N) 
	\end{equation}
like in MATLAB. For each $n=1,2,\ldots,N$, the mode-$n$ matricization 
of a tensor $\ten{X}\in\BB{R}^{I_1\times\cdots\times I_N}$ is 
defined by 
	\begin{equation}
	\BF{X}_{(n)} \in \BB{R}^{I_n\times I_1\cdots I_{n-1}I_{n+1}\cdots I_N}
	\end{equation}
with entries 
	\begin{equation}
	\left(\BF{X}_{(n)} \right)_{i_n,(i_1,\ldots,i_{n-1},i_{n+1},\ldots,i_N)}
	= \ten{X}(i_1,i_2,\ldots,i_N). 
	\end{equation}	
Tensorization is the reverse process of the vectorization, 
by which large-scale vectors and matrices are reshaped into higher-order 
tensors. For instance, a vector of length $I_1I_2\cdots I_N$ can be 
reshaped into a tensor of size $I_1\times I_2\times \cdots \times I_N$, 
and a matrix of size $I_1I_2\cdots I_N \times J_1J_2\cdots J_N$ can 
be reshaped into a tensor of size $I_1\times I_2\times \cdots \times 
I_N \times J_1\times J_2\times\cdots\times J_N$. 

The mode-$n$ product of a tensor 
$\ten{A}\in\BB{R}^{I_1\times\cdots\times I_N}$ and 
a matrix $\BF{B}\in\BB{R}^{J\times I_n}$ is defined by 
	\begin{equation}
	\ten{C} = \ten{A}\times_n\BF{B} \in\BB{R}^{I_1\times\cdots\times 
	I_{n-1}\times J\times I_{n+1}\times\cdots\times I_N}
	\end{equation}
with entries 
	\begin{equation}
	c_{i_1,\ldots,i_{n-1},j,i_{n+1},\ldots,i_N}
	= \sum_{i_n=1}^{I_N} 
	a_{i_1,\ldots,i_N} b_{j,i_n}. 
	\end{equation}
The mode-$(M,1)$ contracted product 
of tensors $\ten{A}\in\BB{R}^{I_1\times I_2\times \cdots\times I_M}$
and $\ten{B}\in\BB{R}^{I_M\times J_2\times J_3\times \cdots\times J_N}$
is defined by 
	\begin{equation}
	\ten{C}=\ten{A}\bullet\ten{B}\in\BB{R}^{I_1\times I_2\times\cdots\times 
I_{M-1}\times J_2\times J_3\times\cdots\times J_N}
	\end{equation}
with entries 
	\begin{equation}
	c_{i_1,i_2,\ldots,i_{M-1},j_2,j_3,\ldots,j_N}
	= \sum_{i_M=1}^{I_M}
	a_{i_1,i_2,\ldots,i_M} b_{i_M,j_2,j_3,\ldots,j_N}. 
	\end{equation}
The mode-$(M,1)$ contracted product is a natural generalization 
of the matrix-by-matrix multiplication. 

Tensors and tensor operations are often represented 
as tensor network diagrams for illustrating the underlying principles 
of algorithms and tensor operations \cite{Holtz2012}. 
Figure \ref{Fig:basic_graph} shows examples of 
the tensor network diagrams for tensors and tensor operations. 
In Figure \ref{Fig:basic_graph}(a), a tensor 
is represented by a node with as many edges as its order. 
In Figure \ref{Fig:basic_graph}(b), the 
mode-$(3,1)$ contracted product is represented as the link
between two nodes. Figure \ref{Fig:basic_graph}(c)
represents the tensorization process of a vector 
into a 3rd order tensor. Figure \ref{Fig:basic_graph}(d)
represents a singular value decomposition of 
an $I\times J$ matrix into the product $\BF{U\Sigma V}^\RM{T}$. 
The matrices $\BF{U}$ and $\BF{V}$ of 
orthonormal column vectors are represented by 
half-filled circles, and the diagonal matrix $\BF{\Sigma}$
is represented by a circle with slash. 

\begin{figure}
\centering
\begin{tabular}{cc}
\includegraphics[height=2cm]{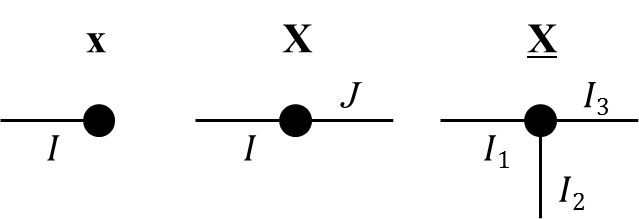} & 
\includegraphics[height=2cm]{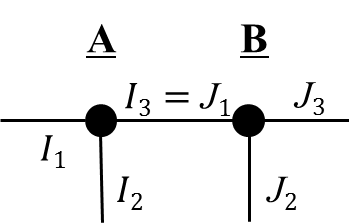} \\
(a) & (b)\\
\includegraphics[height=2cm,valign=t]{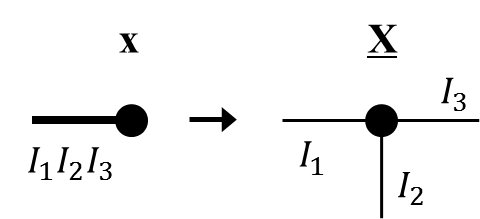} & 
\includegraphics[height=1.3cm,valign=t]{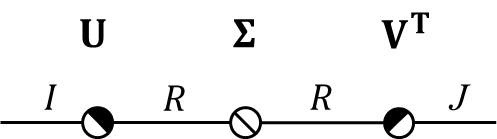} \\
(c) & (d)
\end{tabular}
\caption{\label{Fig:basic_graph}Tensor network diagrams for 
(a) a vector, a matrix, a 3rd order tensor, 
(b) the mode-$(3,1)$ contracted product of two 3rd order tensors, 
(c) the tensorization of a vector, and
(d) singular value decomposition of an $I\times J$ matrix
}
\end{figure}

\subsection{Tensor Train Format}

A tensor $\ten{X}\in\BB{R}^{I_1\times I_2\times\cdots \times I_N}$
is in TT format 
if it is represented by 
	\begin{equation}\label{eqn:TTcontract}
	\ten{X} 
	= 
	\ten{X}^{(1)} \bullet\ten{X}^{(2)} \bullet\cdots\bullet
	\ten{X}^{(N-1)} \bullet \ten{X}^{(N)}, 
	\end{equation}
where 
$\ten{X}^{(n)}\in\BB{R}^{R_{n-1}\times I_n\times R_n},n=1,2,\ldots,N$, 
are 3rd order core tensors which are called as TT-cores, and 
$R_1,R_2,\ldots,R_{N-1}$ are called as TT-ranks. 
It is assumed that $R_0=R_N=1$. 

Figure \ref{Fig:TTformat} shows the tensor network diagram for an 
$N$th order tensor in TT format. Each of the core tensors
is represented as a third order tensor except the first and the last 
TT-cores, which are matrices.  

\begin{figure}
\centering
\includegraphics[height=2.4cm]{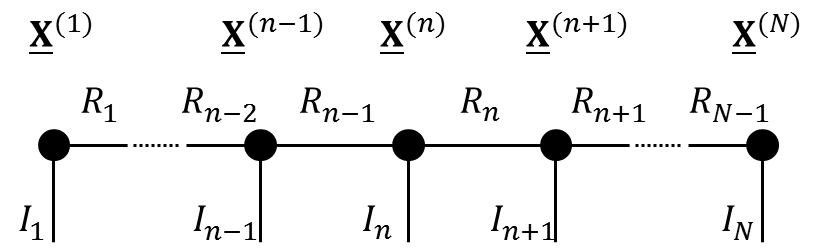}
\caption{\label{Fig:TTformat}Tensor network diagram for 
an $N$th order tensor in TT format
}
\end{figure}

The TT format is often called the matrix product states (MPS)
with open boundary conditions in quantum physics community
because each entry of $\ten{X}$ 
in (\ref{eqn:TTcontract}) can be written by the products of matrices as 
	\begin{equation}
	x_{i_1,i_2,\ldots,i_N} = \BF{X}^{(1)}_{i_1}\BF{X}^{(2)}_{i_2}
	\cdots \BF{X}^{(N)}_{i_N},
	\end{equation}
where $\BF{X}^{(n)}_{i_n} = \ten{X}^{(n)}(:,i_n,:) 
\in\BB{R}^{R_{n-1}\times R_n}$ are the slice matrices of $\ten{X}^{(n)}$. 
Note that $\BF{X}^{(1)}_{i_1}\in\BB{R}^{1\times R_1}$ 
and $\BF{X}^{(N)}_{i_N}\in\BB{R}^{R_N\times 1}$  are row and column 
vectors, respectively. Each entry of $\ten{X}$ can also be written by sums of 
scalar products 
	\begin{equation}
	x_{i_1,i_2,\ldots,i_N} = 	\sum_{r_1=1}^{R_1}\sum_{r_2=1}^{R_2}
	\cdots\sum_{r_{N-1}=1}^{R_{N-1}}
	x^{(1)}_{1,i_1,r_1}x^{(2)}_{r_1,i_2,r_2}\cdots x^{(N)}_{r_{N-1},i_N,1}, 
	\end{equation}
where $x^{(n)}_{r_{n-1},i_n,r_n}=\ten{X}^{(n)}(r_{n-1},i_n,r_n)
\in\BB{R}$ is the entry of the $n$th TT-core $\ten{X}^{(n)}$. 
The tensor $\ten{X}$ in TT format can also be written by 
the outer products of the fibers (or column vectors) as 
	\begin{equation}\label{eqn:TTouter}
	\ten{X} = 
	\sum_{r_1=1}^{R_1}\sum_{r_2=1}^{R_2}\cdots
	\sum_{r_{N-1}=1}^{R_{N-1}}
	\BF{x}^{(1)}_{1,r_1}\circ \BF{x}^{(2)}_{r_1,r_2}\circ\cdots
	\circ\BF{x}^{(N)}_{r_{N-1},1},
	\end{equation}
where $\circ$ is the outer product and $\BF{x}^{(n)}_{r_{n-1},r_n}
= \ten{X}^{(n)}(r_{n-1},:,r_n)\in\BB{R}^{I_n}$ is the mode-$2$ fiber 
of the $n$th TT-core $\ten{X}^{(n)}$. 

The storage cost for a TT format is $\CL{O}(NIR^2)$, where 
$I=\max(I_n)$ and $R=\max(R_n)$, that is linear with the order $N$. 
Any tensor can be represented exactly or approximately in 
TT format by using the TT-SVD algorithm in \cite{Ose2011}. 
Moreover, basic numerical operations such as the matrix-by-vector 
multiplication can be performed in time linear with $N$ under the 
assumption that the TT-ranks are bounded \cite{Ose2011}. 


\subsection{Tensor Train Formats for Vectors and Matrices}

Any large-scale vector or matrix can also be represented 
in TT format. We suppose that a vector 
$\BF{x}\in\BB{R}^{I_1I_2\cdots I_N}$
is tensorized into a tensor 
$\ten{X}\in\BB{R}^{I_1\times I_2\times \cdots \times I_N}$
and consider the TT format (\ref{eqn:TTcontract}) as the 
TT representation of $\BF{x}=\text{vec}(\ten{X})$. 

Similarly, a matrix $\BF{A}\in\BB{R}^{I_1I_2\cdots I_N\times 
J_1J_2\cdots J_N}$ is considered to be tensorized and permuted into 
a tensor $\ten{A}\in\BB{R}^{I_1\times J_1\times I_2\times J_2\times
\cdots\times I_N\times J_N}$. Then, as in \eqref{eqn:TTcontract}, 
the tensor $\ten{A}$ is represented in TT format as contracted products 
of TT-cores
	\begin{equation}\label{eqn:matrixTTc}
	\ten{A} 
	= 
	\ten{A}^{(1)} \bullet\ten{A}^{(2)} \bullet\cdots\bullet
	\ten{A}^{(N)}, 
	\end{equation}
where 
$\ten{A}^{(n)}\in\BB{R}^{R^A_{n-1}\times I_n\times J_n\times R^A_n},n=1,2,\ldots,N$, 
are 4th order TT-cores with TT-ranks $R^A_1,R^A_2,\ldots,R^A_{N-1}$. 
We suppose that $R^A_0=R^A_N=1$.
The entries of $\ten{A}$ can also be represented in TT format 
by the products of slice matrices
	\begin{equation}\label{eqn:matprod_matrixTT}
	a_{i_1,j_1,i_2,j_2,\ldots,i_N,j_N} = \BF{A}^{(1)}_{i_1,j_1}\BF{A}^{(2)}_{i_2,j_2}\cdots 
	\BF{A}^{(N)}_{i_N,j_N}, 
	\end{equation}
where $\BF{A}^{(n)}_{i_n,j_n} = \ten{A}^{(n)}(:,i_n,j_n,:)
\in\BB{R}^{R^A_{n-1}\times R^A_n}$ is the slice of the $n$th TT-core 
$\ten{A}^{(n)}$. 

In this paper, we call the TT formats \eqref{eqn:TTcontract} and 
(\ref{eqn:matrixTTc}) as the vector TT and matrix TT formats, 
respectively. Note that if the indices $i_n$ and $j_n$ are joined 
as $k_n = (i_n, j_n)$ in \eqref{eqn:matprod_matrixTT}, 
then the matrix TT format is reduced to the 
vector TT format. Figure \ref{Fig:matrixTTformat} shows a tensor network 
representing a matrix of size $I_1I_2\cdots I_N\times J_1J_2\cdots J_N$ 
in matrix TT format. Each of the TT-cores is represented as a 
4th order tensor except the first and the last core tensor. 

\begin{figure}
\centering
\includegraphics[height=2.7cm]{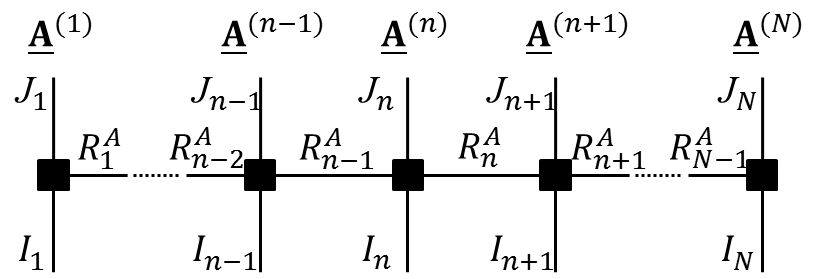}
\caption{\label{Fig:matrixTTformat}Tensor network diagram for 
a matrix of size $I_1I_2\cdots I_N\times J_1J_2\cdots J_N$ 
in matrix TT format
}
\end{figure}

\subsection{Block TT Format}

A group of several vectors can be represented in block TT format 
as follows. Let $\BF{U}= \left[\BF{u}_1,\BF{u}_2,\ldots,\BF{u}_K\right] 
\in\BB{R}^{I_1I_2\cdots I_N\times K}$ denote a matrix with 
$K$ column vectors. Suppose that the matrix $\BF{U}$ is tensorized 
and permuted into a tensor 
$\ten{U}\in\BB{R}^{I_1\times I_2\times\cdots\times I_{n-1}
\times K\times I_n\times \cdots\times I_N}$, where the mode 
of the size $K$ is located between the modes of the sizes 
$I_{n-1}$ and $I_n$ for a fixed $n$. In block TT format, 
the tensor $\ten{U}$ is represented as contracted products 
of TT-cores 
	\begin{equation}\label{eqn:blockTT}
	\ten{U} = 
	\ten{U}^{(1)} \bullet\ten{U}^{(2)} \bullet\cdots\bullet
	\ten{U}^{(N)},
	\end{equation}
where the $n$th TT-core $\ten{U}^{(n)}\in\BB{R}^{R^U_{n-1}\times K\times I_n 
\times R^U_n}$ is a 4th order tensor and the other TT-cores 
$\ten{U}^{(m)}\in\BB{R}^{R^U_{m-1}\times I_m\times R^U_m},m\neq n,$
are 3rd order tensors. 
We suppose that $R^U_0=R^U_N=1$. 
Each entry of $\ten{U}$ can be expressed by the products of slice matrices 
	\begin{equation}\label{eqn:matprod_blockTT}
	u_{i_1,i_2,\ldots,i_{n-1},k,i_n,\ldots,i_N}
	= \BF{U}^{(1)}_{i_1}\BF{U}^{(2)}_{i_2}\cdots \BF{U}^{(n-1)}_{i_{n-1}}
	\BF{U}^{(n)}_{k,i_n}\BF{U}^{(n+1)}_{i_{n+1}}\cdots 
	\BF{U}^{(N)}_{i_N}, 
	\end{equation}
where $\BF{U}^{(m)}_{i_m}\in\BB{R}^{R^U_{m-1}\times R^U_m},m\neq n,$ 
and $\BF{U}^{(n)}_{k,i_n}\in\BB{R}^{R^U_{n-1}\times R^U_n}$ are the 
slice matrices of the $m$th and $n$th TT-cores. We note that for a 
fixed $k\in\{1,2,\ldots,K\}$, the subtensor $\ten{U}^{(n)}(:,k,:,:)
\in\BB{R}^{R_{n-1}\times I_n\times R_n}$ of the $n$th TT-core 
is of order 3. Hence, the $k$th column vector $\BF{u}_k$ is in the 
vector TT format with TT-ranks bounded by $(R^U_1,\ldots,R^U_{N-1})$.


In this paper, we call the block TT format 
\eqref{eqn:blockTT} as the 
block-$n$ TT format, termed by \cite{KresSteinUsh2013},
in order to distinguish 
between different permutations of modes. 
Figure \ref{Fig:blockTTformat} shows a tensor network 
representing a matrix $\BF{U}\in\BB{R}^{I_1I_2\cdots I_N\times K}$
in block-$n$ TT format. We can see that the mode of the size $K$ is 
located at the $n$th TT-core. We remark that the position of the mode 
of the size $K$ is not fixed. 

\begin{figure}
\centering
\includegraphics[height=2.7cm]{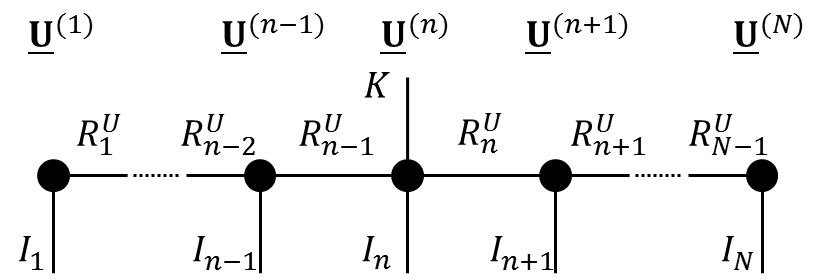}
\caption{\label{Fig:blockTTformat}Tensor network diagram for 
a group of vectors in block-$n$ TT format
}
\end{figure}

In order to clearly state 
the relationship between the TT-cores of the vectors 
$\BF{u}_1,\ldots,\BF{u}_K$ in vector TT format and 
the matrix $\BF{U}$ in block-$n$ TT format, we first define 
the full-rank condition for block-$n$ TT decompositions, 
similarly as in \cite{Holtz2011}. 

\begin{definition}[full-rank condition, \cite{Holtz2011}]
For an arbitrary tensor $\ten{U}\in\BB{R}^{I_1\times\cdots\times 
I_{n-1} \times K\times I_n\times\cdots\times I_N}$, 
a block-$n$ TT decomposition 
	\begin{equation}
	\ten{U} = \ten{U}^{(1)}\bullet\ten{U}^{(2)}\bullet
	\cdots\bullet\ten{U}^{(N)}
	\end{equation}
of TT-ranks $R^U_1,\ldots,R^U_{N-1}$ 
is called minimal or fulfilling the full-rank condition if 
all the TT-cores have full left and right ranks, i.e., 
	\begin{equation}
	R_{m-1} = rank\left(\BF{U}^{(m)}_{(1)}\right), 
	\quad 
	R_{m} = rank\left(\BF{U}^{(m)}_{(3)}\right), 
	\quad
	\text{for }m=1,\ldots,n-1,n+1,\ldots,N, 
	\end{equation}
and 
	\begin{equation}
	R_{n-1} = rank\left(\BF{U}^{(n)}_{(1)}\right), 
	\quad 
	R_{n} = rank\left(\BF{U}^{(n)}_{(4)}\right). 
	\end{equation}
\end{definition}

In principle, any collection of vectors 
$\BF{u}_1,\ldots,\BF{u}_K\in\BB{R}^{I_1I_2\cdots I_{N}}$ can 
be combined and a block-$n$ TT format for 
$\BF{U}=[\BF{u}_1,\ldots,\BF{u}_K]$ can be computed by the TT-SVD 
algorithm proposed in \cite{Ose2011}. 
As described in \cite{Holtz2011}, a minimal block-$n$ TT decomposition for 
$\BF{U}$ can be computed by applying SVD without truncation 
successively. Moreover, the TT-ranks for $\BF{U}$ are determined 
uniquely, which will be called minimal TT-ranks. 

Given a block-$n$ TT decomposition 
\eqref{eqn:blockTT} of the tensor 
$\ten{U}\in\BB{R}^{I_1\times \cdots 
\times I_{n-1}\times K\times I_n\times\cdots \times I_N}$,
we define the contracted products of the left or right TT-cores by 
	\begin{equation}
	\ten{U}^{<m} = \ten{U}^{(1)}\bullet\cdots\bullet\ten{U}^{(m-1)}
	\in\BB{R}^{I_1\times \cdots \times I_{m-1} \times R_{m-1}} 
	\end{equation}
for $m=1,2,\ldots,n$, and 
	\begin{equation}
	\ten{U}^{>m} = \ten{U}^{(m+1)}\bullet\cdots\bullet\ten{U}^{(N)}
	\in\BB{R}^{R_m\times I_{m+1}\times \cdots \times I_N}
	\end{equation}
for $m=n,n+1,\ldots,N$. We define that $\ten{U}^{<1}=\ten{U}^{>N}=1$. 
The tensors $\ten{U}^{\leq m}$ and $\ten{U}^{\geq m}$ are defined in the 
same manner. The mode-$m$ matricization of $\ten{U}^{<m}$ and the mode-$1$ 
matricization of $\ten{U}^{>m}$ are written by
	\begin{equation}
	\begin{split}
	\BF{U}^{<m}_{(m)} & \in\BB{R}^{R_{m-1}\times I_1I_2\cdots I_{m-1}}, \\
	\BF{U}^{>m}_{(1)} & \in\BB{R}^{R_{m}\times I_{m+1}I_{m+2}\cdots I_N}.
	\end{split}
	\end{equation}

We can easily derive the conditions on the vectors 
$\BF{u}_1,\ldots,\BF{u}_K$ based on 
a minimal block-$n$ TT decomposition \eqref{eqn:blockTT} 
of the matrix $\BF{U}$ of TT-ranks $R_1^U,\ldots,R_{N-1}^U$
as follows. 

\begin{prop}
\label{prop:blockTTrank}

Suppose that the matrix $\BF{U}=[\BF{u}_1,\ldots,\BF{u}_K]\in
\BB{R}^{I_1I_2\cdots I_N\times K}$ has the minimal block-$n$
TT decomposition \eqref{eqn:blockTT} of TT-ranks 
$R_1^U,\ldots,R_{N-1}^U$. Let 
	\begin{equation}
	\BF{U}_{k,m} \in\BB{R}^{I_1I_2\cdots I_m \times I_{m+1}\cdots I_N}
	\end{equation}
denote the matrix obtained by reshaping the vector 
$\BF{u}_k \in\BB{R}^{I_1I_2\cdots I_N}$, i.e, 
$\BF{u}_k = \text{vec}(\BF{U}_{k,m})$, for $m=1,2,\ldots,N$.  
Then, we can show that 
	\begin{equation}
	\text{span}\left( \BF{U}_{k,m} \right)
	\subset 
	\text{span}\left( \left(\BF{U}^{< m+1}_{(m+1)}\right)^\RM{T} \right),
	\quad m=1,2,\ldots,n-1, 
	\end{equation}
and 
	\begin{equation}
	\text{span}\left( \BF{U}_{k,m}^\RM{T} \right)
	\subset 
	\text{span}\left( \left(\BF{U}^{> m}_{(1)}\right)^\RM{T} \right),
	\quad m=n,n+1,\ldots,N, 
	\end{equation}
where $\text{span}(\BF{A})$ is the column space of a matrix $\BF{A}$. 
Consequently, the minimal TT-ranks, $R^U_{k,1},\ldots,R^U_{k,N-1}$, 
of the vector $\BF{u}_k$ are bounded by the minimal TT-ranks, 
$R^U_1,\ldots,R^U_{N-1}$, of $\BF{U}$, i.e., 
	\begin{equation}
	R^U_{k,m} \leq R^U_m, 
	\quad m=1,\ldots,N-1, \ k=1,\ldots,K. 
	\end{equation}
\end{prop}
\begin{proof}
Since the vector $\BF{u}_k$ is represented in vector TT format 
as in \eqref{eqn:matprod_blockTT}, 
the matrices $\BF{U}_{k,m}$ are represented by 
	\begin{equation}
	\begin{split}
	\BF{U}_{k,m} &= \left( \BF{U}^{<m+1}_{(m+1)} \right)^\RM{T} 
		\BF{U}^{>m}_{k,(1)}, 
	\quad m=1,2,\ldots,n-1, \\
	\BF{U}_{k,m} &= \left( \BF{U}^{<m+1}_{k,(m+1)} \right)^\RM{T} 
		\BF{U}^{>m}_{(1)}, 
	\quad m=n,n+1,\ldots,N, 
	\end{split} 
	\end{equation}
where $\BF{U}^{>m}_{k,(1)}$ and $\BF{U}^{<m+1}_{k,(m+1)}$
are the matricizations of subtensors of $\ten{U}^{>m}$ and $\ten{U}^{<m+1}$. 
It holds that the minimal TT-ranks $R^U_{k,m}=rank(\BF{U}_{k,m})$ by 
\cite{Holtz2011}. This proves the Proposition \ref{prop:blockTTrank}. 
\end{proof}

\subsection{Matricization of Block TT Format}

A matrix $\BF{U}\in\BB{R}^{I_1I_2\cdots I_N\times K}$ 
having a block-$n$ TT decomposition \eqref{eqn:blockTT} can be expressed 
as a product of matrices, which is useful for describing algorithms 
based on block TT formats. For a fixed $n$, frame matrices
are defined as follows. 

\begin{definition}[Frame matrix, \cite{Dol2013b,Holtz2012,KresSteinUsh2013}]
The frame matrices
$\BF{U}^{\neq n}\in\BB{R}^{I_1I_2\cdots I_N\times R_{n-1}I_nR_n}$
and $\BF{U}^{\neq n-1,n}\in\BB{R}^{I_1I_2\cdots I_N\times R_{n-2}I_{n-1}I_nR_n}$
 are defined by 
	\begin{equation}
	\BF{U}^{\neq n} = \left( \BF{U}^{>n}_{(1)} \right)^\RM{T} \otimes 
					\BF{I}_{I_n} \otimes
					\left( \BF{U}^{<n}_{(n)} \right)^\RM{T}
	\end{equation}
and 
	\begin{equation}
	\BF{U}^{\neq n-1,n} = \left( \BF{U}^{>n}_{(1)} \right)^\RM{T} \otimes 
					\BF{I}_{I_{n}} \otimes
					\BF{I}_{I_{n-1}} \otimes
					\left( \BF{U}^{<n-1}_{(n-1)} \right)^\RM{T}.
	\end{equation}
\end{definition}

The block-$n$ TT tensor $\ten{U}$ is written by $\ten{U}  = \ten{U}^{<n}
\bullet \ten{U}^{(n)} \bullet \ten{U}^{>n},$
where the $n$th TT-core is a 4th order tensor, 
$\ten{U}^{(n)}\in\BB{R}^{R_{n-1}\times K\times I_n\times R_n}$. 
The matrix $\BF{U}$ is the transpose of the mode-$n$ matricization 
$\BF{U}_{(n)}\in\BB{R}^{K\times I_1I_2\cdots I_N}$ of $\ten{U}$, which 
can be expressed by
	\begin{equation}\label{eqn:linearize01}
	\begin{split}
	\BF{U}_{(n)} &
	= \left( \ten{U}^{<n}\bullet \ten{U}^{(n)} \bullet \ten{U}^{>n}  \right)_{(n)}
	\\
	&= \left( \ten{U}^{(n)}
			 \times_1 \left(\BF{U}^{<n}_{(n)}\right)^\RM{T}
			 \times_4 \left(\BF{U}^{>n}_{(1)}\right)^\RM{T}
	    \right)_{(2)}
	\\
	&= \BF{U}^{(n)}_{(2)}
	\left(\BF{U}^{>n}_{(1)}\otimes\BF{I}_{I_n}
	\otimes \BF{U}^{<n}_{(n)} \right).
	\end{split}
	\end{equation}
Next, we consider the contraction of two neighboring core tensors as
	\begin{equation}
	\ten{U}^{(n-1,n)} = \ten{U}^{(n-1)}\bullet\ten{U}^{(n)}
	\in\BB{R}^{R_{n-2}\times I_{n-1}\times K\times I_n\times R_n}.
	\end{equation}
Then the block-$n$ TT tensor $\ten{U}$ 
is written by 
	$\ten{U}  = \ten{U}^{<n-1}\bullet \ten{U}^{(n-1,n)} \bullet \ten{U}^{>n},$ and we can get an another expression for
the mode-$n$ matricization as 
	\begin{equation}\label{eqn:linearize02}
	\BF{U}_{(n)} = \BF{U}^{(n-1,n)}_{(3)}
	\left(\BF{U}^{>n}_{(1)}\otimes\BF{I}_{I_{n}}
	\otimes\BF{I}_{I_{n-1}}
	\otimes \BF{U}^{<n-1}_{(n-1)} \right)
	\in\BB{R}^{K\times I_1I_2\cdots I_N}. 
	\end{equation}
From \eqref{eqn:linearize01} and \eqref{eqn:linearize02}, 
the matrix $\BF{U}=[\BF{u}_1,\BF{u}_2,\ldots,\BF{u}_K]
\in\BB{R}^{I_1I_2\cdots I_N\times K}$ in block-$n$ TT format 
can be written by 
	\begin{equation}\label{eqn:Umatrix_als}
	\BF{U} = \BF{U}^{\neq n}\BF{U}^{(n)}, 
	\qquad n=1,2,\ldots,N,
	\end{equation}
where $\BF{U}^{(n)}=( \BF{U}^{(n)}_{(2)} )^\RM{T}
\in\BB{R}^{R_{n-1}I_nR_n\times K}$, and by 
	\begin{equation}\label{eqn:Umatrix_dmrg}
	\BF{U} = \BF{U}^{\neq n-1,n}\BF{U}^{(n-1,n)}, 
	\qquad n=2,3,\ldots,N,
	\end{equation}
where $\BF{U}^{(n-1,n)} = ( \BF{U}^{(n-1,n)}_{(3)} )^\RM{T}
\in\BB{R}^{R_{n-2}I_{n-1}I_{n}R_{n}\times K}$.

\subsection{Orthogonalization of Core Tensors}

\begin{definition}[Left- and right-orthogonality, \cite{Holtz2011}]
A 3rd order core tensor 
$\ten{U}^{(m)}\in\BB{R}^{R_{m-1}\times I_m\times R_m}$
is called left-orthogonal if 
	\begin{equation}
	\BF{U}_{(3)}^{(m)}\left(\BF{U}_{(3)}^{(m)}\right)^\RM{T} = \BF{I}_{R_m}, 
	\end{equation}
and right-orthogonal if 
	\begin{equation}
	\BF{U}_{(1)}^{(m)}\left(\BF{U}_{(1)}^{(m)}\right)^\RM{T} = \BF{I}_{R_{m-1}}. 
	\end{equation}
\end{definition}

We can show that the matricizations $\BF{U}^{<n}_{(n)}$ and 
$\BF{U}^{>n}_{(1)}$ have orthonormal rows if the left core tensors 
$\ten{U}^{(1)},\ldots,\ten{U}^{(n-1)}$ are left-orthogonalized and 
the right core tensors $\ten{U}^{(n+1)},\ldots,\ten{U}^{(N)}$ are 
right-orthogonalized \cite{Lee2014}. Consequently, the frame matrices 
$\BF{U}^{\neq n}$ and $\BF{U}^{\neq n-1,n}$ have orthonormal 
columns if each of the left and right core tensors is properly 
orthogonalized. From the expressions \eqref{eqn:Umatrix_als} and 
\eqref{eqn:Umatrix_dmrg}, 
we can guarantee orthonormality of the columns of $\BF{U}$
by orthogonalizing the TT-cores. 
Figure \ref{Fig:blockorthTTformat} shows a tensor network diagram  
for the matrix $\BF{U}$ in block-$n$ TT format where 
all the core tensors are either left or right orthogonalized 
except the $n$th core tensor. In this case we can guarantee 
that the frame matrices $\BF{U}^{\neq n}$ and $\BF{U}^{\neq n-1,n}$
have orthonormal columns. 

\begin{figure}
\centering
\includegraphics[height=2.7cm]{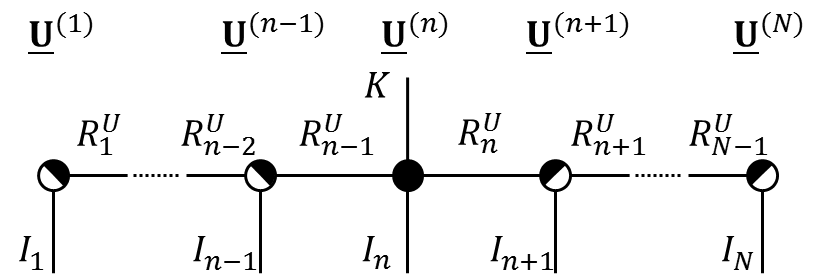}
\caption{\label{Fig:blockorthTTformat}Tensor network diagram for 
a group of vectors in block-$n$ TT format whose 
TT-cores are either left- or right-orthogonalized except the 
$n$th TT-core. The left- or right-orthogonalized tensors are 
represented by half-filled circles. }
\end{figure}

\section{SVD Algorithms Based on Block TT Format}

In this section, we describe the new SVD algorithms, 
which we call the ALS for SVD (ALS-SVD) 
and MALS for SVD (MALS-SVD). 

In the ALS-SVD and MALS-SVD, 
the left and right singular vectors 
$\BF{U}=[\BF{u}_1,\ldots,\BF{u}_N]\in\BB{R}^{I_1I_2\cdots I_N\times K}$ and 
$\BF{V}=[\BF{v}_1,\ldots,\BF{v}_N]\in\BB{R}^{J_1J_2\cdots J_N\times K}$ are initialized 
with block-$N$ TT formats. 
For each $n\in\{1,2,\ldots,N\}$, we suppose that 
$\BF{U}$ and $\BF{V}$ are represented by 
block-$n$ TT formats
	\begin{equation}
	\ten{U} = \ten{U}^{(1)}\bullet\ten{U}^{(2)}\bullet\cdots\bullet
			\ten{U}^{(N)}, \qquad
	\ten{V} = \ten{V}^{(1)}\bullet\ten{V}^{(2)}\bullet\cdots\bullet
			\ten{V}^{(N)},
	\end{equation}
where the $n$th core tensors 
$\ten{U}^{(n)}\in\BB{R}^{R^U_{n-1}\times K\times I_n\times R^U_n}$
and $\ten{V}^{(n)}\in\BB{R}^{R^V_{n-1}\times K\times J_n\times R^V_n}$
are 4th order tensors and the other core tensors are 3rd order tensors. 
We suppose that all the $1,2,\ldots,(n-1)$th core tensors are 
left-orthogonalized and all the $n+1,n+2,\ldots,N$th core tensors are 
right-orthogonalized.

Figure \ref{fig:figUAV} illustrates the tensor network diagrams representing 
the $\text{trace}\left(\BF{U}^\RM{T}\BF{AV}\right)$ in 
the maximization problem \eqref{eqn:maximize}.
Note that the matrix $\BF{A}\in\BB{R}^{I_1I_2\cdots I_N\times 
J_1J_2\cdots J_N}$ is in matrix TT format. 
In the algorithms we don't need to compute the large-scale 
matrix-by-vector products $\BF{AV}$ or $\BF{A}^\RM{T}\BF{U}$. 
All the necessary basic computations are performed based on 
efficient contractions of core tensors.

\begin{figure}
\centering
\begin{tabular}{cc}
\includegraphics[height=6.2cm]{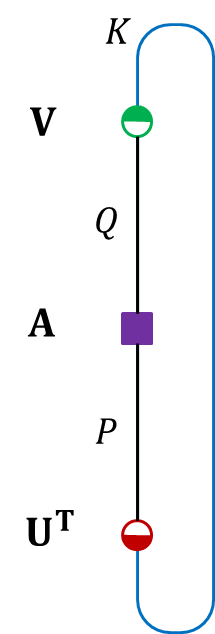}
& 
\includegraphics[height=6.2cm]{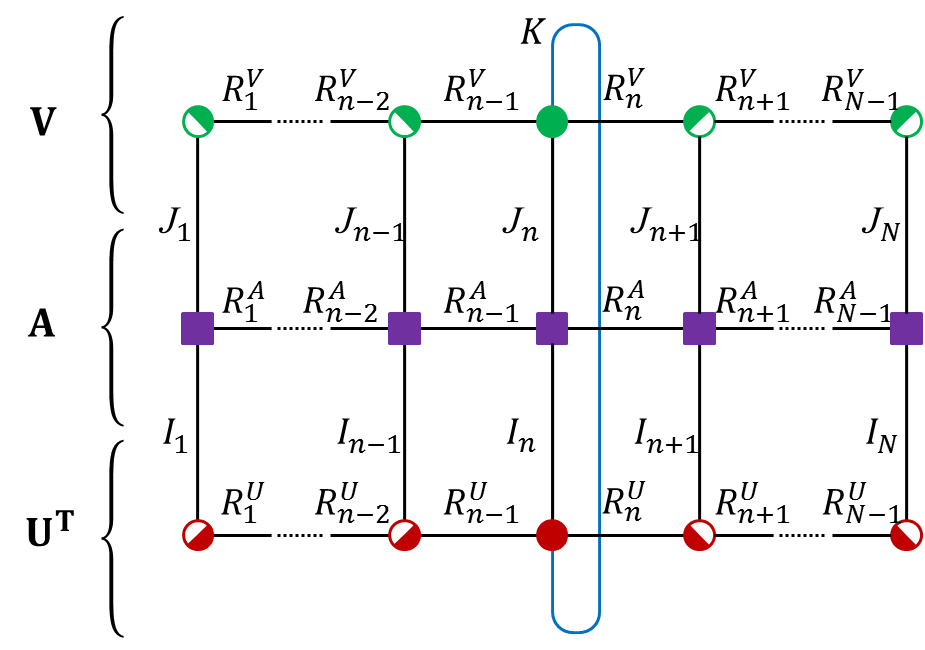}
\\
(a) 
& 
(b) 
\end{tabular}
\caption{\label{fig:figUAV}Tensor network diagrams for the 
$\text{trace}\left(\BF{U}^\RM{T}\BF{AV}\right)$ in the maximization 
problem \eqref{eqn:maximize}. (a) $\BF{U}\in\BB{R}^{P\times K}$ 
and $\BF{V}\in\BB{R}^{Q\times K}$ are matrices with orthonormal column 
vectors and $\BF{A}\in\BB{R}^{P\times Q}$ is a matrix. 
(b) The matrices $\BF{U}$ and $\BF{V}$ are represented in block-$n$ TT 
format and the matrix $\BF{A}$ is represented in matrix TT format, 
where $P=I_1I_2\cdots I_N$ and $Q=J_1J_2\cdots J_N$. Each of the 
TT-cores of $\BF{U}$ and $\BF{V}$ are orthogonalized in order to 
keep the orthogonality constraint.}
\end{figure}

\subsection{ALS for SVD Based on Block TT Format}

The ALS algorithm for SVD based on block TT format 
is described in Algorithm \ref{Alg:ALS}. Recall that, from 
\eqref{eqn:Umatrix_als}, the matrices $\BF{U}$ and $\BF{V}$ of 
singular vectors are written by 
	\begin{equation}
	\BF{U} = \BF{U}^{\neq n}\BF{U}^{(n)}, 
	\qquad
	\BF{V} = \BF{V}^{\neq n}\BF{V}^{(n)}, 
	\end{equation} 
where $\BF{U}^{(n)}=(\BF{U}^{(n)}_{(2)})^\RM{T}
\in\BB{R}^{R^U_{n-1}I_nR^U_n\times K}$ and 
$\BF{V}^{(n)}=(\BF{V}^{(n)}_{(2)})^\RM{T}
\in\BB{R}^{R^V_{n-1}J_nR^V_n\times K}$. 
Note that $(\BF{U}^{\neq n})^\RM{T}\BF{U}^{\neq n}=
\BF{I}_{R^U_{n-1}I_nR^U_n}$ and 
$(\BF{V}^{\neq n})^\RM{T}\BF{V}^{\neq n} = \BF{I}_{R^V_{n-1}I_nR^V_n}.$
Given that all the core tensors are fixed except 
the $n$th core tensors, the maximization problem 
\eqref{eqn:maximize} is reduced to the 
smaller optimization problem as
	\begin{equation}\label{max_als}
	\begin{split}
	\maximize_{\BF{U}^{(n)}, \BF{V}^{(n)}}
	& \qquad
	\text{trace}\left(\BF{U}^\RM{T}\BF{AV}\right)
	= \text{trace}\left( (\BF{U}^{(n)} )^\RM{T}\overline{\BF{A}}_n\BF{V}^{(n)}\right)\\
	\text{subject to}
	& \qquad 
	(\BF{U}^{(n)})^\RM{T}\BF{U}^{(n)}=
	(\BF{V}^{(n)})^\RM{T}\BF{V}^{(n)}=\BF{I}_K, 
	\end{split}
	\end{equation}
where the matrix $\overline{\BF{A}}_n$ is defined by
	\begin{equation}
	\overline{\BF{A}}_n = (\BF{U}^{\neq n})^\RM{T}\BF{A}\BF{V}^{\neq n}
	\in\BB{R}^{R^U_{n-1}I_nR^U_n\times R^V_{n-1}J_nR^V_n}. 
	\end{equation}
We call $\overline{\BF{A}}_n$ as the \textit{projected} matrix. 
In the case that the TT-ranks 
$\{R^U_n\}$ and $\{R^V_n\}$ are small enough, the projected matrix 
$\overline{\BF{A}}_n$ has much smaller sizes than 
$\BF{A}\in\BB{R}^{I_1I_2\cdots I_N\times J_1J_2\cdots J_N}$, 
so that any standard and efficient SVD algorithms can be applied. 
In Section \ref{sec:3_3}, we will describe how the reduced local optimization 
problem \eqref{max_als} can be efficiently solved. 
In practice, the matrix $\overline{\BF{A}}_n$ don't need to be computed
explicitly. Instead, the local matrix-by-vector multiplications 
$\overline{\BF{A}}_n^\RM{T}\BF{u}$ and $\overline{\BF{A}}_n\BF{v}$
for any vectors $\BF{u}\in\BB{R}^{R^U_{n-1}I_nR^U_n}$ and 
$\BF{v}\in\BB{R}^{R^V_{n-1}J_nR^V_n}$ are computed more efficiently
based on the contractions of core tensors of $\BF{U}$, $\BF{A}$, and $\BF{V}$.

\begin{algorithm}
\caption{\label{Alg:ALS}ALS for SVD based on block TT format}
  \SetKwInOut{Input}{Input}
  \SetKwInOut{Output}{Output}

  \Input{$\BF{A}\in\BB{R}^{I_1I_2\cdots I_N\times J_1J_2\cdots J_N}$ 
	in matrix TT format, $K\geq 2$, $\delta\geq 0$ (truncation parameter)}
  \Output{Dominant singular values $\BF{\Sigma}=\text{diag}(\sigma_1,\sigma_2,\ldots,\sigma_K)$
		and corresponding 
		singular vectors $\BF{U} \in\BB{R}^{I_1I_2\cdots I_N\times K}$ and 
		 $\BF{V} \in\BB{R}^{J_1J_2\cdots J_N\times K}$ in block-$N$ TT format 
		with TT-ranks $R^U_1,R^U_2,\ldots,R^U_{N-1}$ for $\BF{U}$
		and $R^V_1,R^V_2,\ldots,R^V_{N-1}$ for $\BF{V}$. }
  Initialize $\BF{U}$ and $\BF{V}$ 
		in block-$N$ TT format with 
		left-orthogonalized TT-cores $\ten{U}^{(1)},\ldots,
		\ten{U}^{(N-1)},\ten{V}^{(1)},\ldots,\ten{V}^{(N-1)}$ and 
		small TT-ranks $R_1^U,\ldots,R_{N-1}^U,R_1^V,\ldots,R_{N-1}^V$. 
\\
  Compute the 3rd order tensors $\ten{L}^{<1},\ldots,\ten{L}^{<N}$ 
  recursively by \eqref{def:psileft0} and \eqref{def:psileft}. Set $\ten{R}^{>N}=1$.
\\
\Repeat{a stopping criterion is met (See Section \ref{sec:stopping})}{
  \For(right-to-left half sweep){$n=N,N-1,\ldots,2$}{
	\tcp{Optimization}
		Compute $\BF{U}^{(n)}$ and $\BF{V}^{(n)}$ by solving \eqref{max_als}. 
\\
		Update the singular values $\BF{\Sigma} = 
		(\BF{U}^{(n)})^\RM{T}\overline{\BF{A}}_n\BF{V}^{(n)}$. 
\\
	\tcp{Matrix Factorization and Adaptive Rank Estimation}
		Reshape $\ten{U}^{(n)}
		= reshape(\BF{U}^{(n)}, [R^U_{n-1},  I_n,  R^U_n,  K])$,  
		$\ten{V}^{(n)}
		= reshape(\BF{V}^{(n)}, [R^V_{n-1},  J_n,  R^V_n,  K]).$
\\
		Compute $\delta$-truncated SVD: 
			\begin{equation}
			\begin{split}
			[\BF{U}_1, \BF{S}_1, \BF{V}_1] &= \RM{SVD}_\delta
		  	\left( \BF{U}^{(n)}_{(\{1,4\}\times\{2,3\})} \right),
			\\
			[\BF{U}_2, \BF{S}_2, \BF{V}_2] &= \RM{SVD}_\delta
			\left( \BF{V}^{(n)}_{(\{1,4\}\times\{2,3\})} \right). 
			\end{split}
			\end{equation}
\\
		Set $R^{U,new}=rank(\BF{V}_1)$, $R^{V,new}=rank(\BF{V}_2)$.
\\
		Update TT-cores 
			\begin{equation}
			\begin{split}
			\ten{U}^{(n)} &= reshape(\BF{V}_1^\text{T}, [R^{U,new},  I_n,  R^U_n]),  
			\\
			\ten{V}^{(n)} &= reshape(\BF{V}_2^\text{T}, [R^{V,new}, J_n,  R^V_n]). 
			\end{split}
			\end{equation}
\\
		Compute multiplications
			\begin{equation}
			\begin{split}
			\BF{U}^{(n-1)} &= reshape(\ten{U}^{(n-1)}, [R^U_{n-2}I_{n-1},R^U_{n-1}]) 
		   	\cdot reshape(\BF{U}_1\BF{S}_1,[R^U_{n-1},KR^{U,new}]), 
			\\
			\BF{V}^{(n-1)} &= reshape(\ten{V}^{(n-1)}, [R^V_{n-2}J_{n-1},R^V_{n-1}]) 
		   	\cdot reshape(\BF{V}_1\BF{S}_1,[R^V_{n-1},KR^{V,new}]). 
			\end{split}
			\end{equation}
\\
		Update TT-cores 
			\begin{equation}
			\begin{split}
			\ten{U}^{(n-1)} &= reshape(\BF{U}^{(n-1)}, [R^U_{n-2},I_{n-1},K,R^{U,new}]), 
			\\
			\ten{V}^{(n-1)} &= reshape(\BF{V}^{(n-1)}, [R^V_{n-2},J_{n-1},K,R^{V,new}]). 
			\end{split}
			\end{equation}
\\
		Update TT-ranks $R^U_{n-1}=R^{U,new}$, $R^V_{n-1}=R^{V,new}$.
\\
		Compute the 3rd order tensor $\ten{R}^{>n-1}$ by \eqref{def:psiright}
  }
  \For{$n=1,2,\ldots,N-1$}{
		Carry out left-to-right half sweep similarly
  }
}
\end{algorithm}

We note that the $K$ dominant singular values
$\BF{\Sigma}=\text{diag}(\sigma_1,\sigma_2,\ldots,\sigma_K)$ 
are equivalent to the $K$ dominant singular values of the 
projected matrix $\overline{\BF{A}}_n$ in the sense that 
$\BF{\Sigma} = \BF{U}^\RM{T}\BF{AV} = 
(\BF{U}^{(n)})^\RM{T}\overline{\BF{A}}_n\BF{V}^{(n)}$. 
Hence, the singular values are updated at each iteration 
by the singular values estimated by 
the standard SVD algorithms for the reduced optimization 
problem \eqref{max_als}. 

The TT-ranks of block TT formats are adaptively determined 
by separating the mode corresponding to the size $K$ from 
the $n$th TT-cores by using $\delta$-truncated SVD. 
We say that the iteration is during the right-to-left half sweep 
if the mode of the size $K$ is moving from the $n$th TT-core 
to the $(n-1)$th TT-core, whereas the iteration is during the 
left-to-right halft sweep if the mode of the size $K$ moves
from the $n$th TT core to the $(n+1)$th TT-core. 
During the right-to-left half sweep, 
the $\delta$-truncated SVD 
decomposes unfolded $n$th TT-cores as 
	\begin{equation}\label{eqn:deltaSVDshift}
	\begin{split}
	\BF{U}^{(n)}_{(\{1,4\}\times\{2,3\})}
	& = \BF{U}_1 \BF{S}_1 \BF{V}_1^\RM{T} + \BF{E}_1
	\in\BB{R}^{R^U_{n-1}K\times I_{n}R^U_{n}}, \\
	\BF{V}^{(n)}_{(\{1,4\}\times\{2,3\})}
	& = \BF{U}_2 \BF{S}_2 \BF{V}_2^\RM{T} +\BF{E}_2
	\in\BB{R}^{R^V_{n-1}K\times J_{n}R^V_{n}}, 
	\end{split}
	\end{equation}
where $\|\BF{E}_1\|_\RM{F}\leq \delta\|\ten{U}^{(n)}\|_\RM{F}$ and 
$\|\BF{E}_2\|_\RM{F}\leq \delta\|\ten{V}^{(n)}\|_\RM{F}$. 
Then, the TT-ranks are updated by $R^U_{n-1} = rank(\BF{V}_1)$
and $R^V_{n-1} = rank(\BF{V}_2)$, which are simply the numbers of 
columns of $\BF{V}_1$ and $\BF{V}_2$. The $n$th TT-cores are 
updated by reshaping $\BF{V}_1^\RM{T}$ and $\BF{V}_2^\RM{T}$. 
Note that in the case that $K=1$, the TT-ranks $R^U_{n-1}$ and 
$R^V_{n-1}$ cannot be increased because, for example,  
	\begin{equation}\label{eqn:boundTTrank}
	R^{U,new} = rank(\BF{V}_1)= rank(\BF{U}_1\BF{S}_1) \leq R^U_{n-1}K = R^U_{n-1}. 
	\end{equation}

Figure \ref{Fig:ALS} illustrates the ALS scheme based on block TT format 
for the first two iterations. In the figure, the $n$th TT-core is computed 
by a local optimization algorithm for the maximization problem \eqref{max_als}, 
and then the block-$n$ TT format
is converted to the block-$(n-1)$ TT format via the $\delta$-truncated SVD. 

\begin{figure}
\centering
\includegraphics[width=10cm]{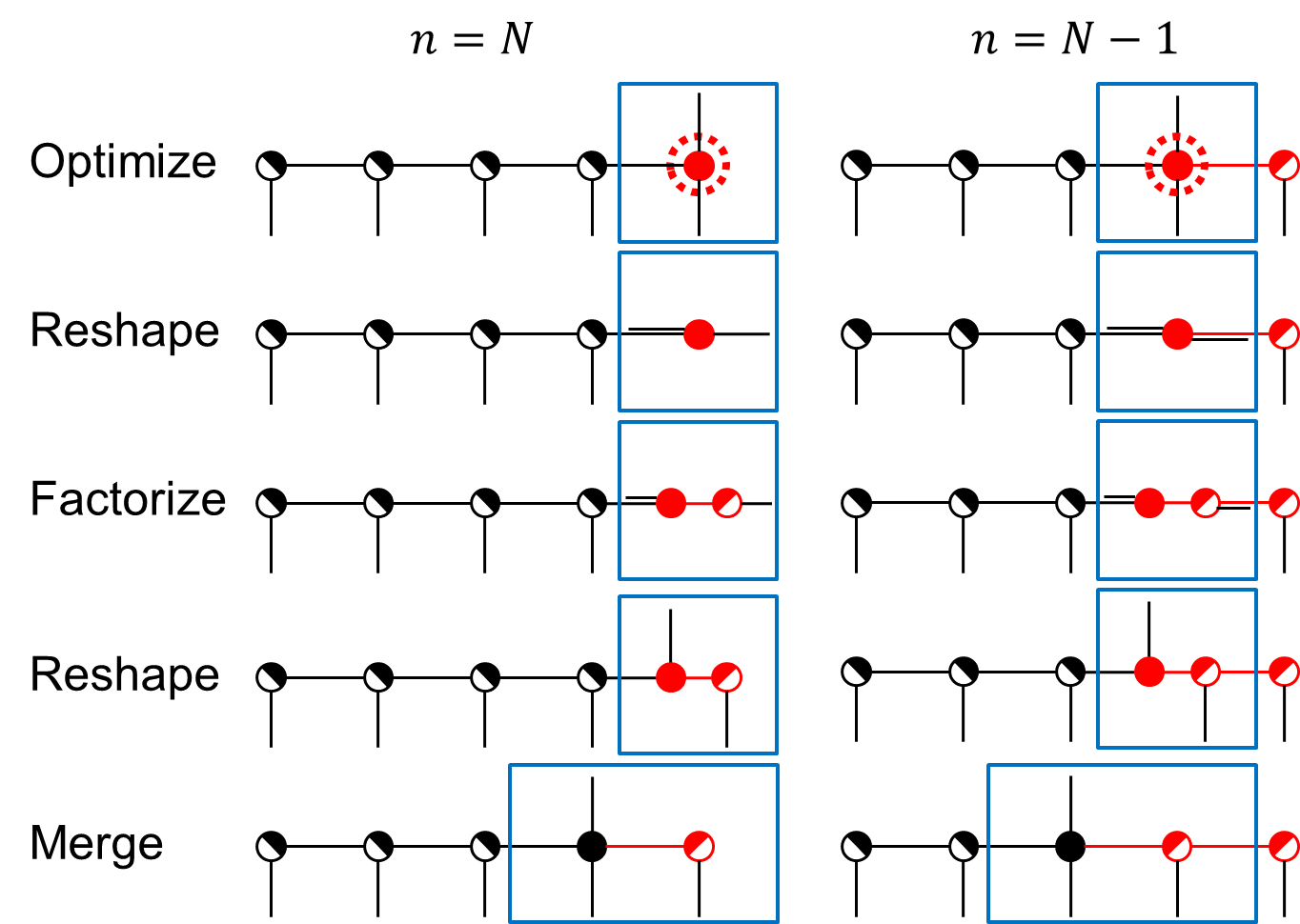} 
\caption{\label{Fig:ALS}Illustration of 
the ALS scheme based on block TT format 
for the first two iterations 
during right-to-left half sweep}
\end{figure}

\subsection{MALS for SVD Based on Block TT Format}

In the ALS scheme, the TT-ranks cannot be adaptively determined 
if $K=1$. Moreover, a small value of $K$ often slows the rate of 
convergence because of the relatively slow growth of TT-ranks,
which is described in the inequality in \eqref{eqn:boundTTrank}. 
On the other hand, the MALS scheme shows relatively 
fast convergence and the TT-ranks can be adaptively determined 
even if $K=1$. 

The MALS algorithm for SVD is described in Algorithm 
\ref{Alg:MALS}. In the MALS scheme, 
the right-to-left half sweep means the iterations when 
the $(n-1)$th and $n$th TT-cores are updated for $n=N,N-1,\ldots,2$, 
and the left-to-right half sweep means the iterations when the 
$n$th and $(n+1)$th TT-cores are updated for $n=1,2,\ldots,N-1$. 
At each iteration during the right-to-left half sweep, 
the left and right singular vectors $\BF{U}$ and $\BF{V}$
are represented in block-$n$ TT format. 
From \eqref{eqn:Umatrix_dmrg}, 
the matrices $\BF{U}$ and $\BF{V}$ are written by 
	\begin{equation}
	\BF{U}=\BF{U}^{\neq n-1,n}\BF{U}^{(n-1,n)}, 
	\qquad 
	\BF{V}=\BF{V}^{\neq n-1,n}\BF{V}^{(n-1,n)} 
	\end{equation}
for $n=2,3,\ldots,N,$
where $\BF{U}^{(n-1,n)}\in\BB{R}^{R^U_{n-2}I_{n-1}I_{n}R^U_{n}\times K}$ and $\BF{V}^{(n-1,n)}\in\BB{R}^{R^V_{n-2}J_{n-1}J_{n}R^V_{n}\times K}$
are matricizations of the merged TT-cores 
$\ten{U}^{(n-1)}\bullet\ten{U}^{(n)}$ and 
$\ten{V}^{(n-1)}\bullet\ten{V}^{(n)}$. 
Given that all the TT-cores are fixed except the $(n-1)$th and $n$th TT-cores, the large-scale optimization problem \eqref{eqn:maximize} is
reduced to 
	\begin{equation}\label{max_mals}
	\begin{split}
	\maximize_{\BF{U}^{(n-1,n)}, \BF{V}^{(n-1,n)}}
	& \qquad
	\text{trace}\left(\BF{U}^\RM{T}\BF{AV}\right)
	= \text{trace}\left( (\BF{U}^{(n-1,n)} )^\RM{T}\overline{\BF{A}}_{n-1,n}\BF{V}^{(n-1,n)}\right)\\
	\text{subject to}
	& \qquad 
	(\BF{U}^{(n-1,n)})^\RM{T}\BF{U}^{(n-1,n)}=
	(\BF{V}^{(n-1,n)})^\RM{T}\BF{V}^{(n-1,n)}=\BF{I}_K, 
	\end{split}
	\end{equation}
where 
	\begin{equation}
	\overline{\BF{A}}_{n-1,n} = 
(\BF{U}^{\neq n-1,n})^\RM{T}\BF{A}\BF{V}^{\neq n-1,n} \in\BB{R}^{R^U_{n-2}I_{n-1}I_{n}R^U_{n}
\times R^V_{n-2}J_{n-1}J_{n}R^V_{n}}
	\end{equation}
is called as the projected matrix. 

Figure \ref{Fig:MALS} illustrates the MALS scheme. In the MALS, 
two neighboring core tensors are first merged and updated 
by solving the optimization problem \eqref{max_mals}. 
Then, the $\delta$-truncated SVD factorizes it back into two 
core tensors. The block-$n$ TT format is transformed 
into either the block-$(n-1)$ TT format or the 
block-$(n+1)$ TT format consequently. 

\begin{figure}
\centering
\includegraphics[width=10cm]{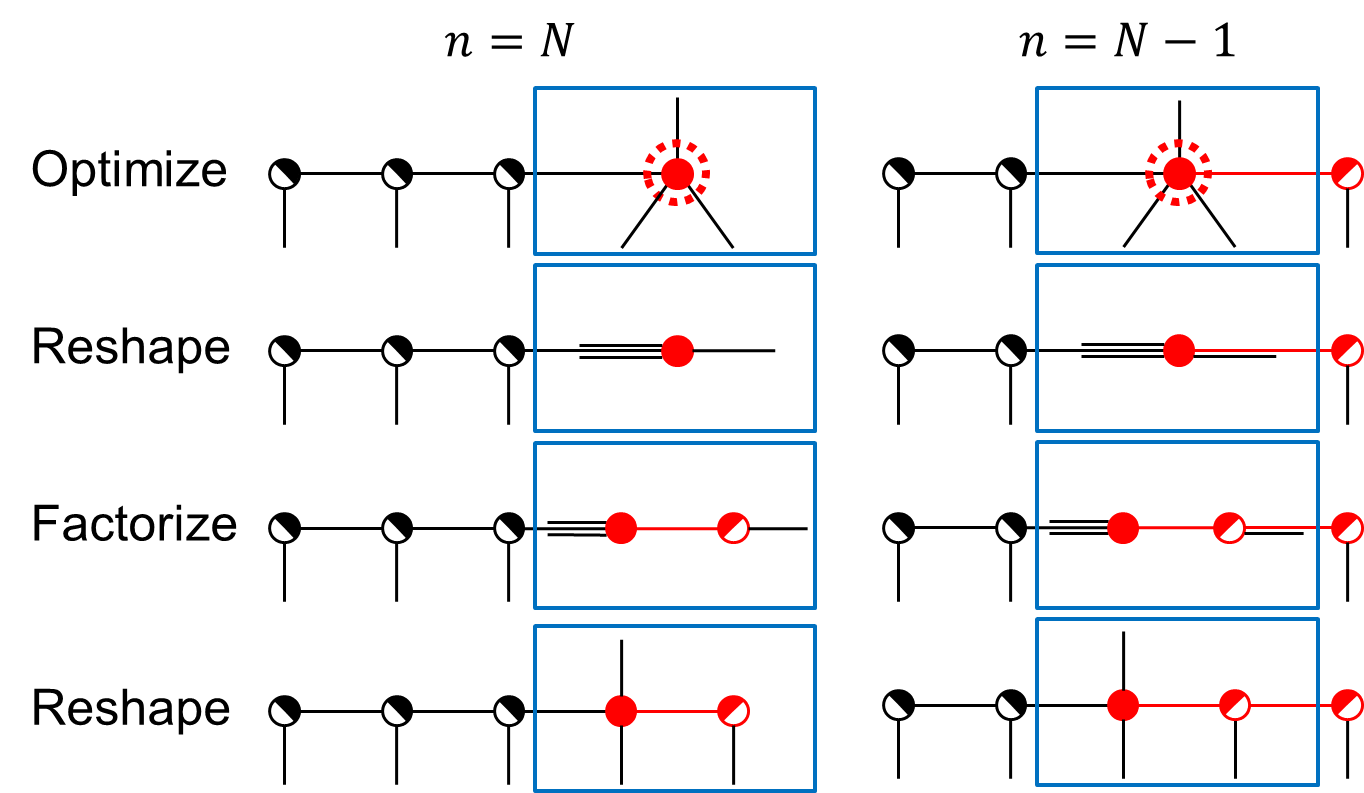} 
\caption{\label{Fig:MALS}Illustration of 
the MALS scheme based on block TT format 
for the first two iterations 
during right-to-left half sweep}
\end{figure}

\begin{algorithm}
\caption{\label{Alg:MALS}MALS for SVD based on block TT format}
  \SetKwInOut{Input}{Input}
  \SetKwInOut{Output}{Output}

  \Input{$\BF{A}\in\BB{R}^{I_1I_2\cdots I_N\times J_1J_2\cdots J_N}$ 
		in matrix TT format, $K\geq 1$, $\delta\geq 0$ (truncation parameter)}
  \Output{Dominant singular values $\BF{\Sigma}=
		\text{diag}(\sigma_1,\sigma_2,\ldots,\sigma_K)$
		and corresponding singular vectors 
		$\BF{U} \in\BB{R}^{I_1I_2\cdots I_N\times K}$ and 
		$\BF{V} \in\BB{R}^{J_1J_2\cdots J_N\times K}$ in block-$N$ TT format 
		with TT-ranks $R^U_1,R^U_2,\ldots,R^U_{N-1}$ for $\BF{U}$
		and $R^V_1,R^V_2,\ldots,R^V_{N-1}$ for $\BF{V}$.}
  Initialize $\BF{U}$ and $\BF{V}$ 
		in block-$N$ TT format with 
		left-orthogonalized TT-cores $\ten{U}^{(1)},\ldots,
		\ten{U}^{(N-2)},\ten{V}^{(1)},\ldots,\ten{V}^{(N-2)}$ and 
		small TT-ranks $R_1^U,\ldots,R_{N-1}^U,R_1^V,\ldots,R_{N-1}^V$. 
\\
  Compute the 3rd order tensors $\ten{L}^{<1},\ldots,\ten{L}^{<N-1}$ 
  recursively by \eqref{def:psileft0} and \eqref{def:psileft}. Set $\ten{R}^{>N}=1$.
\\
\Repeat{a stopping criterion is met (See Section \ref{sec:stopping})}{
  \For(right-to-left half sweep){$n=N,N-1,\ldots,2$}{
	\tcp{Optimization}
	Compute $\BF{U}^{(n-1,n)}$ and $\BF{V}^{(n-1,n)}$ by solving 
	\eqref{max_mals}. 
\\
	Update the singular values $\BF{\Sigma}=(\BF{U}^{(n-1,n)})^\RM{T} 
	\overline{\BF{A}}_{n-1,n} \BF{V}^{(n-1,n)}$. 
\\
	\tcp{Matrix Factorization and Adaptive Rank Estimation}
		Reshape $\ten{U}^{(n-1,n)} = reshape(\BF{U}^{(n-1,n)}, 
		[R^U_{n-2},  I_{n-1}, I_{n}, R^U_{n}, K])$,  
		$\ten{V}^{(n-1,n)} = reshape(\BF{V}^{(n-1,n)}, 
		[R^V_{n-2}, J_{n-1}, J_{n}, R^V_{n}, K])$. 
\\
		Compute $\delta$-truncated SVD: 
			\begin{equation}
			\begin{split}
			[\BF{U}_1, \BF{S}_1, \BF{V}_1] &= \RM{SVD}_\delta
		  	\left( \BF{U}^{(n-1,n)}_{(\{1,2,5\}\times\{3,4\})} \right), 
			\\
			[\BF{U}_2, \BF{S}_2, \BF{V}_2] &= \RM{SVD}_\delta
		  	\left( \BF{V}^{(n-1,n)}_{(\{1,2,5\}\times\{3,4\})} \right). 
			\end{split}
			\end{equation}
\\
		Update TT-ranks $R^U_{n-1}=rank(\BF{V}_1)$,
		$R^V_{n-1}=rank(\BF{V}_2)$.
\\
		Update TT-cores 
			\begin{equation}
			\begin{split}
			\ten{U}^{(n)} &=reshape(\BF{V}_1^\text{T}, [R^U_{n-1},I_{n},R^U_n]), 
			\\
			\ten{V}^{(n)} &=reshape(\BF{V}_2^\text{T}, [R^V_{n-1},J_{n},R^V_n]).
			\end{split}
			\end{equation}
\\
		Update TT-cores
			\begin{equation}
			\begin{split}
			\ten{U}^{(n-1)} &= reshape(\BF{U}_1\BF{S}_1, 
			[R_{n-2}^U,I_{n-1},K,R_{n-1}^U]), 
			\\
			\ten{V}^{(n-1)} &= reshape(\BF{U}_2\BF{S}_2, 
			[R_{n-2}^V,J_{n-1},K,R_{n-1}^V]).
			\end{split}
			\end{equation}
\\
		Compute the 3rd order tensor $\ten{R}^{>n-1}$ by \eqref{def:psiright}
  }
  \For{$n=1,2,\ldots,N-1$}{
		Carry out left-to-right half sweep similarly
  }
}
\end{algorithm}

\subsection{Efficient Computation of Projected Matrix-by-Vector Product}
\label{sec:3_3}

In order to solve the reduced optimization problems \eqref{max_als}
and \eqref{max_mals}, we consider the eigenvalue decomposition 
of the block matrices 
	\begin{equation}
	\overline{\BF{B}}_n = \begin{bmatrix}
	\BF{0} & \overline{\BF{A}}_n \\
	\overline{\BF{A}}_n^\RM{T} & \BF{0}
	\end{bmatrix}, 
	\qquad 
	\overline{\BF{B}}_{n-1,n} = \begin{bmatrix}
	\BF{0} & \overline{\BF{A}}_{n-1,n} \\
	\overline{\BF{A}}_{n-1,n}^\RM{T} & \BF{0}
	\end{bmatrix}. 
	\end{equation}
It can be shown that the $K$ largest eigenvalues of 
$\overline{\BF{B}}_n$ correspond to the $K$ dominant singular 
values of the projected matrix $\overline{\BF{A}}_n$, and 
the eigenvectors of $\overline{\BF{B}}_n$ correspond to 
a concatenation of the left and right singular vectors of 
$\overline{\BF{A}}_n$. See Appendix \ref{sec:app1} for more detail. 
The same holds for $\overline{\BF{B}}_{n-1,n}$ and 
$\overline{\BF{A}}_{n-1,n}$. 

For computing the 
eigenvalue decomposition of the above matrices, 
we don't need to compute the matrices explicitly, 
but we only need to compute matrix-by-vector 
products. 
Let 
	\begin{equation}
	\BF{x}
	\in\BB{R}^{R^U_{n-1}I_nR^U_n}, 
	\quad 
	\BF{y}
	\in\BB{R}^{R^V_{n-1}J_nR^V_n}, 
	\quad
	\tilde{\BF{x}}
	\in\BB{R}^{R^U_{n-2}I_{n-1}I_nR^U_n}, 
	\quad
	\tilde{\BF{y}}
	\in\BB{R}^{R^V_{n-2}J_{n-1}J_nR^V_n}
	\end{equation}
be given vectors. Then, the matrix-by-vector products are expressed by 
	\begin{equation}
	\overline{\BF{B}}_n \begin{bmatrix} \BF{x}\\\BF{y}  \end{bmatrix}
	= \begin{bmatrix} \overline{\BF{A}}_n\BF{y}\\
	   \overline{\BF{A}}_n^\RM{T}\BF{x}  \end{bmatrix}, 
	\qquad 
	\overline{\BF{B}}_{n-1,n} \begin{bmatrix} \tilde{\BF{x}}\\\tilde{\BF{y}}  \end{bmatrix}
	= \begin{bmatrix} \overline{\BF{A}}_{n-1,n}\tilde{\BF{y}}\\
	   \overline{\BF{A}}_{n-1,n}^\RM{T}\tilde{\BF{x}}  \end{bmatrix},  
	\end{equation}
which consist of the projected matrix-by-vector products, 
$\overline{\BF{A}}_n\BF{y}, \overline{\BF{A}}_n^\RM{T}\BF{x}, 
\overline{\BF{A}}_{n-1,n}\tilde{\BF{y}}, $ and 
$\overline{\BF{A}}_{n-1,n}^\RM{T}\tilde{\BF{x}}$. 

The computation of the projected matrix-by-vector products 
is performed in an iterative way as follows. 
Let $\BF{U}^{(m)}_{i_m} = \ten{U}^{(m)}(:,i_m,:)$, 
$\BF{A}^{(m)}_{i_m,j_m} = \ten{A}^{(m)}(:,i_m,j_m,:)$, 
and $\BF{V}^{(m)}_{j_m} = \ten{V}^{(m)}(:,j_m,:)$ be 
the slice matrices of the three $m$th core tensors for $m\neq n$.
Let 
	\begin{equation}\label{eqn:z_block_rec}
	\BF{Z}^{(m)} = \sum_{i_m=1}^{I_m} \sum_{j_m=1}^{J_m} 
	\BF{U}^{(m)}_{i_m} \otimes 
	\BF{A}^{(m)}_{i_m,j_m} \otimes \BF{V}^{(m)}_{j_m}
	\in\BB{R}^{R^U_{m-1}R^A_{m-1}R^V_{m-1} \times R^U_{m}R^A_{m}R^V_{m}}. 
	\end{equation}
We define 3rd order tensors 
$\ten{L}^{<m} \in\BB{R}^{R^U_{m-1}\times R^A_{m-1} \times R^V_{m-1}}$,  
$m=1,2,\ldots,n,$ and 
$\ten{R}^{>m} \in\BB{R}^{R^U_{m}\times R^A_{m}\times R^V_{m}}$, 
$m=n,n+1,\ldots,N,$ recursively by 
	\begin{equation}\label{def:psileft0}
	\text{vec}\left(\ten{L}^{<1}\right) = 1, 
	\end{equation}
	\begin{equation}\label{def:psileft}
	\text{vec}\left(\ten{L}^{<m}\right)^\RM{T} = 
	\text{vec}\left(\ten{L}^{<m-1}\right)^\RM{T} \BF{Z}^{(m-1)}
	\in\BB{R}^{1\times R^U_{m-1}R^A_{m-1}R^V_{m-1}}, 
	\qquad m=2,3,\ldots,n,
	\end{equation}
and 
	\begin{equation}
	\text{vec}\left(\ten{R}^{>N}\right) = 1, 
	\end{equation}
	\begin{equation}\label{def:psiright}
	\text{vec}\left(\ten{R}^{>m}\right) = 
	\BF{Z}^{(m+1)}\text{vec}\left(\ten{R}^{>m+1}\right)
	\in\BB{R}^{R^U_{m}R^A_{m}R^V_{m}\times 1}, 
	\qquad m=n,n+1,\ldots,N-1. 
	\end{equation}
Recall that $(\BF{U}^{(n)})^\RM{T}\overline{\BF{A}}_n\BF{V}^{(n)}
=\BF{U}^\RM{T}\BF{AV}$. 
Let $\BF{u}^{(n)}_{k_1}$ and $\BF{v}^{(n)}_{k_2}$
denote the $k_1$th and $k_2$th column vectors of 
the matrices $\BF{U}^{(n)}$ and $\BF{V}^{(n)}$. 
From the matrix product representations of the 
matrix TT and block TT formats \eqref{eqn:matprod_matrixTT}
and \eqref{eqn:matprod_blockTT}, we can show that 
the $(k_1,k_2)$ entry of 
$(\BF{U}^{(n)})^\RM{T}\overline{\BF{A}}_n\BF{V}^{(n)}$ is 
expressed by 
	\begin{equation}
	\begin{split}
	(\BF{u}^{(n)}_{k_1})^\RM{T}
	\overline{\BF{A}}_n\BF{v}^{(n)}_{k_2}
	& = 
	\BF{u}_{k_1}^\RM{T}\BF{Av}_{k_2} \\
	& = \BF{Z}^{(1)}\cdots \BF{Z}^{(n-1)}
	\left(\sum_{i_n=1}^{I_n}\sum_{j_n=1}^{J_n}
	\BF{U}^{(n)}_{k_1,i_n}
	\otimes \BF{A}^{(n)}_{i_n,j_n} 
	\otimes \BF{V}^{(n)}_{k_2,j_n} \right)
	\BF{Z}^{(n+1)}\cdots \BF{Z}^{(N)} 
	\\
	& = 
	\text{vec}\left( \ten{L}^{<n} \right)^\RM{T}
	\left(\sum_{i_n=1}^{I_n}\sum_{j_n=1}^{J_n}
	\BF{U}^{(n)}_{k_1,i_n}
	\otimes \BF{A}^{(n)}_{i_n,j_n} 
	\otimes \BF{V}^{(n)}_{k_2,j_n} \right)
	\text{vec}\left( \ten{R}^{>n} \right), 
	\end{split}
	\end{equation}
which is the contraction of the tensors 
$\ten{L}^{<n}, \ten{U}^{(n)}, \ten{A}^{(n)}, 
\ten{V}^{(n)},$ and $\ten{R}^{>n}$. 
Thus, the computation of 
$\overline{\BF{A}}_n\BF{y}$ 
is performed by the contraction
of the tensors $\ten{L}^{<n}$, $\ten{Y}$, $\ten{A}^{(n)}$, 
and $\ten{R}^{>n}$, where 
$\ten{Y}\in\BB{R}^{R^V_{n-1}\times J_n\times R^V_n}$
is the tensorization of the vector $\BF{y}$. 
In the same way, the computation of 
$\overline{\BF{A}}_n^\RM{T}\BF{x}$ 
is performed by the contraction
of the tensors $\ten{L}^{<n}$, $\ten{X}$, $\ten{A}^{(n)}$, 
and $\ten{R}^{>n}$, where 
$\ten{X}\in\BB{R}^{R^U_{n-1}\times I_n\times R^U_n}$
is the tensorization of the vector $\BF{x}$.

Similarly, we can derive 
an expression for the projected matrix $\overline{\BF{A}}_{n-1,n}$ as
	\begin{equation}
	\begin{split}
	&(\BF{u}^{(n-1,n)}_{k_1})^\RM{T}
	\overline{\BF{A}}_{n-1,n}\BF{v}^{(n-1,n)}_{k_2}
	\\
	&= \text{vec}\left( \ten{L}^{<n-1} \right)^\RM{T}
	\left(\sum_{i_{n-1}}\sum_{i_n}\sum_{j_{n-1}}\sum_{j_n}
	\BF{U}^{(n-1,n)}_{i_{n-1},k_1,i_n}
	\otimes 
	\left(\BF{A}^{(n-1)}_{i_{n-1},j_{n-1}} 
	\BF{A}^{(n)}_{i_n,j_n} 
	\right)
	\otimes \BF{V}^{(n-1,n)}_{j_{n-1},k_2,j_n} \right)
	\text{vec}\left( \ten{R}^{>n} \right), 
	\end{split}
	\end{equation}
which is the contraction of the tensors 
$\ten{L}^{<n-1}$, $\ten{U}^{(n-1,n)}$, $\ten{A}^{(n-1)}$, 
$\ten{A}^{(n)}$, $\ten{V}^{(n-1,n)}$, and 
$\ten{R}^{>n}$. Thus, the computation of 
$\overline{\BF{A}}_{n-1,n}\tilde{\BF{y}}$ 
is performed by the contraction
of the tensors $\ten{L}^{<n-1}$, $\widetilde{\ten{Y}}$, 
$\ten{A}^{(n-1)}$, $\ten{A}^{(n)}$, 
and $\ten{R}^{>n}$, where 
$\widetilde{\ten{Y}}\in\BB{R}^{R^V_{n-2}\times J_{n-1}\times J_n
\times R^V_n}$
is the tensorization of the vector $\tilde{\BF{y}}$. 
In the same way, the computation of 
$\overline{\BF{A}}_n^\RM{T}\tilde{\BF{x}}$ 
is performed by the contraction
of the tensors $\ten{L}^{<n-1}$, $\widetilde{\ten{X}}$, 
$\ten{A}^{(n-1)}$, $\ten{A}^{(n)}$, 
and $\ten{R}^{>n}$, where 
$\widetilde{\ten{X}}\in\BB{R}^{R^U_{n-2}\times I_{n-1}\times I_n
\times R^U_n}$
is the tensorization of the vector $\tilde{\BF{x}}$. 

Figure \ref{Fig:ProjMatProd} illustrates the tensor 
network diagrams for the computation of
the projected matrix-by-vector products
$\overline{\BF{A}}_n\BF{y}$ and $\overline{\BF{A}}_{n-1,n}\tilde{\BF{y}}$
for the vectors $\BF{y}\in\BB{R}^{R^V_{n-1}J_nR^V_n}$
and $\tilde{\BF{y}}\in\BB{R}^{R^V_{n-2}J_{n-1}J_nR^V_n}$. 
Based on the tensor network diagrams, 
we can easily specify the sizes of the tensors and 
how the tensors are contracted with each other. 

\begin{figure}
\centering
\begin{tabular}{c}
\includegraphics[width=8cm]{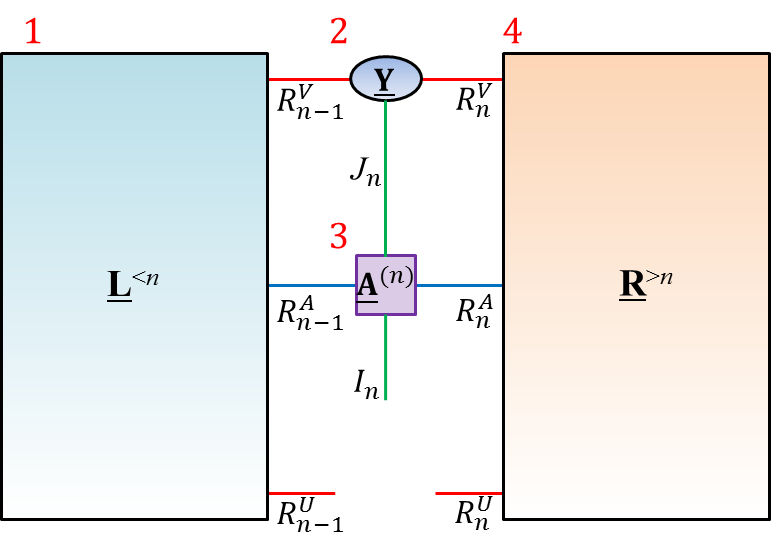}  \\
(a) $\overline{\BF{A}}_n\BF{y}$ for ALS-SVD \\
\includegraphics[width=8cm]{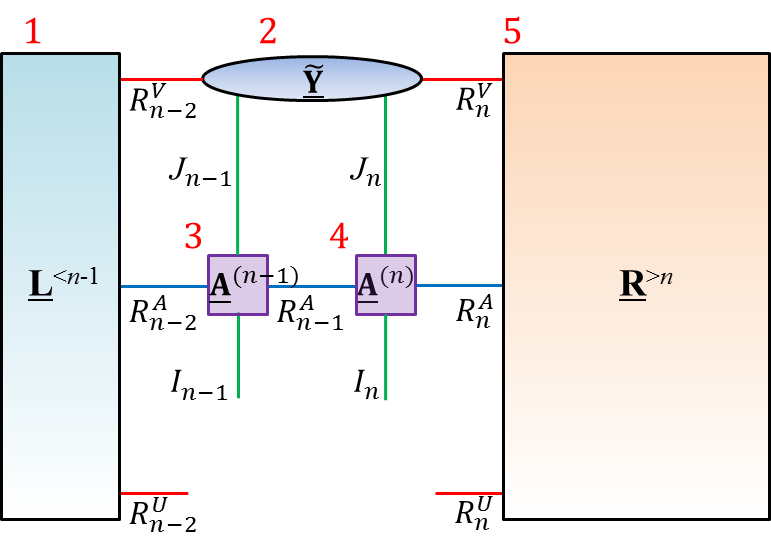}  \\
(b) $\overline{\BF{A}}_{n-1,n}\tilde{\BF{y}}$ for MALS-SVD
\end{tabular}
\caption{\label{Fig:ProjMatProd}
Computation of the projected matrix-by-vector products
(a) $\overline{\BF{A}}_n\BF{y}$ and 
(b) $\overline{\BF{A}}_{n-1,n}\tilde{\BF{y}}$
for solving the reduced optimization problems in 
the ALS-SVD and MALS-SVD algorithms. 
(a) The computation of $\overline{\BF{A}}_n\BF{y}$
is carried out by the contraction of 
$\ten{L}^{<n}$, $\ten{Y}$, $\ten{A}^{(n)}$, and 
$\ten{R}^{>n}$.
(b) The computation of $\overline{\BF{A}}_{n-1,n}\tilde{\BF{y}}$
is carried out by the contraction of 
$\ten{L}^{<n-1}$, $\widetilde{\ten{Y}}$, $\ten{A}^{(n-1)}$, 
$\ten{A}^{(n)}$, and 
$\ten{R}^{>n}$.
An efficient order of contraction is expressed by 
the numbers on each tensor. 
}
\end{figure}

\subsection{Computational Complexity}

Let $R=max(\{R^U_n\},\{R^V_n\})$, $R_A=max(\{R^A_n\})$, 
and $I=max(\{I_n\},\{J_n\})$. The computational complexities 
for the ALS-SVD and MALS-SVD algorithms are summarized 
in Table \ref{table1:complexity}. We may assume that $K\geq I$
because we usually choose very small values for $I_n$ and $J_n$, 
e.g., $I=I_n=J_n=2$. 
The computational complexities in Table \ref{table1:complexity}
correspond to each iteration, so the total computational costs 
for one full sweep (right-to-left and left-to-right half sweeps) 
grow linearly with the order $N$ given that 
$R,R_A,I,$ and $K$ are bounded.

\begin{table}
\caption{\label{table1:complexity}
Computational complexities of the ALS-SVD and 
MALS-SVD algorithms for one iteration
}
\vspace{0.5pc}
\centering
\begin{tabular}{cll}
\hline
 & \multicolumn{1}{c}{ALS-SVD} & 
	\multicolumn{1}{c}{MALS-SVD} \\
\hline
Orthonormalization & line 5: $\CL{O}(K^2IR^2)$ & line 5: $\CL{O}(K^2I^2R^2)$ \\
Projected Matrix-by-
	& line 5: $\CL{O}(KIR_A(R+IR_A)R^2)$ 
	& line 5: $\CL{O}(KI^2R_A(R + IR_A)R^2)$ \\
  Vector Products  & $\qquad\quad $& \\
Factorization 
	& line 8: $\CL{O}(KI^2R^3)$ 
	& line 8: $\CL{O}(KI^3R^3)$ \\
Updating Tensor 
	& line 12: $\CL{O}(IR_A(R + IR_A)R^2)$ 
	& line 14: $\CL{O}(IR_A(R + IR_A)R^2)$ \\
$\ten{L}^{<n}$ or $\ten{R}^{>n}$ 	&& \\
\hline
\end{tabular}
\end{table}

At the optimization step at line 5 of the ALS-SVD, 
standard SVD algorithms perform two important computations 
\cite{Liu2013}: the orthonormalization of a matrix 
$\BF{V}^{(n)}\in\BB{R}^{R^V_{n-1}J_nR^V_n\times K}$
via QR decomposition, and the projected matrix-by-vector 
products, $\overline{\BF{A}}_n \BF{V}^{(n)}$. 
The computational complexity for the QR decomposition is 
$\CL{O}(K^2IR^2)$. For the MALS-SVD, the size of 
$\BF{V}^{(n-1,n)}$ is $R^V_{n-2}J_{n-1}J_nR^V_n\times K$, 
which leads to the computational complexity $\CL{O}(K^2I^2R^2)$. 

The computational complexities for the projected matrix-by-vector 
products can be conveniently analyzed by using the tensor network 
diagrams in Figure \ref{Fig:ProjMatProd}. 
In Figure \ref{Fig:ProjMatProd}(a), an efficient order of contraction
for computing $\overline{\BF{A}}_n\BF{y}$ is 
$(\ten{L}^{<n}, \ten{Y}, \ten{A}^{(n)}, \ten{R}^{>n})$, and 
its computational complexity is 
$\CL{O}(IR_A(R+IR_A)R^2)$. 
Since the matrix-by-vector product is performed for $K$ column 
vectors of $\BF{V}^{(n)}$, 
the computational complexity is multiplied by $K$. 
On the other hand, if we compute the contractions in the 
order of $(\ten{L}^{<n}, \ten{A}^{(n)}, \ten{Y}, 
\ten{R}^{>n})$, however, the computational complexity 
increases to $\CL{O}(I^2R_A(R + R_A)R^2)$. 
On the other hand, an explicit computation of 
the matrix $\overline{\BF{A}}_n$ can be performed 
by the contraction of $\ten{L}^{<n}, \ten{A}^{(n)},$ and 
$\ten{R}^{>n}$, which costs $\CL{O}(I^2R_A(R^2 + R_A)R^2)$. 
Thus, it is recommended to avoid computing the 
projected matrices explicitly. 

In Figure \ref{Fig:ProjMatProd}(b), one of the most efficient 
orders of contractions for computing 
$\overline{\BF{A}}_{n-1,n}\tilde{\BF{y}}$
is $(\ten{L}^{<n-1}, \widetilde{\ten{Y}}, $ $\ten{A}^{(n-1)}, 
\ten{A}^{(n)}, \ten{R}^{>n})$, and the 
computational complexity amounts to 
$\CL{O}(I^2R_A(R + IR_A)R^2)$. 
However, if we follow the order of $(\ten{L}^{<n-1}, \ten{A}^{(n-1)}, 
\ten{A}^{(n)}, \widetilde{\ten{Y}}, \ten{R}^{>n})$ for 
contraction, it costs as highly as $\CL{O}(I^4R_A(R+R_A)R^2)$. 
Moreover, if we have to compute the projected matrix 
$\overline{\BF{A}}_{n-1,n}$ explicitly, its computational cost increases to 
$\CL{O}(I^4R_A(R^2 + R_A)R^2)$. 

For performing the factorization step of the ALS-SVD, 
the truncated SVD costs $\CL{O}(min(K^2IR^3, KI^2R^3))$. 
Since we assume that $K\geq I$, we choose 
$\CL{O}(KI^2R^3)$. For the MALS-SVD, the computational 
complexity is $\CL{O}(KI^3R^3)$. 

Both the ALS-SVD and MALS-SVD maintain 
the recursively defined {\it left} and {\it right} tensors 
$\ten{L}^{<m},m=1,2,\ldots,n,$ and $\ten{R}^{>m},m=n,n+1,\ldots,N,$
during iteration. At each iteration, the algorithms update the left or right 
tensors after the factorization step by the definitions \eqref{def:psileft}
or \eqref{def:psiright}. For example, during the right-to-left half sweep, 
the right tensor $\ten{R}^{>n-1}$ is computed by the equation 
\eqref{def:psiright}.  Figure \ref{Fig:righttensornetwork} shows the 
tensor network diagram for the computation of $\ten{R}^{>n-1}$ 
during the right-to-left half sweep. In the figure we can see that the 
mode corresponding to the size $K$ has already been shifted to the 
$(n-1)$th core tensor. From \eqref{def:psiright} and 
\eqref{eqn:z_block_rec}, we can see that the computation is performed 
by the contraction of the tensors 
$\ten{U}^{(n)}, \ten{A}^{(n)}, \ten{V}^{(n)},$ and $\ten{R}^{>n}$. 
An optimal order of contraction is 
$(\ten{R}^{>n}, \ten{V}^{(n)}, \ten{A}^{(n)}, \ten{U}^{(n)})$, 
and its computational complexity is $\CL{O}(IR_A(R + IR_A)R^2)$
for both the ALS-SVD and MALS-SVD.  
On the other hand, if we compute the contractions
in the order of $(\ten{R}^{>n}, \ten{A}^{(n)}, \ten{U}^{(n)}, \ten{V}^{(n)})$
or $(\ten{R}^{>n}, \ten{U}^{(n)}, \ten{V}^{(n)}, \ten{A}^{(n)})$, 
then the computational cost increases to 
$\CL{O}(I^2R_A(R + R_A)R^2)$.

\begin{figure}
\centering
\includegraphics[width=8.2cm]{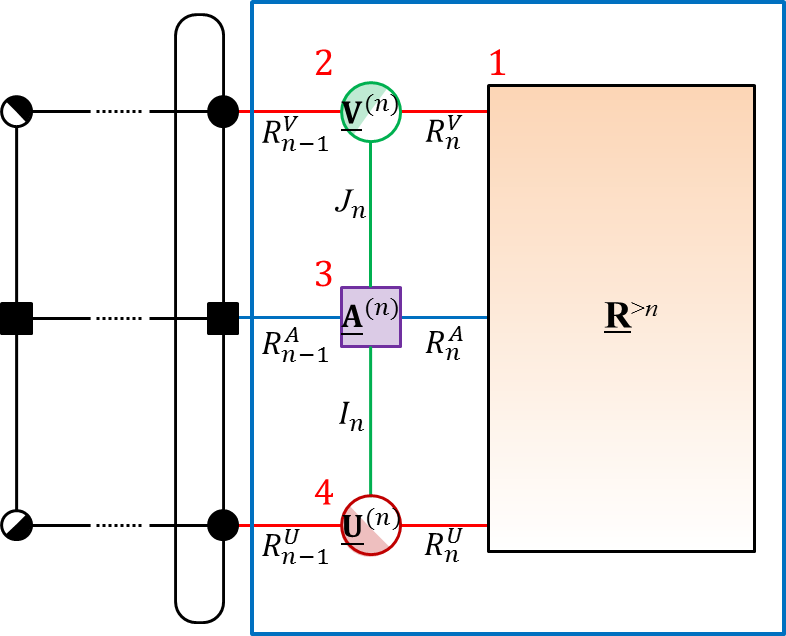}
\caption{\label{Fig:righttensornetwork}Iterative computation of 
the right tensor $\ten{R}^{>n-1}$ by the 
contraction of the tensors $\ten{R}^{>n}$, 
$\ten{V}^{(n)}$, $\ten{A}^{(n)}$, and $\ten{U}^{(n)}$
during the right-to-left half sweep. The optimal 
order of contraction is expressed 
by the numbers on each tensor. 
}
\end{figure}

\subsection{Computational Considerations}

\subsubsection{Initialization}

The computational costs of the ALS-SVD and MALS-SVD 
are significantly affected by the TT-ranks of the left and 
right singular vectors. Even if we have good initial guesses for 
the left and right singular vectors $\BF{U}$ and $\BF{V}$, 
it may take much computational time until convergence 
if their TT-ranks are large. Therefore, it is advisable to initialize 
$\BF{U}$ and $\BF{V}$ in block TT format with relatively small 
TT-ranks. The minimum values of the TT-ranks are determined
by 
	\begin{equation}
	R^U_n=\left\lceil K/(I_{n+1}\cdots I_N)\right\rceil,
	\quad 
	R^V_n=\left\lceil K/(J_{n+1}\cdots J_N)\right\rceil,
	\quad n=1,\ldots,N, 
	\end{equation}
because each TT-core is initially left-orthogonalized and 
satisfies, e.g., $R^U_{n-1}I_n\geq R^U_n, n=1,2,\ldots,N-1$, 
and the $N$th TT-core is also orthogonalized in the sense 
that $(\BF{U}^{(N)})^\RM{T}\BF{U}^{(N)}=\BF{I}_K$, 
and it must satisfy $R^U_{N-1}I_N\geq K$. Since both the 
ALS-SVD and MALS-SVD can adaptively determine TT-ranks 
during iteration process when $K\geq 2$, they update the 
singular vectors very fast for the first a few sweeps and 
usually make good initial updates for themselves. 

\subsubsection{Stopping Criterion}
\label{sec:stopping}

The ALS-SVD and MALS-SVD algorithms can terminate if 
the relative residual decreases below a prescribed 
tolerance parameter $\epsilon$: 
	\begin{equation}\label{eqn:rel_resid}
	r = \frac{ \|\BF{A}^\RM{T}\BF{U}-\BF{V\Sigma}\|_\RM{F} }
	{ \|\BF{\Sigma}\|_\RM{F} } < \epsilon.
	\end{equation}
The computational cost for computing $r$ is proportional 
to the order $N$ \cite{Ose2011}. However, since the 
product $\BF{A}^\RM{T}\BF{U}$ increases the TT-ranks to 
the multiplications, $\{R^A_nR^U_n\}$, we perform 
the truncation method, called TT-rounding, proposed in \cite{Ose2011}. 


\subsubsection{Truncation Parameter}

At each iteration, the $\delta$-truncated SVD is used to determine 
close to optimal TT-ranks and simultaneously to orthogonalize 
TT-cores. The $\delta$ value significantly affects the convergence 
rate of the ALS-SVD and MALS-SVD algorithms. If $\delta$ is small, 
then the algorithms usually converge fast within one or two full 
sweeps, but the TT-ranks may also grow largely, which causes high 
computational costs. In \cite{Holtz2011},
it was reported that a MALS-based algorithm often resulted in 
a rapidly increasing TT-ranks during iteration process. 
On the other hand, if $\delta$ is large, TT-ranks grow slowly, but 
the algorithms may not converge to the desired accuracy but converge 
to a local minimum. 

In this paper, we initially set the $\delta$ value to 
$\delta_0=\epsilon/\sqrt{N-1}$ in numerical simulations. 
In \cite{Ose2011}, it was shown that the $\delta_0$ 
yields a guaranteed accuracy in approximation and truncation 
algorithms. In the numerical simulations, we observed that 
the TT-based algorithms including the proposed 
algorithms converged to the desired accuracy $\epsilon$
within 1 to 3 full sweeps in most of the cases. 

In the case that the algorithm could not converge to the 
desired tolerance $\epsilon$ after the $N_{sweep}$ number 
of full sweeps, we restart the algorithm with different random 
initializations. In this case, the $\delta$ value may not be changed or 
decreased to, e.g., $0.1\delta$. If $K$ is not small, the algorithms 
usually converge in 1 or 2 full sweeps, so a small $N_{sweep}$ value 
is often sufficient. But if $K$ is small, e.g., $K=2$, then the rank 
growth in each sweep is relatively slow and a more number of sweeps 
may be necessary. Moreover, in order to reduce the computational 
complexity, we used a rather large $\delta$ value at the first half sweep, 
for instance, $100\delta$. In this way, we can speed up the 
computation while not harming the convergence.


%
%

\section{Numerical simulations}

In numerical simulations, we computed the $K$ dominant singular values and 
correponding singular vectors 
of several different types of very large-scale matrices. 
We compared the following SVD methods including 
two standard methods, LOBPCG and SVDS. 

\begin{enumerate}
\item LOBPCG: The LOBPCG \cite{Kny2001} method 
can compute the $K$ largest or smallest eigenvalues 
of Hermitian matrices. We applied a MATLAB version of 
LOBPCG to the matrix 
$\begin{bmatrix}\BF{0}&\BF{A}\\\BF{A}^\RM{T}&\BF{0}\end{bmatrix}$
to compute its $K$ largest eigenvalues 
$\BF{\Lambda}=\BF{\Sigma}=\text{diag}(\sigma_1,\ldots,\sigma_K)$
and the corresponding eigenvectors $\BF{W}=2^{-1/2}
\begin{bmatrix}\BF{U}^\RM{T}&\BF{V}^\RM{T}\end{bmatrix}$. 

\item SVDS: The MATLAB function SVDS computes a few singular 
values and vectors of matrices by using the MATLAB function 
EIGS, which applies the Fortran package ARPACK \cite{arpack97}. 
We applied the function EIGS directly to the matrix 
$\begin{bmatrix} \BF{0} & \BF{A} \\ \BF{A}^\RM{T} & \BF{0}
\end{bmatrix}$, so that the matrix-by-vector products are performed 
more efficiently than in SVDS. 
We obtained its $2K$ eigenvalues/vectors of largest magnitudes, which are plus/minus 
largest singular values of $\BF{A}$: $\pm \sigma_1,\ldots,\pm \sigma_K$. And then 
we selected only the $K$ largest singular values among them.

\item ALS-SVD, MALS-SVD: The ALS-SVD and MALS-SVD 
algorithms are described in Algorithms \ref{Alg:ALS} and
\ref{Alg:MALS}. For the local optimization at each iteration, 
we applied the MATLAB function EIGS with the projected matrix-by-vector 
product described in Section \ref{sec:3_3}. We computed $2K$
eigenvalues/vectors of largest magnitudes, $\pm \sigma_1,\ldots,\pm \sigma_K$, 
at each local optimization as described in the SVDS method above. 

\item ALS-EIG, MALS-EIG: The ALS and MALS schemes are implemented 
for computing the $K$ largest eigenvalues of the Hermitian matrix 
$\BF{A}^\RM{T}\BF{A}$ by maximizing the block Rayleigh quotient 
\cite{Dol2013b}. We applied the MATLAB function EIGS for optimization 
at each iteration. After computing the $K$ eigenvalues 
of largest magnitudes $\BF{\Lambda}=\BF{\Sigma}^2$ and eigenvectors 
$\BF{V}$ of $\BF{A}^\RM{T}\BF{A}$,
the iteration stops when 
	\begin{equation}
	\|\BF{A}^\RM{T}\BF{AV\Sigma}^{\dagger} - 
	\BF{V\Sigma}\|^2 < \epsilon^2 \|\BF{\Sigma}\|^2 , 
	\end{equation}
where $\BF{\Sigma}^{\dagger}$ is the pseudo-inverse. 
The left singular vectors are 
computed by $\BF{U} = \BF{AV\Sigma}^{-1}$ if $\BF{\Sigma}$ is 
invertible. If $\BF{\Sigma}$ is not invertible, then eigenvalue 
decomposition of $\BF{AA}^\RM{T}$ is computed to obtain $\BF{U}$. 
In this way, the relative residual can be controlled below $\epsilon$ as 
	\begin{equation}
	\begin{split}
	\|\BF{A}^\RM{T}\BF{U}-\BF{V\Sigma}\|^2 
	&= \|\BF{A}^\RM{T}\BF{AV\Sigma}^{-1} 
	-\BF{V\Sigma}\|^2 \leq \epsilon^2 \|\BF{\Sigma}\|^2 . 
	\end{split}
	\end{equation}
\end{enumerate}

We note that the computation of $\BF{A}^\RM{T}\BF{A}$ and 
$\BF{U}=\BF{AV\Sigma}^{-1}$ is followed by the truncation algorithm 
of \cite{Ose2011} to reduce the TT-ranks. 
Especially, the computational cost for the truncation of 
$\BF{A}^\RM{T}\BF{A}$ is $\CL{O}(NIR_A^6)$, which is 
quite large compared to the computational costs of the ALS-SVD and 
MALS-SVD in Table \ref{table1:complexity}.
Moreover, for the two standard methods, LOBPCG and SVDS, 
the matrix $\BF{A}$ is in full matrix format, whereas for 
the block TT-based SVD methods, the matrix 
$\BF{A}$ is in matrix TT format. 
In the simulations, we stopped running the two standard methods 
when the size of the matrix grows larger than $2^{13}\times 2^{13}$, 
because not only the computational time is high, but also 
the storage cost exceed the desktop memory capacity. 

We implemented our code in MATLAB. The simulations were 
performed on a desktop computer with Intel Core i7 X 980 CPU
at 3.33 GHz and 24GB of memory running Windows 7 Professional
and MATLAB R2007b. In the simulations, we performed 30 repeated 
experiments independently and averaged the results. 


\subsection{Random Matrix with Prescribed Singular Values}
\label{sec:4-1random} 

In order to measure the accuracy of computed singular values, 
we built matrices of rank 25 in matrix TT format by 
	\begin{equation}\label{eqn:matrix05USV}
	\BF{A} = \BF{U}_0\BF{\Sigma}_0\BF{V}_0^\RM{T} \in\BB{R}^{2^N\times 2^N}, 
	\end{equation}
where $\BF{U}_0\in\BB{R}^{2^N\times 25}$ and 
$\BF{V}_0\in\BB{R}^{2^N\times 25}$ are left and right singular vectors 
in block-$N$ TT format, where each of TT-cores are generated by 
standard normal distribution and then orthogonalized to have orthonormal 
column vectors in $\BF{U}_0$ and $\BF{V}_0$. 
The singular values $\BF{\Sigma}_0=\text{diag}(
\sigma^0_1,\sigma^0_2,\ldots,\sigma^0_{25})$ 
are given by 
	\begin{equation}
	\sigma^0_k = \beta^{k-1},\ k=1,2,\ldots,25. 
	\end{equation}
The $\beta$ takes values from $\{0.2,0.3,0.4,0.5,0.6\}$.
Figure \ref{Fig:betaUSV} illustrates the 25 singular values 
for each $\beta$ value. The block TT-ranks of 
$\BF{U}_0$ and $\BF{V}_0$ are set at the fixed value 
$5$. 
The TT-cores of $\BF{A}$ in \eqref{eqn:matrix05USV} 
were calculated based on 
the TT-cores of $\BF{U}_0$ and $\BF{V}_0$ as 
	\begin{equation}
	\begin{split}
	\BF{A}^{(n)}_{:,i_n,j_n,:}
	&= \BF{U}^{(n)}_{:,i_n,:} \otimes 
	\BF{V}^{(n)}_{:,j_n,:}
	\in\BB{R}^{R^U_{n-1}R^V_{n-1}\times R^U_nR^V_n}, 
	\quad 
	n=1,2,\ldots,N-1, \\
	\BF{A}^{(N)}_{:,i_N,j_N,:}
	&= \sum_{k=1}^{25}
	\BF{U}^{(N)}_{:,k,i_N,:} \otimes 
	\BF{V}^{(N)}_{:,k,j_N,:}
	\in\BB{R}^{R^U_{N-1}R^V_{N-1}\times R^U_{N}R^V_{N}}, 
	\end{split}
	\end{equation}
for $	i_n=1,\ldots,I_n, j_n=1,\ldots,J_n.$
We set $\epsilon=10^{-8}$ and $K=10$. 
The relative error for the estimated singular values $\BF{\Sigma}$ is 
calculated by 
	\begin{equation}
	\frac { \|\BF{\Sigma} - \BF{\Sigma}_0\|_\RM{F} }
	{ \|\BF{\Sigma}_0\|_\RM{F} }. 
	\end{equation}

\begin{figure}
\centering
\includegraphics[width=7.5cm]{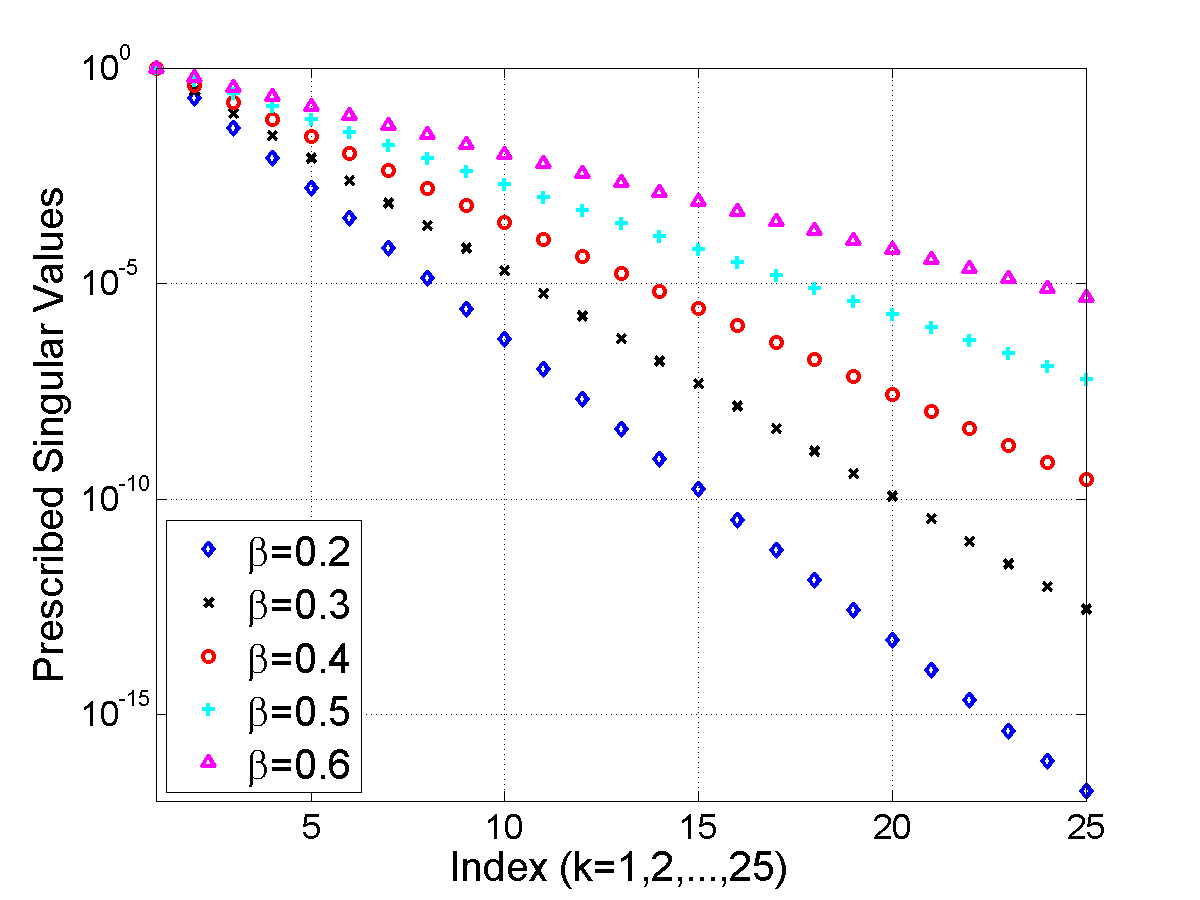}
\caption{\label{Fig:betaUSV}Singular values of the random matrices
for Section \ref{sec:4-1random}}
\end{figure}

Figure \ref{Fig:USV-convergence} shows the convergence of the proposed 
ALS-SVD and MALS-SVD algorithms for different $K$ values, $K=2,10$,
fixed dimension $N=50$, and truncation parameter $\delta=10^{-8}/\sqrt{N-1}$. 
In Figures \ref{Fig:USV-convergence}(a) and (b), the value $K=2$ is 
relatively small. In this case, the maximum of the TT-ranks of the 
right singular vectors $\BF{U}$ increases slowly in Figure \ref{Fig:USV-convergence}(b). 
The convergence of the MALS-SVD is faster than the ALS-SVD 
in Figure \ref{Fig:USV-convergence}(a), because its TT-ranks increases
faster than the ALS-SVD. Note that the relatively fast convergence of the MALS-SVD 
for small $K$ values were explained in the inequality \eqref{eqn:boundTTrank}. 
On the other hand, Figures \ref{Fig:USV-convergence}(c) and (d) 
show the results when the $K$ value is relatively large, i.e., $K=10$. 
In Figure \ref{Fig:USV-convergence}(d), the TT-ranks of $\BF{U}$
increase fastly in a few initial iterations, and then decrease to the 
optimal TT-ranks a few iterations before the convergence. In this case, 
we can see that both the ALS-SVD and MALS-SVD converge fastly in 
Figure \ref{Fig:USV-convergence}(c). Note that the fast convergence
does not imply small computational costs because the large 
TT-ranks will slow the speed of computation at each iteration. 

\begin{figure}
\centering
\begin{tabular}{cc}
\includegraphics[width=7.5cm]{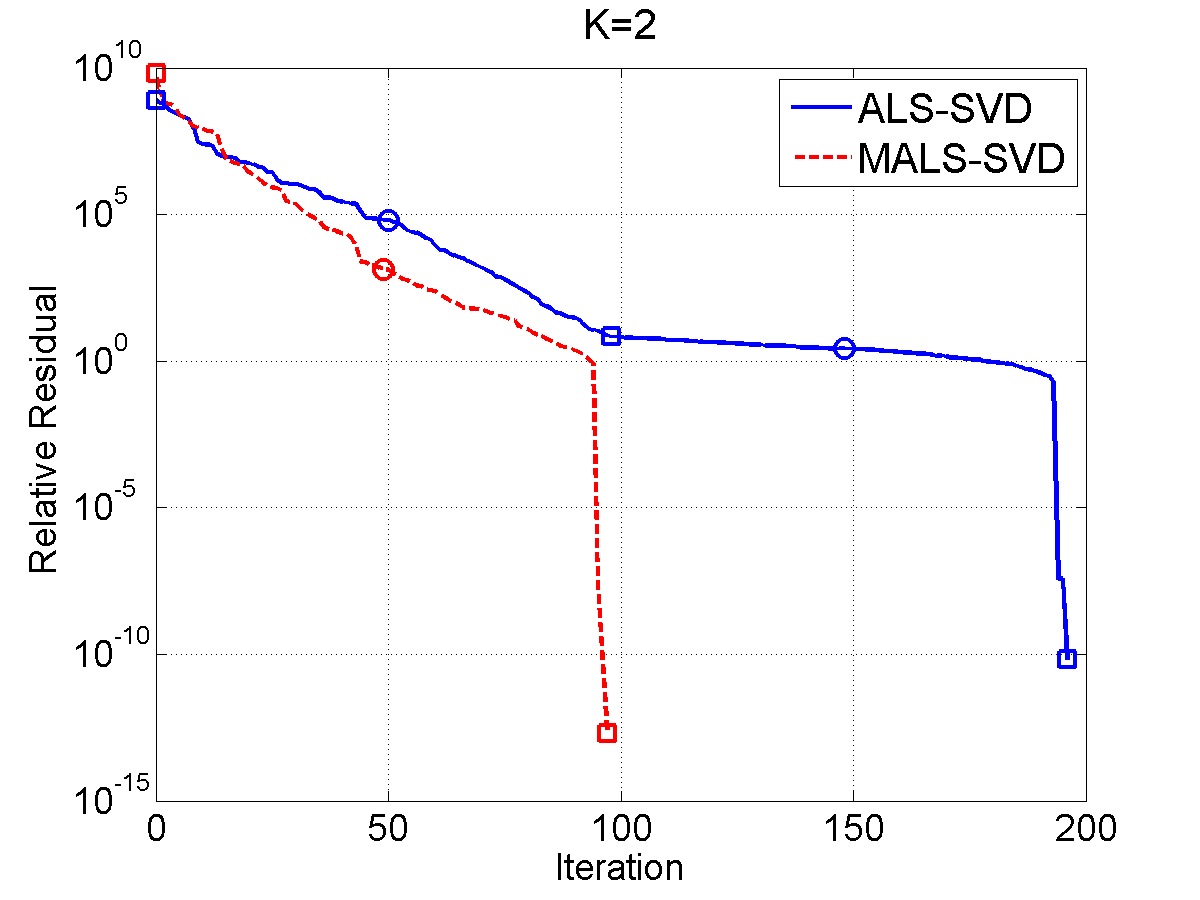} & 
\includegraphics[width=7.5cm]{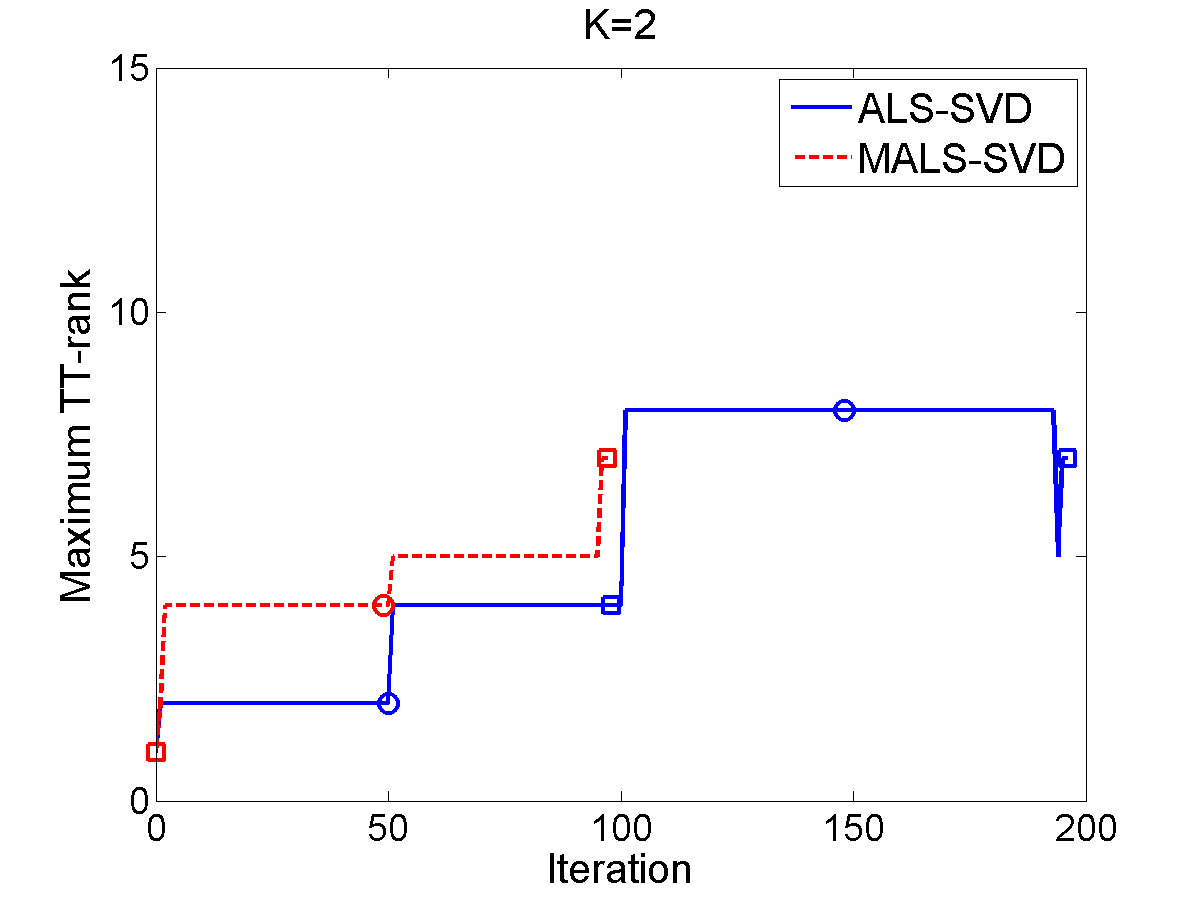} \\
(a) & (b) \\
\includegraphics[width=7.5cm]{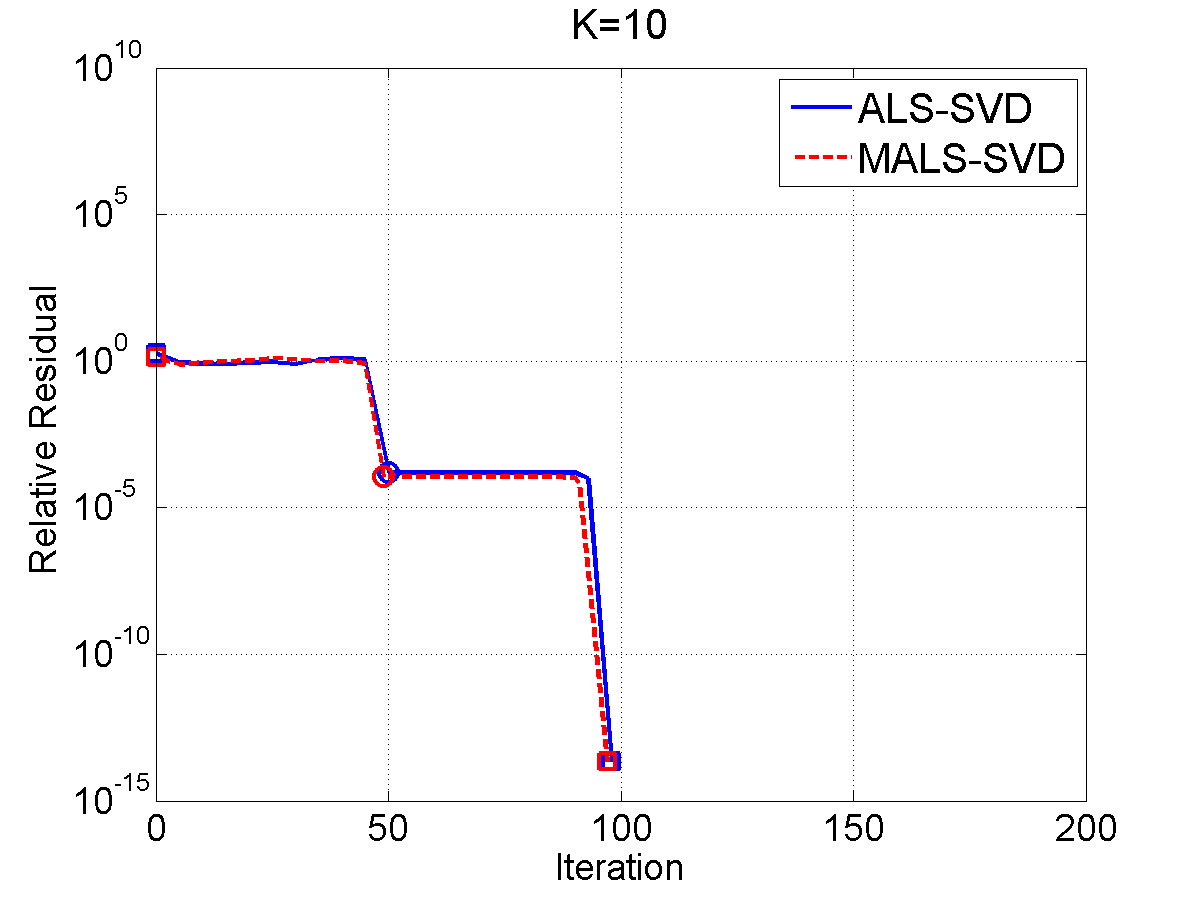} & 
\includegraphics[width=7.5cm]{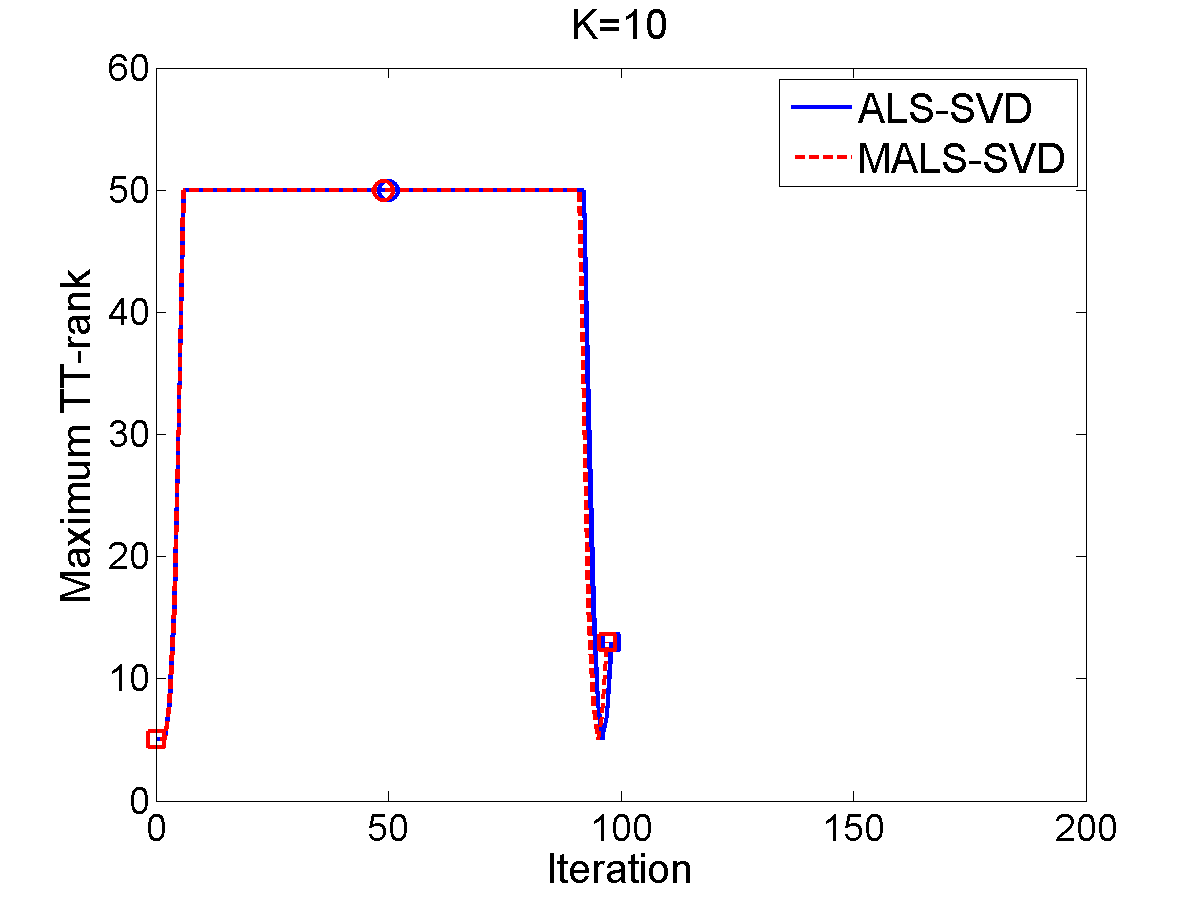} \\
(c) & (d) 
\end{tabular}
\caption{\label{Fig:USV-convergence}
Convergence of the ALS-SVD and MALS-SVD algorithms 
for the dimension $N=50$ for Section \ref{sec:4-1random}. 
Top panels (a) and (b) are the results for $K=2$ 
and the bottom panels (c) and (d) are the results for $K=10$. 
Circles represent the end of the right-to-left half sweep, and the 
squares represent each of the full-sweeps. }
\end{figure}

In Figure \ref{Fig:0503USV-time}, we can see that the computational costs
of the TT-based algorithms grow only logarithmically with the matrix size, 
whereas the times for the standard SVD algorithms grow exponentially with $N$.  
Among the TT-based algorithms, the ALS-SVD and MALS-SVD show the smallest 
computational costs. The ALS-EIG and MALS-EIG have higher computational costs
because the product $\BF{A}^\RM{T}\BF{A}$ increases its matrix TT-ranks and the 
subsequent truncation step results in high computational costs.
The LOBPCG and SVDS show a fast rate of increase in the computational cost. 
The LOBPCG and SVDS were stopped running for larger matrix sizes 
than $2^{13}\times 2^{13}$ because of the computational cost and 
desktop computer memory capacity. 
The black dotted line shows a predicted computational time for the LOBPCG. 

Moreover, the ALS-based methods are faster than the MALS-based methods 
because the MALS-based methods solve larger optimization problems at 
each iteration over the merged TT-cores. However, the MALS-based methods 
can determine TT-ranks during iteration even if $K=1$, and the rate of convergence
per iteration is faster for small $K$ values as shown in Figure \ref{Fig:USV-convergence}. 

\begin{figure}
\centering
\includegraphics[width=7.5cm]{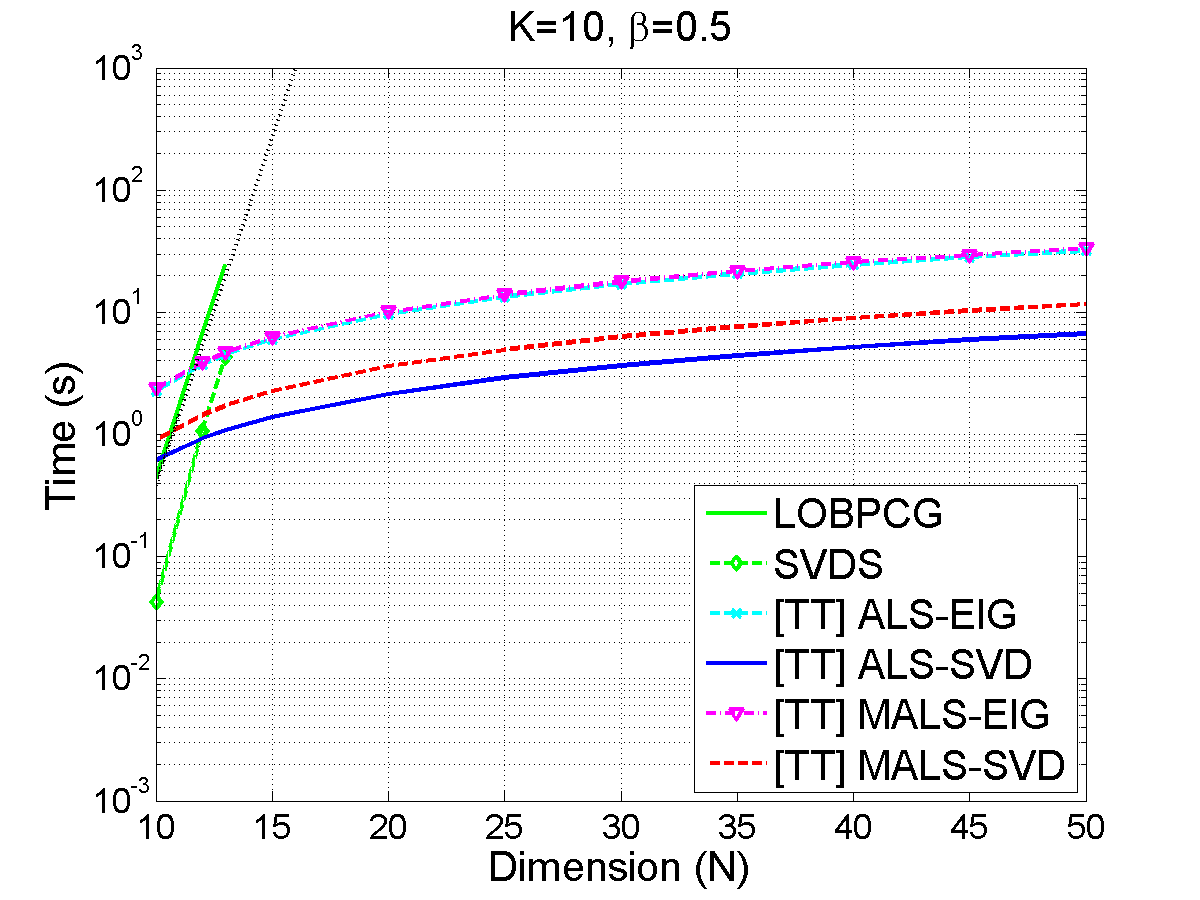}  
\caption{\label{Fig:0503USV-time}
Performances for $2^N\times 2^N$ random matrices with $10\leq N \leq 50$ 
and fixed $\beta=0.5$}
\end{figure}

Figure \ref{Fig:0504USV-errorsv} shows the performances of the 
four TT-based algorithms for various $\beta$ values 
for the random matrices $\BF{A}\in\BB{R}^{2^N\times 2^N}$ with $N=50$. 
In Figure \ref{Fig:0504USV-errorsv}(a), the ALS-SVD and MALS-SVD show the 
smallest computational times over all the $\beta$ values. 
In Figures \ref{Fig:0504USV-errorsv}(b) and (c), 
we can see that the ALS-SVD and MALS-SVD 
accurately estimate the block TT-ranks and 
the $K=10$ dominant singular values. 
On the other hand, the ALS-EIG and MALS-EIG 
estimate the block TT-ranks and the singular values slightly less accurately, 
especially for small $\beta$ values. 
We note that the ALS-EIG and MALS-EIG take square roots 
on the obtained eigenvalues to compute the singular values. 

\begin{figure}
\centering
\begin{tabular}{cc}
\includegraphics[width=7.5cm]{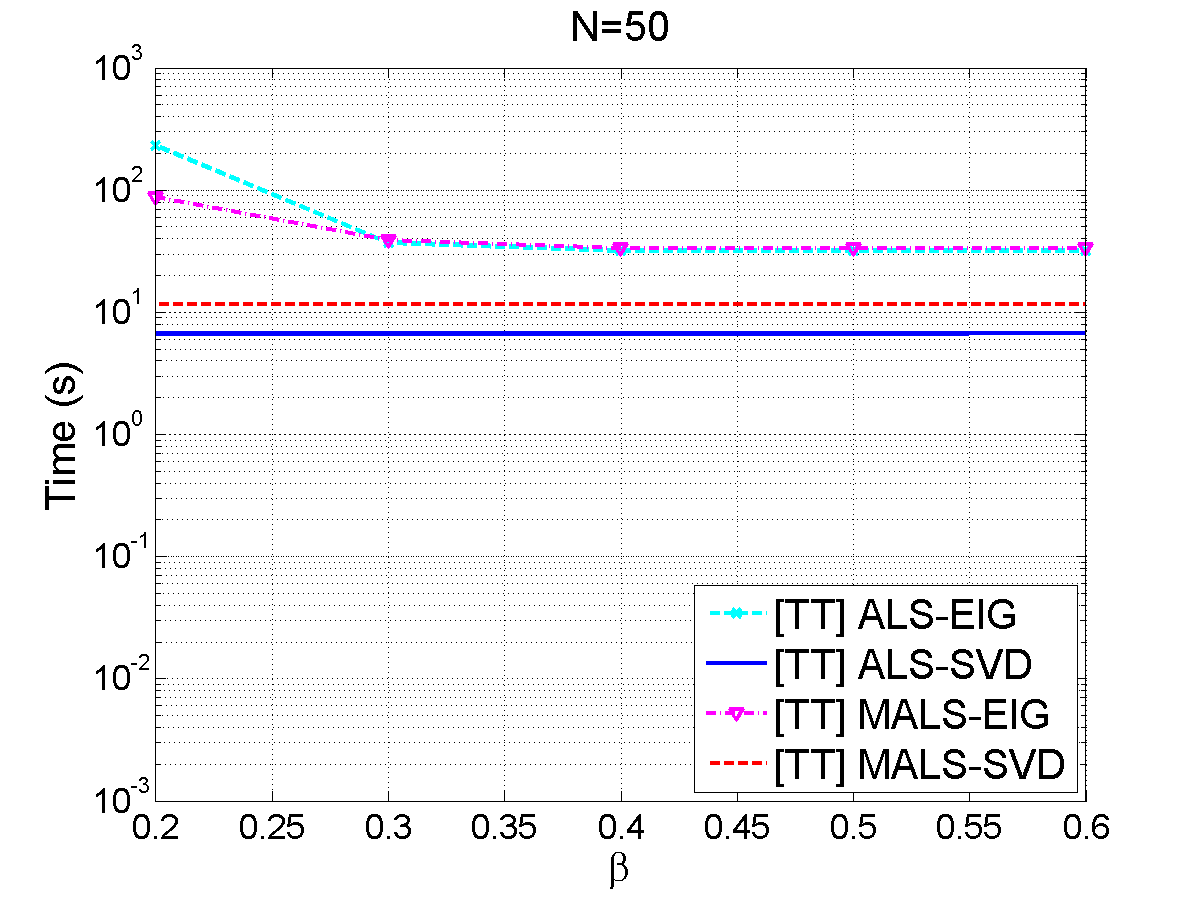}  & 
\includegraphics[width=7.5cm]{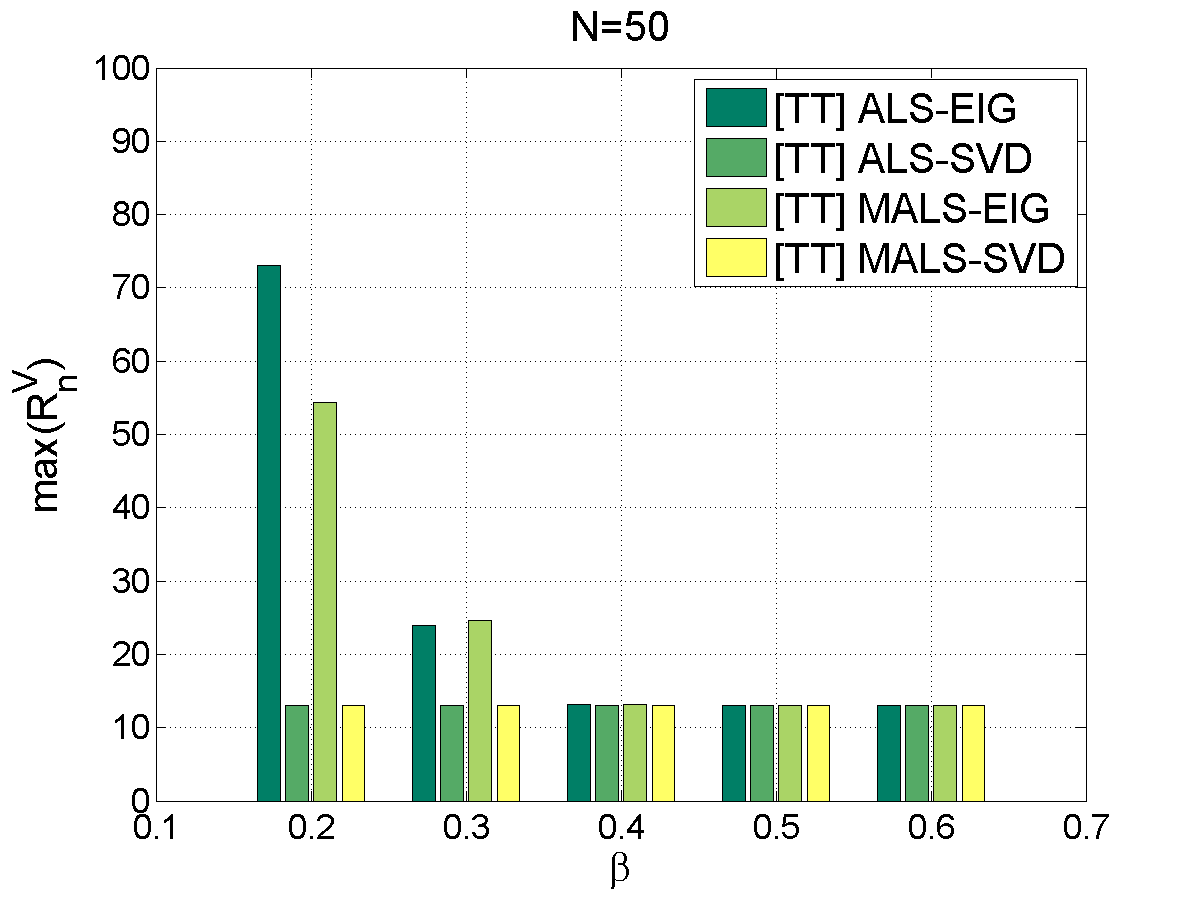}  \\
(a) & (b) \\
\multicolumn{2}{c}{
\includegraphics[width=7.5cm]{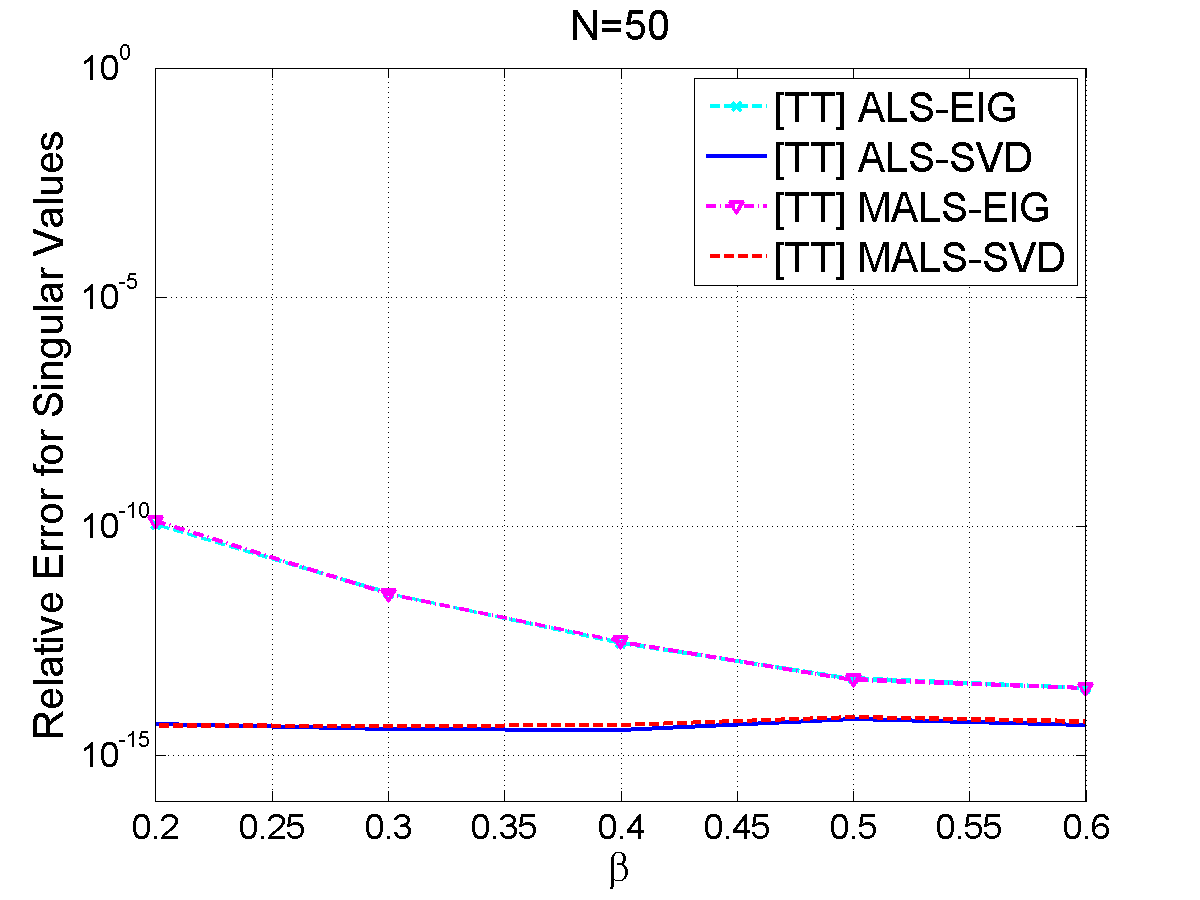}}\\
\multicolumn{2}{c}{(c)}
\end{tabular}
\caption{\label{Fig:0504USV-errorsv}
Performances for $2^N\times 2^N$ random matrices with prescribed singular values 
for $N=50$ and various $\beta$ values. (a) The computational time, 
(b) maximum block TT-rank of right singular vectors, and (c) relative error for singular values. 
 }
\end{figure}

\subsection{A Submatrix of Hilbert Matrix}
\label{sec:4-3hilbert} 

The Hilbert matrix $\BF{H}\in\BB{R}^{P\times P}$ is a symmetric matrix 
with entries $h_{i,j} = (i+j-1)^{-1},i,j=1,2,\ldots,P$. It is known that 
the eigenvalues of the Hilbert matrix decay to zero very fast. 
In this simulation, we consider a rectangular submatrix 
$\BF{A}\in\BB{R}^{2^N\times 2^{N-1}}$ of the Hilbert matrix defined by 
	\begin{equation}
	\BF{A} = \BF{H}(:,1:2^{N-1})
	\in\BB{R}^{2^N\times 2^{N-1}}
	\end{equation}
in MATLAB notation, in order to apply the SVD algorithms to 
the non-symmetric matrix $\BF{A}$. 
The matrix TT representation of $\BF{A}$ was computed  
based on the explicit TT representation of Hankel matrices, 
which is described in Appendix \ref{sec:app2} in detail.
For this purpose, we applied the cross approximation algorithm 
FUNCRS2 in TT-Toolbox \cite{Ose2011b}
to transform the vector 
$[1,1,2^{-1},3^{-1},\ldots,(2^{N+1}-1)^{-1}]^\RM{T}
\in\BB{R}^{2^{N+1}}$ 
to a vector TT format with the
relative approximation error of $10^{-8}$. 
Then we used the explicit TT representation of Hankel 
matrices to convert the vector into the Hilbert matrix in 
matrix TT format. Finally, we could find that the 
maximum value of the matrix TT-ranks, $\text{max}(R^A_n)$, 
are between 14 and 22 over $10\leq N\leq 50$.

For comparison of performances of the SVD algorithms, we 
set $\epsilon=10^{-3}$ and $K=10$.  
In Figure \ref{Fig:performance_hilbert}(a), 
the computational costs of the TT-based algorithms grow
logarithmically with the matrix size. The ALS-SVD and MALS-SVD 
have the least computational costs. 
In Figure \ref{Fig:performance_hilbert}(b),
we can see that the maximum block TT-ranks, $\text{max}(R^V_n)$, 
are relatively large for $N=30,40,50$, which may have 
affected the computational costs 
in Figure \ref{Fig:performance_hilbert}(a).

\begin{figure}
\centering
\begin{tabular}{cc}
\includegraphics[width=7.5cm]{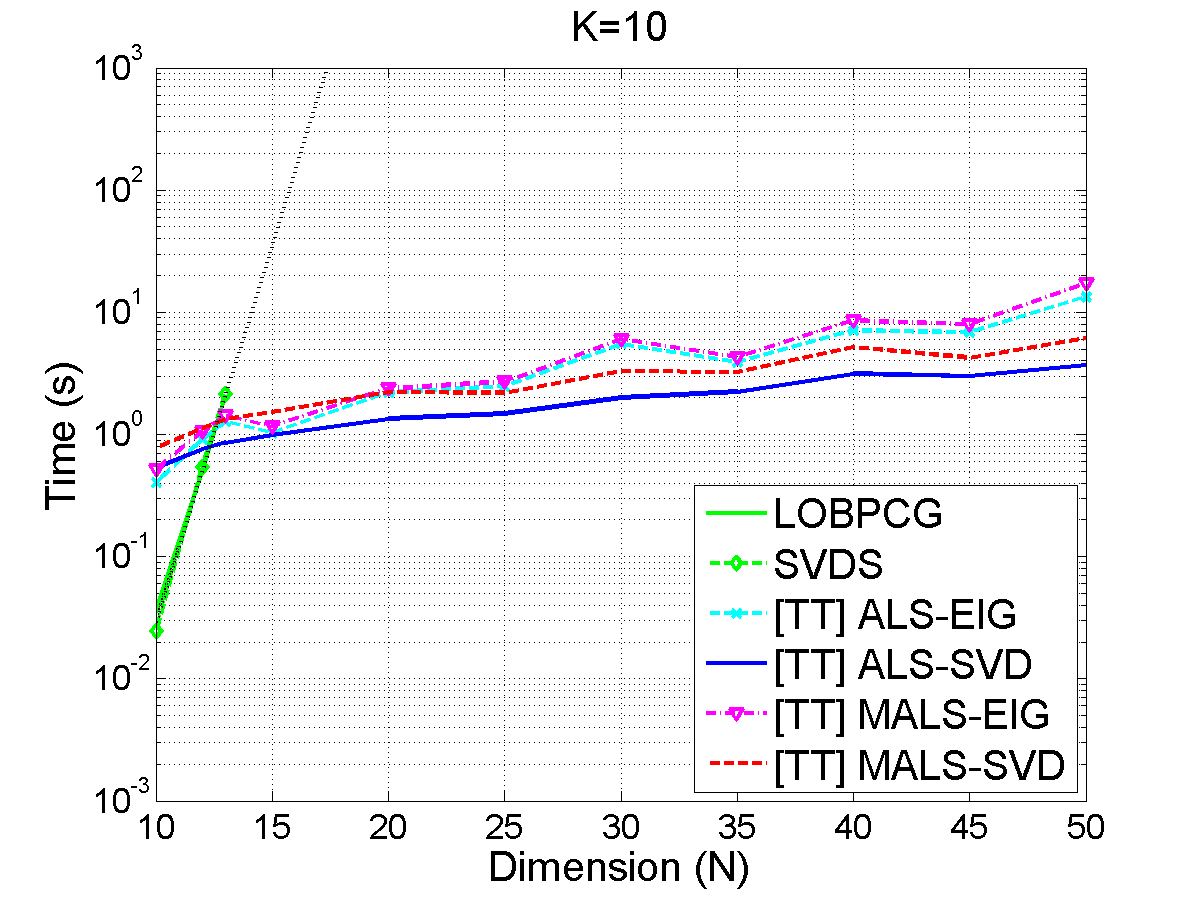} & 
\includegraphics[width=7.5cm]{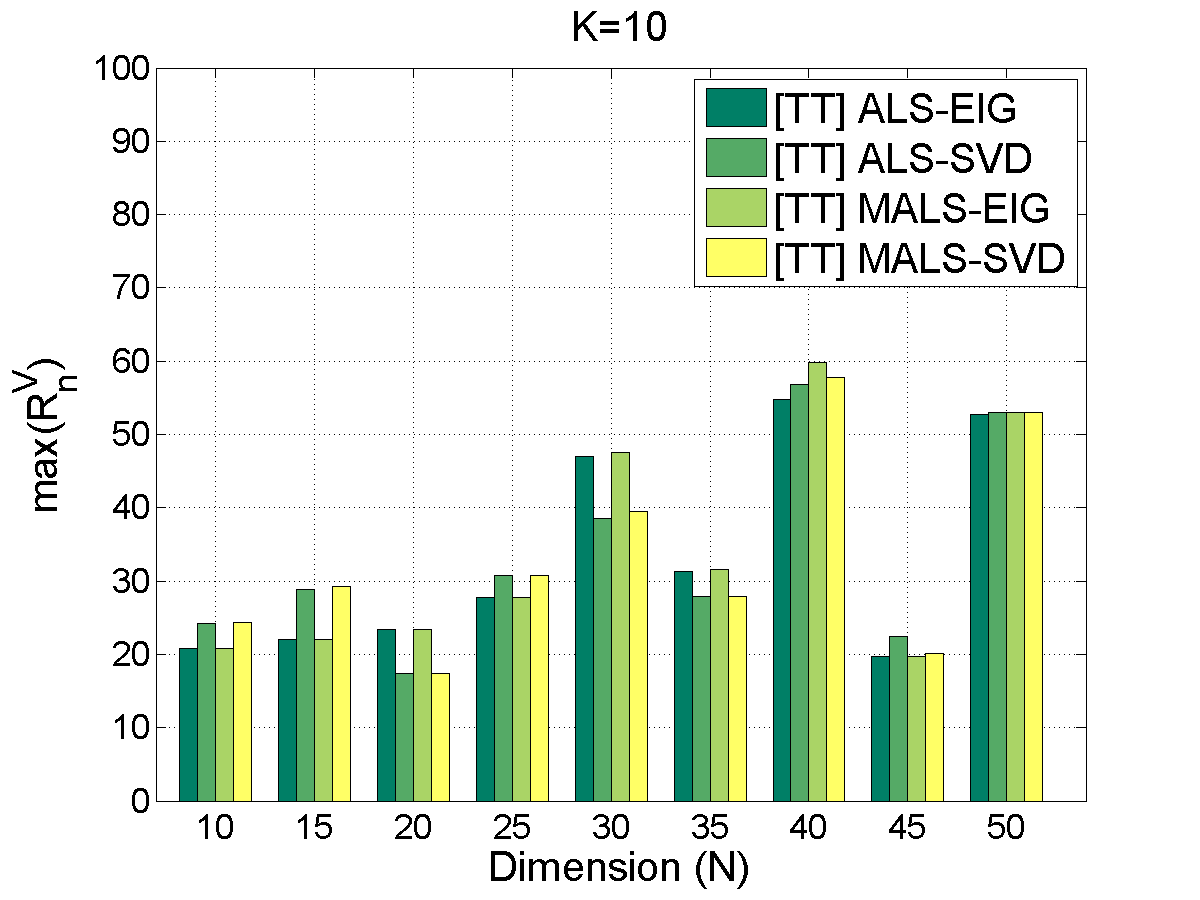} \\
(a)&(b) 
\end{tabular}
\caption{\label{Fig:performance_hilbert}
Performances for the $2^N\times 2^{N-1}$ submatrices of Hilbert matrices
with $10\leq N\leq 50$ and $K=10$.  
(a) Computational cost 
and (b) maximum block TT-rank of the right 
singular vectors. }
\end{figure}

\subsection{Random Tridiagonal Matrix}

A tridiagonal matrix is a banded matrix whose nonzero entries are 
only on the main diagonal, the first diagonal above the main diagonal, 
and the first diagonal below the main diagonal. 
The matrix TT representation of a tridiagonal matrix is 
described in Appendix \ref{sec:app3}. 
We randomly generated three vector TT tensors  
$\BF{a},\BF{b},\BF{c}$ with TT-cores 
drawn from the standard normal distribution 
and TT-ranks $R_n=R=5,1\leq n\leq N-1$. Then,
each TT-cores are orthogonalized to yield 
$\|\BF{a}\|=\|\BF{b}\|=\|\BF{c}\|=1$. 
We built the $2^N\times 2^N$ tridiagonal matrix $\BF{A}$
whose sub, main, and super diagonals are $\BF{a},\BF{b},$ and $\BF{c}$. 
The matrix TT-ranks of $\BF{A}$ are bounded by $5R$,
which are largely reduced after truncation to around 17. 

For performance evaluation, we set $\epsilon=10^{-8}$ and  
$K=10$. In this simulation, all the TT-based SVD algorithms converged 
within 3 full sweeps. 

Figure \ref{Fig:performance_tridiag}(a) shows that the 
computational costs for the TT-based SVD algorithms 
are growing logarithmically with the matrix size over $10\leq N\leq 50$. 
The ALS-SVD and MALS-SVD have the smallest computational costs. 
The ALS-EIG and MALS-EIG have relatively high computational costs 
because the matrix TT-ranks are relatively large in this case so the 
truncation of $\BF{A}^\RM{T}\BF{A}$ was computationally costly. 
Figure \ref{Fig:performance_tridiag}(b) shows that the maximum value
of the block TT-ranks, $\text{max}(R^V_n)$, are bounded by 20 
and slowly decreasing as $N$ increases. 
We note that the diagonal entries of the matrices were randomly 
generated and the 10 dominant singular values were close to each other, 
similarly as the identity matrix. 
We conclude that the TT-based SVD algorithms can accurately 
compute several dominant singular values and singular vectors
even in the case that the singular values are almost identical. 

\begin{figure}
\centering
\begin{tabular}{cc}
\includegraphics[width=7.5cm]{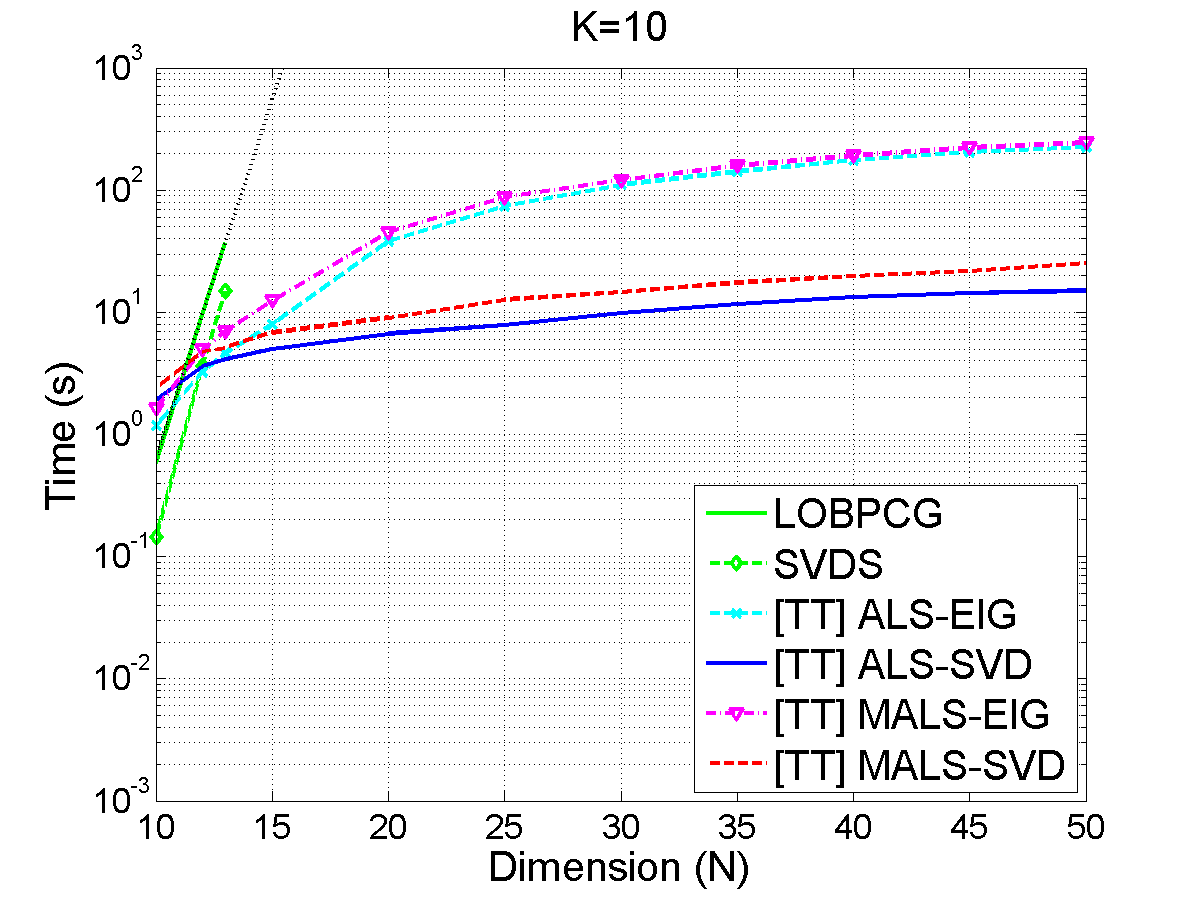} & 
\includegraphics[width=7.5cm]{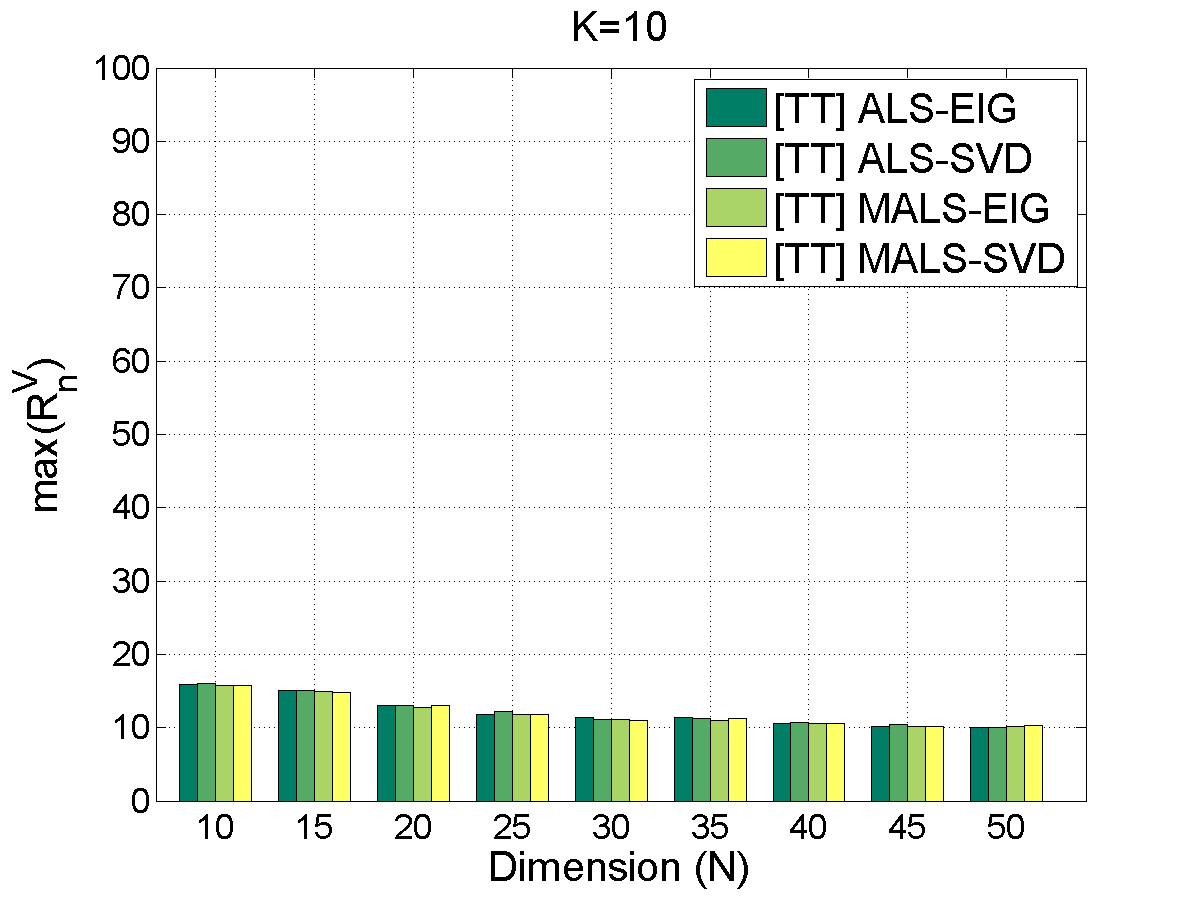} \\
(a)&(b) 
\end{tabular}
\caption{\label{Fig:performance_tridiag}
Performances for the $2^N\times 2^{N}$ random tridiagonal matrices 
with $K=10$ and $10\leq N\leq 50$.  
(a) Computational cost and (b) maximum block TT-rank of the right 
singular vectors. 
}
\end{figure}


\subsection{Random Toeplitz Matrix}
\label{sec:4-2toeplitz} 

Toeplitz matrix is a square matrix that has constant diagonals. 
An explicit matrix TT representation of a Toeplitz matrix is 
described in \cite{Kaz2013}. 
See Appendix \ref{sec:app2} for more detail. 
We generated a vector 
$\BF{x} = [x_1,\ldots,x_{2^{N+1}}]^\RM{T}$
in vector TT format with its TT-cores drawn from the 
standard normal distribution and fixed TT-ranks $R_n=R=5$. 
Then, we converted $\BF{x}$ into a Toeplitz matrix 
$\BF{A}\in\BB{R}^{2^N\times 2^N}$
in matrix TT format with entries 
	\begin{equation}
	a_{i,j} = x_{2^N+i-j}, \quad i,j=1,2,\ldots,2^N. 
	\end{equation}
The matrix TT-ranks of $\BF{A}$ are bounded by 
twice the TT-ranks of $\BF{x}$, i.e., $2R$ \cite{Kaz2013}. 

Since the matrix $\BF{A}$ is generated randomly, we 
cannot expect that the TT-ranks of the left and right singular 
vectors are bounded over $N$. Instead, we fixed the 
maximum of the TT-ranks, $R_{max}$, and compared the 
computational times and relative residuals of the algorithms. 
We set $K=10$, and $N_{sweep}=2$.  

Figures \ref{Fig:0303randomToepl}(a) and (b) show 
the performances of the SVD algorithms 
for $R_{max}=15$ and $10\leq N\leq 30$. 
In Figure \ref{Fig:0303randomToepl}(a), 
we can see that the computational costs of the 
TT-based algorithms grow logarithmically with the matrix size
because the matrix TT-ranks and the block TT-ranks are bounded
by $2R=10$ and $R_{max}=15$, respectively. 
In Figure \ref{Fig:0303randomToepl}(b), 
the relative residual values remain almost constantly around $0.15$. 
Figures \ref{Fig:0303randomToepl}(c) and (d)
show the computational costs and relative residuals 
for various $10\leq R_{max} \leq 30$ and fixed $N=30$. 
We can see that the computational cost for the ALS-SVD is the 
smallest and growing slowly as $R_{max}$ increases.

\begin{figure}
\centering
\begin{tabular}{cc}
\includegraphics[width=7.5cm]{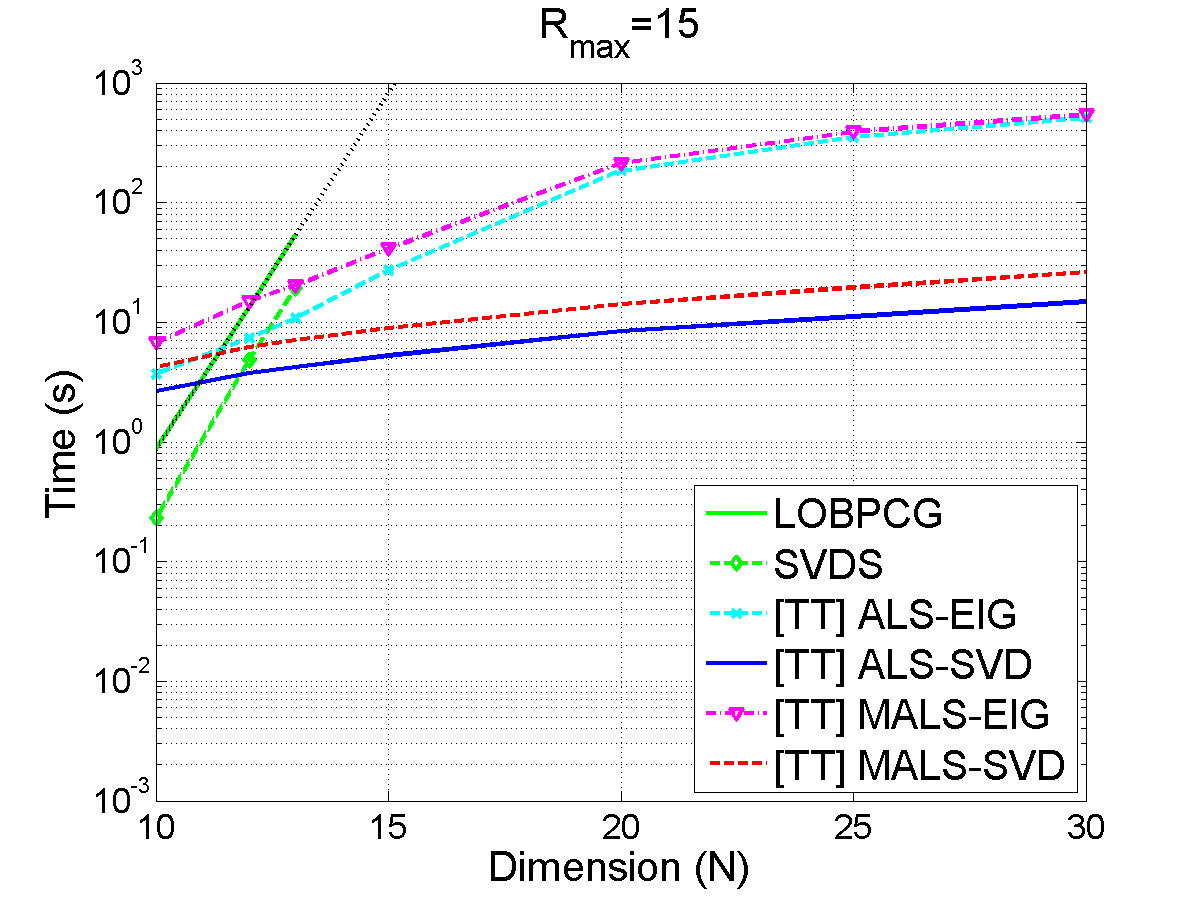} & 
\includegraphics[width=7.5cm]{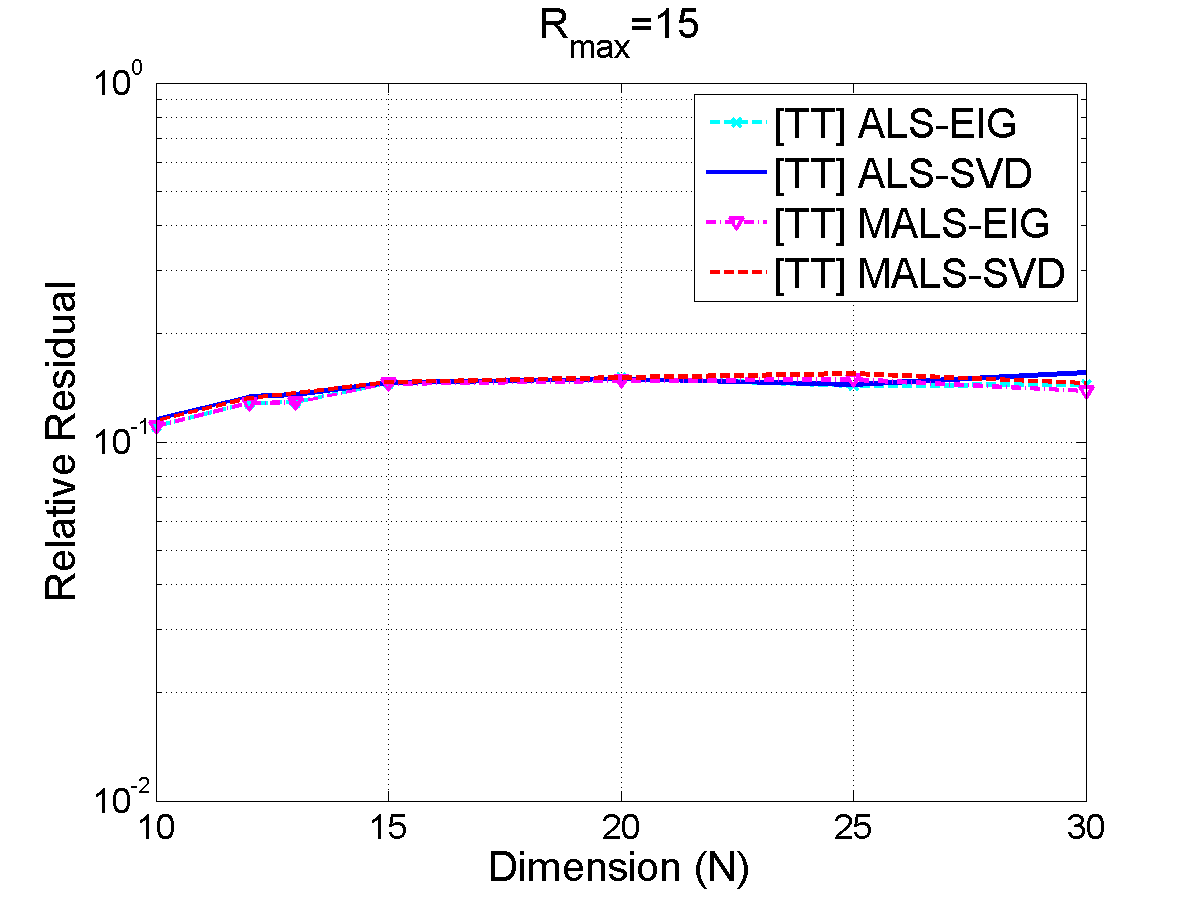}\\
(a)&(b)\\
\includegraphics[width=7.5cm]{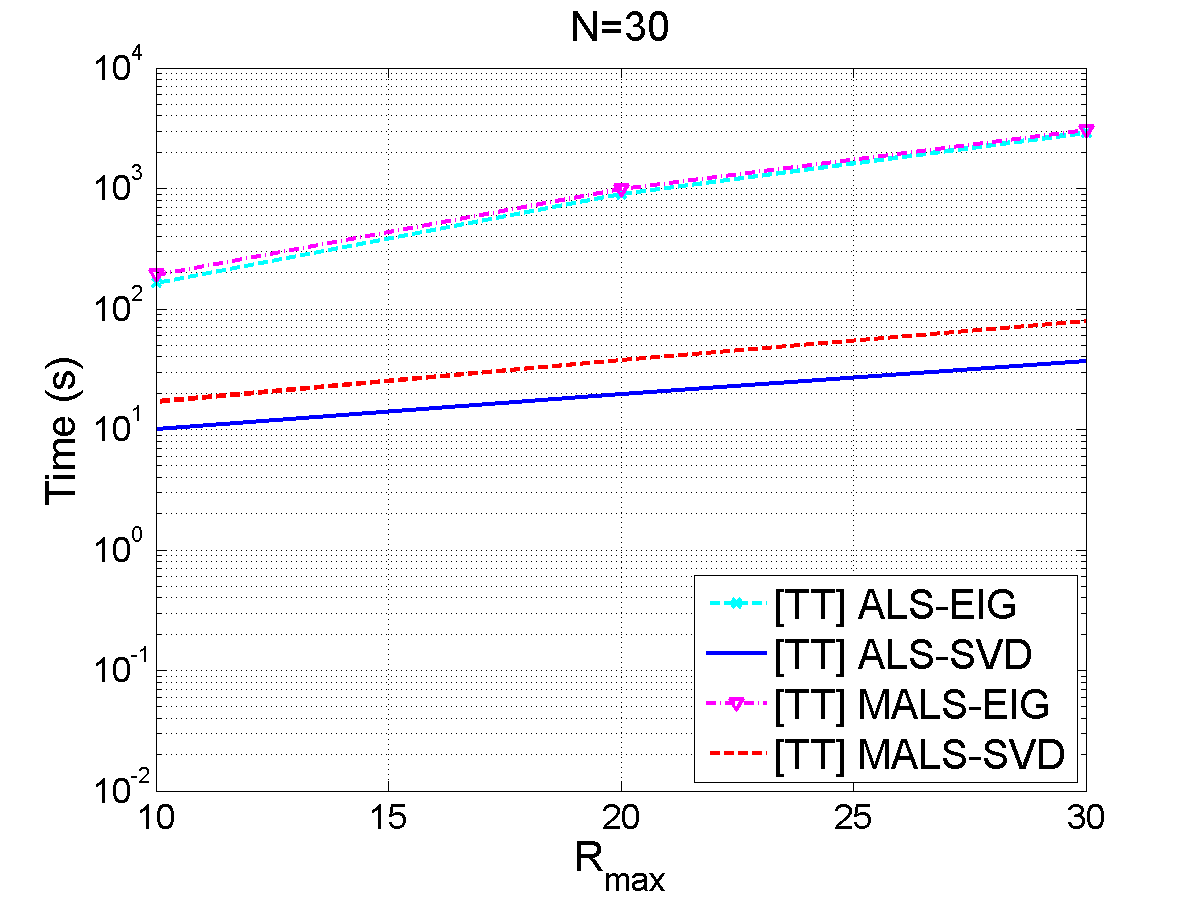} & 
\includegraphics[width=7.5cm]{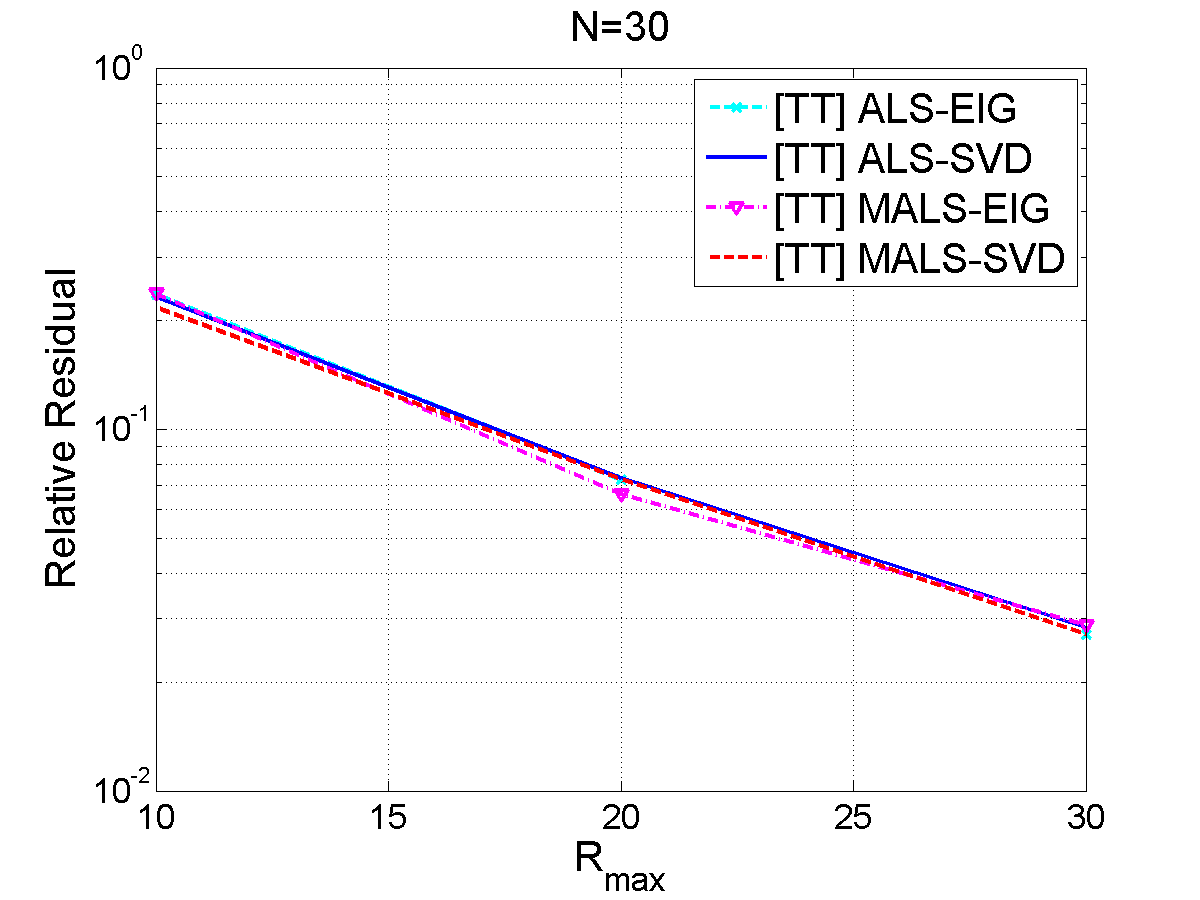}\\
(c)&(d)
\end{tabular}
\caption{\label{Fig:0303randomToepl}
Performances for $2^N\times 2^N$ random Toeplitz matrices. 
The block TT-ranks of the left and right singular vectors are 
bounded by $R_{max}$. 
(a) Computational cost and 
(b) relative residual for fixed $R_{max}=15$ and various $10\leq N\leq 30$. 
(c) Computational cost and 
(b) relative residual for various $10\leq R_{max}\leq 30$ and fixed $N=30$. 
}
\end{figure}

\section{Conclusion and Discussions}

In this paper, we proposed new SVD algorithms 
for very large-scale structured matrices
based on TT decompositions. 
Unlike previous researches focusing only on eigenvalue decomposition (EVD) of
symmetric positive semidefinite matrices
\cite{Dol2013b,
Holtz2012,HuckleWald2012,Kho2010,
KresSteinUsh2013,Leb2011,Mach2011,
USch2011}, 
the proposed algorithms do not assume symmetricity 
of the data matrix $\BF{A}$. 
We investigated the computational complexity
of the proposed algorithms rigorously, and 
provided optimized ways of tensor contractions 
for the fast computation of the singular values. 
We conducted extensive simulations
to demonstrate the effectiveness of the proposed SVD algorithms 
compared with the other TT-based algorithms which are 
based on the EVD of symmetric positive semidefinite matrices.

Once a very large-scale matrix is represented in 
matrix TT format, the proposed ALS-SVD and MALS-SVD 
algorithms can compute the SVD in logarithmic time complexity 
with respect to the matrix size 
under the assumption that the TT-ranks are bounded. 
Unlike the EVD-based methods, 
the proposed algorithms avoid 
truncation of the matrix TT-ranks of the product $\BF{A}^\RM{T}\BF{A}$
but directly optimize the maximization problem \eqref{eqn:maximize}. 
In the simulated experiments, we demonstrated that 
the computational costs of the EVD-based algorithms 
are highly affected by the matrix TT-ranks,  
and the proposed methods are highly competitive 
compared with the EVD-based algorithms. 
Moreover, we showed that the 
proposed methods can compute the SVD 
of $2^{50}\times 2^{50}$ matrices accurately even in a 
few seconds on desktop computers.

The proposed SVD methods can compute 
a few dominant singular values and corresponding singular vectors
of TT-structured matrices. 
The structured matrices used in the simulations are 
random matrices with prescribed singular values, 
Hilbert matrix, random tridiagonal matrix, and 
random Toeplitz matrix. 
The singular vectors are represented as block TT formats, 
and the TT-ranks are adaptively determined during 
iteration process. 
Moreover, we also presented the case of random Toeplitz matrices, 
where the block TT-ranks of the singular vectors are not bounded
as the matrix size increases. In this case, the proposed methods 
computed approximate solutions based on fixed TT-ranks with 
reasonable approximation errors. 
Since the TT-ranks are fixed, the computational cost will be much 
reduced if the $\delta$-truncated SVD step is replaced with 
the QR decomposition.

In the simulated experiments, we observed that the 
truncation parameter $\delta$ for the $\delta$-trucated
SVD highly affects the convergence. 
If $\delta$ is too large, then the algorithm 
falls into local minimum and its accuracy does not 
improve any more. If $\delta$ is too small, then 
the TT-ranks grow fastly and the computational 
cost increases. We initialized the $\delta$ value by 
$\delta_0=\epsilon/\sqrt{N-1}$ as proposed by 
Oseledets \cite{Ose2011}, in which case the proposed 
algorithms usually achieved the desired accuracy. 
If the TT-based ALS and MALS algorithms 
fall into local minimum, we restarted the algorithms 
with new initial block TT tensors. 
If a proper $\delta$ value is selected and the 
number $K$ of singular values is large enough, 
then the algorithms converge usually in at most 3 full sweeps. 
Moreover, the MALS algorithm shows faster convergence
than the ALS algorithm because the 
TT-ranks can be increased more fastly at each iteration.

The performance of the TT-based algorithms are highly 
dependent on the choice of the optimization algorithms 
for solving the reduced local problems. 
In the simulations we applied the MATLAB function EIGS
to the matrix $\begin{bmatrix}\BF{0}&\BF{A}\\
\BF{A}^\RM{T}&\BF{0}\end{bmatrix}$
in order to obtain accurate singular values.

In order to convert a very large-scale matrix into matrix TT format, 
it is suggested to employ cross approximation methods
\cite{Ballani2013,Ose2010}. 
We also applied the MATLAB function FUNCRS2 in 
TT-Toolbox \cite{Ose2011b} for Hilbert matrices in the 
numerical simulations. 

The proposed algorithms rely on the optimization 
with the trace function described in the maximization problem 
\eqref{eqn:maximize}, so they cannot be applied 
for computing $K$ smallest singular values directly. 
In Appendix \ref{sec:app1}, we explained how the 
$K$ smallest singular values and corresponding singular 
vectors can be computed by using the EVD-based algorithms. 
In the future work, we will develop a more efficient method 
for computing a few smallest singular values and 
corresponding singular vectors based on TT decompositions.

\appendix

\section{Optimization Problems for Extremal Singular Values}
\label{sec:app1}



The SVD of a matrix $\BF{A}\in\BB{R}^{P\times Q}$ is closely 
related to the eigenvalue decomposition (EVD)
of the following $(P+Q)\times(P+Q)$ matrix
	\begin{equation}
	\BF{B} = \begin{bmatrix} \BF{0} & \BF{A} \\ 
	\BF{A}^\RM{T} & \BF{0} \end{bmatrix}. 
	\end{equation}
In this section, we show the relationship between
the EVD optimization problems of $\BF{B}$ and the 
SVD optimization problems of $\BF{A}$. 

\subsection{Eigenvalues of $\BF{B}$}
We assume that $P\geq Q$. The SVD of the matrix 
$\BF{A}\in\BB{R}^{P\times Q}$ can be expressed as
	\begin{equation}
	\BF{A} 
	= \begin{bmatrix} \BF{U}_0 & \BF{U}_0^{\bot} \end{bmatrix}
	\begin{bmatrix}\BF{\Sigma}_0\\\BF{0}\end{bmatrix}
	\BF{V}_0^\RM{T}, 
	\end{equation}
where $\BF{U}_0\in\BB{R}^{P\times Q}$, 
$\BF{U}_0^\bot \in\BB{R}^{P\times (P-Q)}$, 
and $\BF{V}_0\in\BB{R}^{Q\times Q}$ 
are the matrices of singular vectors and 
$\BF{\Sigma}_0\in\BB{R}^{Q\times Q}$ is the diagonal matrix 
with nonnegative diagonal entries
$\sigma_1\geq \sigma_2\geq \cdots \geq \sigma_Q$. 

\begin{lem}
The EVD of the matrix $\BF{B}$ can be written by
	\begin{equation}
	\BF{B} = \BF{W}_0\BF{\Lambda}_0\BF{W}_0^\RM{T}, 
	\end{equation}
where 
	\begin{equation}
	\BF{W}_0 = \frac{1}{\sqrt{2}}
	\begin{bmatrix} \BF{U}_0& \BF{U}_0 & \sqrt{2}\BF{U}_0^\bot \\ 
	\BF{V}_0 & -\BF{V}_0 & \BF{0}\end{bmatrix}
	\in\BB{R}^{(P+Q)\times(P+Q)}, 
	\end{equation}
	\begin{equation}
	\BF{\Lambda}_0 = 
	\begin{bmatrix}
	\BF{\Sigma}_0 & & \\
	& -\BF{\Sigma}_0 & \\
	& & \BF{0}
	\end{bmatrix}
	\in\BB{R}^{(P+Q)\times(P+Q)}. 
	\end{equation}
\end{lem}
\begin{proof}
We can compute that $\BF{W}_0^\RM{T}\BF{W}_0=\BF{I}_{P+Q}$. 
We can show that $\BF{BW}_0 = \BF{W}_0\BF{\Lambda}_0$. 
\end{proof}

We can conclude that the eigenvalues of the matrix $\BF{B}$ consist of 
$\pm \sigma_1,\pm \sigma_2,\ldots,\pm \sigma_Q$, and an extra zero
of multiplicity $P-Q$. 


\subsection{Maximal Singular Values}

The $K$ largest eigenvalues of the matrix $\BF{B}$ can be computed 
by solving the trace maximization problem \cite[Theorem 1]{Fan1949}
	\begin{equation}\label{eqn:maximize_W}
	\begin{split}
	\maximize_{\BF{W}}
	& \qquad
	\text{trace}\left( \BF{W}^\RM{T}\BF{B}\BF{W} \right)\\
	\text{subject to}
	& \qquad
	\BF{W}^\RM{T}\BF{W}=\BF{I}_K. 
	\end{split}
	\end{equation}
Instead of building the matrix $\BF{B}$ explicitly, 
we solve the equivalent maximization problem described as follows. 
\begin{prop}
For $K\leq Q$, the maximization problem (\ref{eqn:maximize_W}) is equivalent to 
	\begin{equation}
	\begin{split}
	\maximize_{\BF{U}, \BF{V}}
	& \qquad
	\text{trace}\left( \BF{U}^\RM{T}\BF{A}\BF{V} \right)\\
	\text{subject to}
	& \qquad 
	\BF{U}^\RM{T}\BF{U}=
		\BF{V}^\RM{T}\BF{V}=\BF{I}_K. 
	\end{split}
	\end{equation}
\end{prop}
\begin{proof}
Let 
	\begin{equation}
	\BF{W}=\frac{1}{\sqrt{2}}\begin{bmatrix}\BF{U}\\ \BF{V}\end{bmatrix}
	\in\BB{R}^{(P+Q)\times K}, 
	\end{equation}
then 
	\begin{equation}
	\text{trace}\left( \BF{W}^\RM{T}\BF{B}\BF{W} \right)
	=\text{trace}\left(\BF{U}^\RM{T}\BF{AV} \right). 
	\end{equation}
First, we can show that 
	\begin{equation}
	\max_{\BF{W}^\RM{T}\BF{W} = \BF{I}_K} 
	\text{trace}\left( \BF{W}^\RM{T}\BF{B}\BF{W} \right)
	\geq 
	\max_{\BF{U}^\RM{T}\BF{U} 
	= \BF{V}^\RM{T}\BF{V} = \BF{I}_K} 
	\text{trace}\left(\BF{U}^\RM{T}\BF{AV} \right). 
	\end{equation}
Next, we can show that 
the maximum value $\sigma_1+\sigma_2+\cdots+\sigma_K$
of $\text{trace}\left( \BF{W}^\RM{T}\BF{B}\BF{W} \right)$
is obtained by $\text{trace}\left(\BF{U}^\RM{T}\BF{AV} \right)$ when 
$\BF{U}$ and $\BF{V}$ are equal to the first $K$ singular vectors 
of $\BF{U}_0$ and $\BF{V}_0$. 
\end{proof}

\subsection{Minimal Singular Values}

Suppose that $P=Q$. The $K$ minimal singular values of $\BF{A}$
can be obtained by computing $2K$ eigenvalues of $\BF{B}$
with the smallest magnitudes, that is, $\pm \sigma_{Q-K+1}, 
\pm \sigma_{Q-K+2}, \ldots, \pm \sigma_Q$. 
Computing the $2K$ eigenvalues of 
$\BF{B}$ with the smallest magnitudes can be formulated 
by the following trace minimization problem
	\begin{equation}\label{eqn:minimize_W}
	\begin{split}
	\minimize_{\BF{W}}
	& \qquad
	\text{trace}\left( \BF{W}^\RM{T}\BF{B}^2\BF{W} \right)\\
	\text{subject to}
	& \qquad
	\BF{W}^\RM{T}\BF{W}=\BF{I}_{2K}. 
	\end{split}
	\end{equation}
We can translate the above minimization problem into 
the equivalent minimization problem without building 
the matrix $\BF{B}$ explicitly as follows. 

\begin{prop}
The minimization problem (\ref{eqn:minimize_W}) 
is equivalent to the following two minimization problems
if a permutation ambiguity is allowed:
	\begin{equation}\label{eqn:minimize_1}
	\begin{split}
	\minimize_{\BF{V}}
	& \qquad
	\text{trace}\left( \BF{V}^\RM{T}\BF{A}^\RM{T}\BF{A}\BF{V} \right)\\
	\text{subject to}
	& \qquad 
	\BF{V}^\RM{T}\BF{V}=\BF{I}_K
	\end{split}
	\end{equation}
and 
	\begin{equation}\label{eqn:minimize_2}
	\begin{split}
	\minimize_{\BF{U}}
	& \qquad
	\text{trace}\left( \BF{U}^\RM{T}\BF{AA}^\RM{T}\BF{U} \right)\\
	\text{subject to}
	& \qquad 
	\BF{U}^\RM{T}\BF{U}=\BF{I}_K. 
	\end{split}
	\end{equation}
\end{prop}
That is, the $K$ minimal singular values of $\BF{A}$ can be computed 
by applying the eigenvalue decomposition 
for $\BF{A}^\RM{T}\BF{A}$ and $\BF{A}\BF{A}^\RM{T}$. 

\begin{proof}
Let
	\begin{equation}
	\BF{W}=
	\frac{1}{\sqrt{2}}\begin{bmatrix}\BF{U}&\BF{U}\\ 
	\BF{V}&-\BF{V}\end{bmatrix}
	\in\BB{R}^{(P+Q)\times 2K}, 
	\end{equation}
then we have
	\begin{equation}
	\text{trace}\left(\BF{W}^\RM{T}\BF{B}^2\BF{W}\right)
	= \text{trace}\left(\BF{U}^\RM{T}\BF{AA}^\RM{T}\BF{U}\right)
	+ \text{trace}\left(\BF{V}^\RM{T}\BF{A}^\RM{T}\BF{A}\BF{V}\right). 
	\end{equation}
By algebraic manipulation, we can derive that 
the constraint $\BF{W}^\RM{T}\BF{W}=\BF{I}_{2K}$ is equivalent 
to $\BF{U}^\RM{T}\BF{U}=\BF{V}^\RM{T}\BF{V}=\BF{I}_K$. 
\end{proof}

\section{Explicit Tensor Train Representation of Toeplitz Matrix 
and Hankel Matrix}
\label{sec:app2}

Explicit TT representations of Toeplitz matrices are
presented in \cite{Kaz2013}. In this section, we 
summarize some of the simplest results of \cite{Kaz2013}, 
and extend them to Hankel matrix and rectangular 
submatrices.  

First, we introduct the Kronecker product representation
for matrix TT format. 
Given a matrix $\BF{A}\in\BB{R}^{I_1I_2\cdots I_N\times 
J_1J_2\cdots J_N}$ in matrix TT format \eqref{eqn:matrixTTc}, 
each entry of the tensor $\ten{A}\in\BB{R}^{I_1\times J_1\times\cdots\times 
I_N\times J_N}$ is representedy by the sum of scalar products
	\begin{equation}
	a_{i_1,j_1,i_2,j_2,\ldots,i_N,j_N}
	= \sum_{r^A_1=1}^{R^A_1} \sum_{r^A_2=1}^{R^A_2}
	\cdots \sum_{r^A_{N-1}=1}^{R^A_{N-1}}
	a^{(1)}_{1,i_1,j_1,r^A_1} a^{(2)}_{r^A_1,i_2,j_2,r^A_2}
	\cdots a^{(N)}_{r^A_{N-1},i_N,j_N,1},
	\end{equation}
which is equal to each entry of the matrix $\BF{A}$, i.e., 
$a_{(i_1,i_2,\ldots,i_N),(j_1,j_2,\ldots,j_N)}$, where $(i_1,i_2,\ldots,i_N)$
is the multi-index introduced in \eqref{eqn:multi_index_paran}. 
From this expression, we can derive that the matrix $\BF{A}$
can be represented as sums of Kronecker products of matrices 
	\begin{equation} \label{eqn:matrix_TT_kron}
	\BF{A} = \sum_{r^A_1=1}^{R^A_1} 
	\sum_{r^A_2=1}^{R^A_2}
	\cdots \sum_{r^A_{N-1}=1}^{R^A_{N-1}}
	\BF{A}^{(N)}_{1,r^A_{N-1}}\otimes 
	\BF{A}^{(N-1)}_{r^A_{N-1},r^A_{N-2}}\otimes 
	\cdots \otimes \BF{A}^{(1)}_{r^A_1,1}, 
	\end{equation}
where the matrices $\BF{A}^{(n)}_{r^A_n,r^A_{n-1}}\in\BB{R}^{I_n\times J_n}$ 
are defined by 
	\begin{equation}
	\BF{A}^{(n)}_{r^A_n,r^A_{n-1}} = 
	\left(a^{(n)}_{r^A_{n-1},i_n,j_n,r^A_n}\right)_{i_n,j_n}
	= \ten{A}^{(n)}(r^A_{n-1},:,:,r^A_n). 
	\end{equation}
Note that the positions of the indices $r^A_{n-1}$ and $r^A_n$ 
have been switched for notational convenience.

Next, we present the explicit TT representations for 
Toeplitz matrices and Hankel matrices. 
The $2^N\times 2^N$ upper triangular Toeplitz matrix 
generated by $\left[s_1,s_2,\ldots,s_{2^N-1}\right]^\RM{T}$
is written by 
	\begin{equation}\label{eq:toep}
	\BF{T} = \begin{bmatrix}
	0 & s_1 & s_2 & \cdots & s_{2^N-2} & s_{2^N-1}  \\
	  & 0 & s_1 & \cdots & s_{2^N-3} & s_{2^N-2} \\
	  &   & & & \vdots & \vdots \\
	  & 	& &  &  0 & s_1 \\
	  & 	& & &  & 0 
	\end{bmatrix}. 
	\end{equation}
Similarly, the $2^N\times 2^N$ upper anti-triangular Hankel matrix 
generated by $\left[s_1,s_2,\ldots,s_{2^N-1}\right]^\RM{T}$ is written by 
	\begin{equation}\label{eq:hilb}
	\BF{H} = \begin{bmatrix}
	s_{2^N-1} & s_{2^N-2} & \cdots & s_2 & s_1 & 0  \\
	s_{2^N-2} & s_{2^N-3} & \cdots & s_1 & 0 &  \\
	 \vdots & \vdots & & & & \\
	 s_1 & 0	& &  &  & \\
	 0 &	& & &  & 
	\end{bmatrix}. 
	\end{equation}	
The matrix TT representation for 
the Toeplitz matrix is presented in the following theorem. 

\begin{thm}[An explicit matrix TT representation of 
Toeplitz matrix, \cite{Kaz2013}] 

Let 
	\begin{equation}
	\BF{I} = \begin{bmatrix}
	1&0\\0&1 
	\end{bmatrix}, \quad
	\BF{J} = \begin{bmatrix}
	0&1\\0&0 
	\end{bmatrix}, \quad
	\BF{K} = \begin{bmatrix}
	0&0\\1&0 
	\end{bmatrix}
	\end{equation}
be $2\times 2$ matrices and let
	\begin{equation}\label{block_matrixL01}
	\widetilde{\BF{L}}^{(N)}_1 = 
	\begin{bmatrix} \BF{I} & \BF{J} \end{bmatrix}, 
	\quad
	\widetilde{\BF{L}}^{(N-1)}_1 = 
	\cdots = \widetilde{\BF{L}}^{(2)}_1 =
	\begin{bmatrix} \BF{I} & \BF{J}\\ \BF{0}&\BF{K} \end{bmatrix}, 
	\quad
	\widetilde{\BF{L}}^{(1)}_1 = 
	\begin{bmatrix} \BF{J}\\ \BF{K} \end{bmatrix},
	\end{equation}
	\begin{equation}
	\widetilde{\BF{L}}^{(N)}_2 = 
	\begin{bmatrix} \BF{J} & \BF{0} \end{bmatrix}, 
	\quad
	\widetilde{\BF{L}}^{(N-1)}_2 = 
	\cdots = \widetilde{\BF{L}}^{(2)}_2 =
	\begin{bmatrix} \BF{J} & \BF{0}\\ \BF{K}&\BF{I} \end{bmatrix}, 
	\quad
	\widetilde{\BF{L}}^{(1)}_2 = 
	\begin{bmatrix} \BF{0}\\ \BF{I} \end{bmatrix}
	\end{equation}
be block matrices. For each of the 
block matrices $\widetilde{\BF{L}}^{(n)}_{k_n},$ 
$k_n=1,2$, we denote the $(q_n,q_{n-1})$th block of 
$\widetilde{\BF{L}}^{(n)}_{k_n}$
by $\BF{L}^{(n)}_{q_n,k_n,q_{n-1}}\in\BB{R}^{2\times 2}$, 
that is, 
	\begin{equation}
	\widetilde{\BF{L}}^{(n)}_{k_n} = 
	\left[ \BF{L}^{(n)}_{q_n,k_n,q_{n-1}}  
	\right]_{q_n,q_{n-1}}, 
	\quad k_n=1,2.
	\end{equation} 
Suppose that the vector 
$\BF{s}=\left[s_1,s_2,\ldots,s_{2^N}\right]^\RM{T}$
of length $2^N$ is represented in TT format as 
	\begin{equation}
	s_{k_1,k_2,\ldots,k_N} =
	\sum_{r_1=1}^{R_1} \sum_{r_2=1}^{R_2}\cdots 
	\sum_{r_{N-1}=1}^{R_{N-1}}
	s^{(1)}_{1,k_1,r_1}s^{(2)}_{r_1,k_2,r_2}\cdots 
	s^{(N)}_{r_{N-1},k_N,1}. 
	\end{equation}
Then, the upper triangular Toeplitz matrix \eqref{eq:toep}
is expressed in matrix TT format as 
	\begin{equation}
	\BF{T} = 
	\sum_{t_1=1}^{2R_1} \sum_{t_2=1}^{2R_2}\cdots 
	\sum_{t_{N-1}=1}^{2R_{N-1}}
	\BF{T}^{(N)}_{1,t_{N-1}} \otimes 
	\BF{T}^{(N-1)}_{t_{N-1}, t_{N-2}}\otimes
	\cdots \otimes \BF{T}^{(1)}_{t_1,1}, 
	\end{equation}
where $\BF{T}^{(n)}_{t_n,t_{n-1}}\in\BB{R}^{2\times 2}$ 
are defined by 
	\begin{equation}
	\BF{T}^{(n)}_{(r_n,q_n), (r_{n-1},q_{n-1})} = 
	\sum_{k_n=1}^{2} s^{(n)}_{r_{n-1},k_n,r_n}
	\BF{L}^{(n)}_{q_{n},k_n,q_{n-1}}
	\end{equation}
with $t_n=(r_n,q_n)$ for $n=1,2,\ldots,N.$
\end{thm}

The matrix TT representation for Toeplitz matrices can be 
extended to Hankel matrices as follows. 

\begin{cor}[An explicit matrix TT representation of 
Hankel matrix] 

Let 
	\begin{equation}\label{basic_matrix02}
	\BF{P} = \begin{bmatrix}
	0&1\\1&0 
	\end{bmatrix}, \quad
	\BF{Q} = \begin{bmatrix}
	1&0\\0&0 
	\end{bmatrix}, \quad
	\BF{R} = \begin{bmatrix}
	0&0\\0&1 
	\end{bmatrix}
	\end{equation}
be $2\times 2$ matrices and let
	\begin{equation}\label{block_matrixM01}
	\widetilde{\BF{M}}^{(N)}_1 = 
	\begin{bmatrix} \BF{P} & \BF{Q} \end{bmatrix}, 
	\quad
	\widetilde{\BF{M}}^{(N-1)}_1 = 
	\cdots = \widetilde{\BF{M}}^{(2)}_1 =
	\begin{bmatrix} \BF{P} & \BF{Q}\\ \BF{0}&\BF{R} \end{bmatrix}, 
	\quad
	\widetilde{\BF{M}}^{(1)}_1 = 
	\begin{bmatrix} \BF{Q}\\ \BF{R} \end{bmatrix},
	\end{equation}
	\begin{equation}
	\widetilde{\BF{M}}^{(N)}_2 = 
	\begin{bmatrix} \BF{Q} & \BF{0} \end{bmatrix}, 
	\quad
	\widetilde{\BF{M}}^{(N-1)}_2 = 
	\cdots = \widetilde{\BF{M}}^{(2)}_2 =
	\begin{bmatrix} \BF{Q} & \BF{0}\\ \BF{R}&\BF{P} \end{bmatrix}, 
	\quad
	\widetilde{\BF{M}}^{(1)}_2 = 
	\begin{bmatrix} \BF{0}\\ \BF{P} \end{bmatrix}
	\end{equation}
be block matrices. 
For each of the block matrices 
$\widetilde{\BF{M}}^{(n)}_{k_n},k_n=1,2,$
we denote the $(q_n,q_{n-1})$th block of 
$\widetilde{\BF{M}}^{(n)}_{k_n}$
by $\BF{M}^{(n)}_{q_n,k_n,q_{n-1}}\in\BB{R}^{2\times 2}$, that is, 
	\begin{equation}
	\widetilde{\BF{M}}^{(n)}_{k_n} = 
	\left[ \BF{M}^{(n)}_{q_n,k_n,q_{n-1}}  
	\right]_{q_n,q_{n-1}}, 
	\quad k_n=1,2.
	\end{equation}
Suppose that $\BF{s}=\left[s_1,s_2,\ldots,s_{2^N}\right]^\RM{T}$
is represented in TT format as 
	\begin{equation}
	s_{k_1,k_2,\ldots,k_N} =
	\sum_{r_1=1}^{R_1} \sum_{r_2=1}^{R_2}\cdots 
	\sum_{r_{N-1}=1}^{R_{N-1}}
	s^{(1)}_{1,k_1,r_1}s^{(2)}_{r_1,k_2,r_2}\cdots 
	s^{(N)}_{r_{N-1},k_N,1}. 
	\end{equation}
Then, the upper anti-triangular Hankel matrix \eqref{eq:hilb}
is expressed in matrix TT format as 
	\begin{equation}\label{eq:matrixTThankel}
	\BF{H} = 
	\sum_{t_1=1}^{2R_1} \sum_{t_2=1}^{2R_2}\cdots 
	\sum_{t_{N-1}=1}^{2R_{N-1}}
	\BF{H}^{(N)}_{1,t_{N-1}}\otimes 
	\BF{H}^{(N-1)}_{t_{N-1},t_{N-2}}\otimes
	\cdots \otimes \BF{H}^{(1)}_{t_1,1}, 
	\end{equation}
where $\BF{H}^{(n)}_{t_n,t_{n-1}}\in\BB{R}^{2\times 2}$ 
are defined by 
	\begin{equation}
	\BF{H}^{(n)}_{(r_n,q_n),(r_{n-1}q_{n-1})} = 
	\sum_{k_n=1}^{2} s^{(n)}_{r_{n-1},k_n,r_n}
	\BF{M}^{(n)}_{q_n,k_n,q_{n-1}}
	\end{equation}
with $t_n=(r_n,q_n)$ for $n=1,2,\ldots,N.$
\end{cor}

The matrix TT representation of a submatrix of the 
Hankel matrix $\BF{H}$ can be derived from the representation 
of $\BF{H}$ \eqref{eq:matrixTThankel}. 

\begin{cor}
The $2^N\times 2^{N-1}$ submatrix $\BF{H}(:,1:2^{N-1})$ of 
the Hankel matrix $\BF{H}$ can be written by 
	\begin{equation}
	\BF{H}(:,1:2^{N-1}) = 
	\sum_{t_1=1}^{2R_1} \sum_{t_2=1}^{2R_2}\cdots 
	\sum_{t_{N-2}=1}^{2R_{N-2}}
	\left( \sum_{t_{N-1}=1}^{2R_{N-1}}
		\BF{H}^{(N)}_{1,t_{N-1}}(:,1)\otimes 
		\BF{H}^{(N-1)}_{t_{N-1},t_{N-2}}  \right) 
	\otimes \cdots \otimes 
	  \BF{H}^{(1)}_{t_1,1}, 
	\end{equation}
where $\BF{H}^{(N)}_{1,t_{N-1}}(:,1)\in\BB{R}^{2\times 1}$ is the first column vector 
of $ \BF{H}^{(N)}_{1,t_{N-1}}$. 
\end{cor}

In the same way, we can derive the matrix TT representation of a
top-left corner submatrix of the Hankel matrix.

\section{Explicit Tensor Train Representation of Tridiagonal Matrix}
\label{sec:app3}

A shift matrix $\BF{F}\in\BB{R}^{P\times P}$ is a banded 
binary matrix whose nonzero entries are on the first diagonal above 
the main diagonal. The $(i,j)$th entry of $\BF{F}$ is 
	$f_{ij} = 1$ if $i+1=j,$ and $f_{ij}=0$ otherwise. 
The following lemma describes a TT representation 
for the shift matrix. 

\begin{lem}[An explicit matrix TT representation of 
the shift matrix, \cite{Kaz2013}]
Let $\widetilde{\BF{L}}^{(n)}_1$
and $\widetilde{\BF{M}}^{(n)}_1$, $n=1,2,\ldots,N$, are the 
block matrices defined by \eqref{block_matrixL01} and \eqref{block_matrixM01}. 
The shift matrix $\BF{F}\in\BB{R}^{2^N\times 2^N}$ is 
represented in TT format by 
	\begin{equation}
	\BF{F} = \sum_{q_1=1}^{2} \sum_{q_2=1}^{2}\cdots 
	\sum_{q_{N-1}=1}^{2}
	\BF{L}^{(N)}_{1,1,q_{N-1}}
	\otimes \BF{L}^{(N-1)}_{q_{N-1},1,q_{N-2}} \otimes
	\cdots \otimes \BF{L}^{(1)}_{q_1,1,1}. 
	\end{equation}
The transpose of the shift matrix $\BF{F}$ is 
represented in TT format by 
	\begin{equation}
	\BF{F}^\RM{T} = \sum_{q_1=1}^{2} \sum_{q_2=1}^{2}\cdots 
	\sum_{q_{N-1}=1}^{2}
	(\BF{L}^{(N)}_{1,1,q_{N-1}})^\RM{T}
	\otimes (\BF{L}^{(N-1)}_{q_{N-1},1,q_{N-2}})^\RM{T} \otimes
	\cdots \otimes (\BF{L}^{(1)}_{q_1,1,1})^\RM{T}. 
	\end{equation}
\end{lem}

A tridiagonal matrix $\BF{A}\in\BB{R}^{2^N\times 2^N}$ 
generated by three vectors $\BF{a,b,c}\in\BB{R}^{2^N}$
is written by 
	\begin{equation}
	\BF{A} = \begin{bmatrix}
		b_1 & c_2 & 0 & & \\
	  a_1 & b_2 & c_3 & 0 &  \\
	   & \ddots & \ddots  & \ddots & \\
	  & 0 & a_{2^{N}-2} & b_{2^{N}-1} & c_{2^N} \\
	  & & 0 & a_{2^{N}-1} & b_{2^N} 
	\end{bmatrix}. 
	\end{equation}
Suppose that the vectors $\BF{a,b,c}$ are given in 
vector TT format. Then, by using the basic operations \cite{Ose2011}, 
we can compute the tridiagonal matrix $\BF{A}$
by 	
	\begin{equation}
	\BF{A} = \BF{F}^\RM{T} \text{diag}(\BF{a})
	+ 
	\text{diag}(\BF{b})
	+
	\BF{F}\text{diag}(\BF{c}). 
	\end{equation}

\end{document}